\documentclass[a4paper,reqno]{amsart}

\usepackage{fullpage}
\usepackage{comment}

\usepackage{cite}

\usepackage{amsmath,amssymb,amsthm,mathtools}

\usepackage{tikz}
\usetikzlibrary{automata,positioning,arrows,calc}

\usepackage{enumitem}

\usepackage{thmtools,thm-restate}
\usepackage[hidelinks]{hyperref}
\definecolor{darkblue}{rgb}{0.0, 0.0, 0.55}
\hypersetup{%
	colorlinks=true,
	linkcolor = darkblue,
	anchorcolor = darkblue,
	citecolor = darkblue,
	filecolor = darkblue,
	urlcolor = darkblue,
	linktoc=all
}

\usepackage[capitalise,noabbrev,sort]{cleveref}

\newcommand{\MCF}{multiple context-free}
\newcommand{\RMCFG}{R-MCFG}

\newtheorem{theorem}{Theorem}[section]
\newtheorem{proposition}[theorem]{Proposition}
\newtheorem{lemma}[theorem]{Lemma}
\newtheorem{corollary}[theorem]{Corollary}

\newtheorem{theoremx}{Theorem}

\crefalias{theoremx}{theorem}

\crefalias{propositionx}{proposition}

\newtheorem{conjecturex}[theoremx]{Conjecture}
\crefalias{conjecturex}{conjecture}

\theoremstyle{definition}
\newtheorem{example}[theorem]{Example}
\newtheorem{definition}[theorem]{Definition}
\newtheorem{remark}[theorem]{Remark}

\renewcommand{\leq}{\leqslant}
\renewcommand{\geq}{\geqslant}
\DeclareMathOperator{\WP}{WP}

\newcommand\leftSide\triangleleft
\newcommand\rightSide\triangleright
\newcommand\Affix{\mathrm{Affix}}
\newcommand\Strip{\mathrm{Strip}}
\newcommand\Rem{\mathrm{Rem}}
\newcommand\MonoidHom{\Delta}
\newcommand\Segment{\mathrm{seg}}
\newcommand\Reverse{\mathrm{Mirror}}

\title{On groups with EDT0L word problem}

\author[A.~Bishop]{Alex Bishop}
\address{
Section de mathématiques,
Université de Genève,
rue du Conseil-Général 7-9,
1205 Genève, Switzerland}
\email{alexbishop1234@gmail.com}

\author[M.~Elder]{Murray Elder}
\address{School of Mathematical and Physical Sciences, University of Technology Sydney, Broadway NSW 2007, Australia}
\email{murray.elder@uts.edu.au}

\author[A.~Evetts]{Alex Evetts}
\address{Department of Mathematics, The University of Manchester, Manchester M13 9PL, UK}
\email{alex.evetts@manchester.ac.uk}

\author[P.~Gallot]{Paul Gallot}
\address{
Database Group, Universität Bremen - FB 03
Postfach 33 04 40, 
28334 Bremen,
Germany}
\email{pgallot@uni-bremen.de}

\author[A.~Levine]{Alex Levine}
\address{School of Engineering, Mathematics and Physics, University of East Anglia, Norwich NR4 7TJ, UK}
\email{a.levine@uea.ac.uk}

\date{}

\begin{document}

\begin{abstract}
	We prove that the word problem for the infinite cyclic group is not EDT0L, and obtain as a corollary that a finitely generated group with EDT0L word problem must be torsion.
  In addition, we show that the property of having an EDT0L word problem is invariant under change of generating set, and passing to finitely generated subgroups. This represents significant progress towards the conjecture that all groups with EDT0L word problem are finite (i.e.~precisely the groups with regular word problem).
\end{abstract}
\maketitle

\section{Introduction}\label{sec:introduction}

An interesting problem is to classify finitely generated groups in terms of  the 
language complexity of their \emph{word problem}, the set of all words over generators and their inverses which spell the identity of the group.
The formal study of this problem began when Anisimov~\cite{Anisimov1971} first observed that the word problem of a group is regular if and only if the group is finite.
An influential result of Muller and Schupp~\cite{MullerSchupp} later showed that the word problem of a group is context-free if and only if the group is virtually free. Moreover, the following equivalences have been shown:  the word problem is one-counter if and only if the group is virtually cyclic~\cite{Herbst}; the intersection of finitely many one-counter languages if and only if the group is virtually abelian~\cite{HoltOwensThomas};  blind multicounter if and only if the group is virtually abelian \cite{EKambOst,Yuyama}; a language accepted by a Petri net if and only if the group is virtually abelian \cite{PetriNet}; an  NTS language if and only if the group is virtually free \cite{Sen}.

The family of EDT0L (Extended alphabet, Deterministic, Table, 0-interaction, Lindenmayer) languages have recently received a great amount of interest within geometric group theory and related areas~\cite{BartholdiPres,BE-coET0L,BroughPerms,Duncan,CDEfree,CiobanuElder2021,CiobanuElderFerov,DiekertRAAGsArxiv,JAIN201968,LeHe2018,Levine2023}.
In this paper, we also consider the family of \emph{finite-index} EDT0L languages.
\Cref{fig:EDTdiagram} shows the relative expressive power of these families of languages, where each forwards arrow represents a strict inclusion.
In particular, the family of EDT0L languages contains the family of finite-index EDT0L languages as a strict subfamily; and the family of context-free languages is incomparable (in terms of set inclusion) with the families of EDT0L and finite-index EDT0L languages.

\begin{figure}[ht]
	\centering
  \begin{tikzpicture}[>=stealth]
    \node (REG) at (0,0) {regular};
    \node (fiEDT0L)[right=2.5em of REG] {\begin{tabular}{c}
         finite-index \\EDT0L
    \end{tabular}};
    \node (EDT0L)[right=2.5em of fiEDT0L] {EDT0L};
    \node (ET0L)[right=2.5em of EDT0L] {ET0L};
    \node (IDX)[right=2.5em of ET0L] {indexed};
    \node (CS)[right=2.5em of IDX] {context-sensitive};

    \coordinate (midpoint) at ($(REG.east)!0.5!(ET0L.west)$); 

    \node (CF)[above=3.5em of midpoint] {context-free};

    \draw[->](REG)--(fiEDT0L);
    \draw[->](fiEDT0L)--(EDT0L);
    \draw[->](EDT0L)--(ET0L);
    \draw[->](ET0L)--(IDX);
    \draw[->](IDX)--(CS);

    \draw[->] (REG.north east) -- (CF.south west);
    \draw[->] (CF.south east) -- (ET0L.north west);
	\end{tikzpicture}

  \caption{Inclusion diagram of formal languages (see Theorem~15 in \cite{Rozenberg1978} and Figure~8 in~\cite{engelfriet1980stack}).}\label{fig:EDTdiagram}
\end{figure}

We begin by showing that having an EDT0L word problem is invariant under change of generating set, which is not immediate since EDT0L languages are not closed under inverse monoid homomorphism \cite{ER1976}.
From the following proposition, it makes sense to speak of a word problem being EDT0L without the need to specify a generating set.
Let $\mathrm{WP}(G,X)$ and $\mathrm{coWP}(G,X)$ denote the word and coword problem, respectively, for a group $G$ with respect to the generating set $X$.

\begin{restatable}{propositionx}{CorInvariantsDFSM}\label{lem:taking-a-submonoid-of-EDT0L}
	Let $G$ be a group with finite monoid generating sets $X$ and $Y$.
	If $\mathrm{WP}(G,X)$ (resp.\@ $\mathrm{coWP}(G,X)$) is EDT0L, then $\mathrm{WP}(G,Y)$ (resp.\@ $\mathrm{coWP}(G,Y)$) is EDT0L of finite index.
	These results also hold if $Y$ instead generates a submonoid of $G$.
    In particular, this implies that, if a word problem is EDT0L, then it is EDT0L of finite index; and that having a word problem that is EDT0L of finite index is independent of choice of generating set.
\end{restatable}

The following conjecture was posed by Ciobanu, Ferov and the second author.

\begin{conjecturex}[Conjecture 8.2 in~\cite{CiobanuElderFerov}]
	If $G$ has an EDT0L word problem, then $G$ is finite.\label{conj:main}
\end{conjecturex}

Towards this conjecture, we have the following results.
\begin{restatable}{theoremx}{MainTheorem}\label{thm:main}
	The word problem for the infinite cyclic group $\mathbb{Z}$ is not EDT0L.
\end{restatable}

From the above theorem and \cref{lem:taking-a-submonoid-of-EDT0L}, we then obtain the following theorem.

\begin{restatable}{theoremx}{MainCorollary}\label{thm:cormain}
	If $G$ has an EDT0L word problem, then $G$ is a torsion group.
\end{restatable}

Our proof of \cref{thm:main} is an expanded version of the proof given by the fourth 
author~\cite[Chapter~8]{Gallot2021} in their PhD thesis.
It remains to be shown whether there can exist an infinite torsion group with EDT0L word problem.

A language related to the word problem of $\mathbb Z$ with respect to a standard generating set is the \emph{Dyck language} $D_1$, where
\[
    D_n
    =
    \bigl\{
        w \in \{a_1,b_1,\dots, a_n,b_n\}^*
    \bigm|
        |w|_{a_i} = |w|_{b_i}
        \text{ and, for each prefix }u\text{ of }w\text{ and each }i,\ 
        |u|_{a_i} \geq |u|_{b_i}
    \bigr\}
\]
where $|\cdot|_x$  counts the instances of the letter  $x$ in $w$.
By definition, we see that $D_1$ is a proper subset of the word problem of $\mathbb Z$ (taking $a_1$ as the generator $1$ and $b_1$ as $-1$).
It was shown in~\cite[Theorem~9]{Ehrenfeucht1977} that the  language $D_n$ is not EDT0L for $n\geq 8$.
This was later strengthened in~\cite[Theorem~6.3.2]{Brigitte1987} where it was shown that $D_1$ is not EDT0L.
However, this result and the techniques used therein were not helpful in proving \cref{thm:main}.

This paper is organised as follows.
In \cref{sec:background} we fix some terminology about geodesic words.
In \cref{sec:EDT0L-languages} we define the families of EDT0L and finite-index EDT0L languages, and provide some tools for proving that an EDT0L language has finite index.
In \cref{sec:edt0l-closure-props} we prove that the class of EDT0L languages is closed under string transduction.
In \cref{sec:edt0l-wp} we prove that if the word problem of a group is EDT0L, then it is EDT0L of finite index, and obtain \cref{lem:taking-a-submonoid-of-EDT0L}.
In \cref{sec:multiple-context-free} we define a \emph{non-permuting, non-erasing, non-branching multiple context-free grammar} to be a multiple context-free grammar that does not permit permutation of variables, does not allow erasure of variables, and restricts productions to maintain a single nonterminal.
For ease of reference, we refer to such a grammar as a \emph{restricted \MCF\ grammar} (abbreviated \emph{\RMCFG}) throughout the paper.
We prove that every finite-index EDT0L language is the language of an \RMCFG.
Finally, in \cref{sec:main-theorem}, we prove
\cref{thm:main,thm:cormain}.

\subsection{High-level overview of the proof of \protect{Theorem~\ref{thm:main}}}

This subsection furnishes a very brief outline of the main ideas of the proof, to motivate the technology required in the sections leading up to and including \cref{sec:main-theorem}.
We start the proof of \cref{thm:main} by assuming for contradiction that the word problem of $\mathbb Z$ with respect to the generators $1,-1$ is generated by a given \RMCFG, and moreover we assume the grammar is in a particular \emph{normal form} (see~\cref{def:normal-form}). 
From this grammar, we obtain a bound $C$ related to the difference in the number of $1$ versus $-1$ letters in a tuple words that appear in a step of any derivation (see \cref{lem:derivation_bound_basic} for details). 
 Using this bound,  we construct a word $\mathcal W$, which we call the \emph{counterexample word}, having a very specific and controlled form (which uses two parameters $m$, $k$ coming from the grammar and bound $C$). An example of $\mathcal W$ for $k=2$, $m=4$ is shown in \cref{fig:placeholder}, where the generator $1$ is represented by an up-step $(1,1)$, and $-1$ by a down-step $(1,-1)$.
\begin{figure}[!h]
    \centering
    \includegraphics{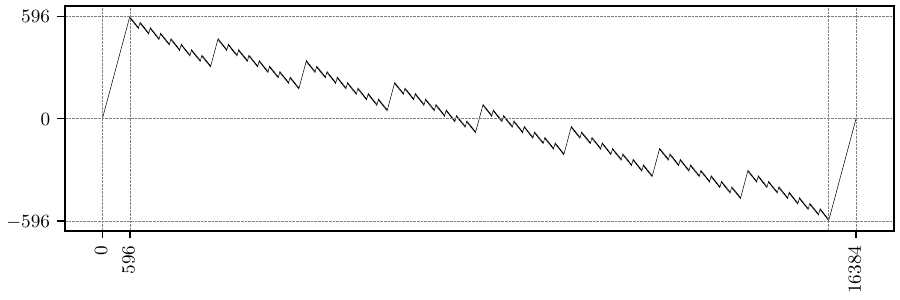}
    \caption{The word $\mathcal W$ defined in \cref{sec:main-theorem} for values  $k=2$ and $m=4$.}
    \label{fig:placeholder}
\end{figure}
The basic idea of the construction of this word is to ensure that it contains ``long'' up-step factors, having lengths exceeding some bound (depending on $C$ and the grammar)  and occurring sufficiently frequently, and
all down-step factors of constant size, 
 such that the total difference between the number of up and down steps is zero (so that $\mathcal W$ is in the word problem of $\mathbb Z$).
We show that any derivation of the R-MCFG in normal form 
generating $\mathcal W$ 
will have the form 
\[S(\mathcal W)\leftarrow H_{\ell}(\mathcal W,\varepsilon,\dots,\varepsilon)\leftarrow\cdots\leftarrow H_{1}(\varepsilon,\dots,\varepsilon)\leftarrow\] (for some nonterminals $H_1,\dots, H_\ell$ and starting nonterminal $S$; see  \cref{def:nonbranching-mcf})
where the words in the tuple $(\varepsilon,\dots,\varepsilon)$  do not have any long up-step 
factors, whereas $\mathcal W$ does (in particular, its prefix and suffix are such factors). \Cref{def:decomposition} defines a property of tuples that is enjoyed by $(\mathcal W,\varepsilon,\dots,\varepsilon)$ and not by 
$(\varepsilon,\dots,\varepsilon)$. Specifically, a tuple is said to \emph{have a decomposition} (as in \cref{def:decomposition}) if 
a particular inequality which compares the lengths and occurrences of long up-step factors at the prefix and suffix of coordinates, as well as several other quantities, is satisfied.
We then prove (\cref{prop:decomp-full}) that for each step of a derivation of the R-MCFG, if the target tuple has a decomposition then so must the source tuple. Since  $(\mathcal W,\varepsilon,\dots,\varepsilon)$ has a decomposition and $(\varepsilon,\dots,\varepsilon)$ does not, we have our contradiction.

\subsection{Notation}

We write $\mathbb N = \{0,1,2,3,\ldots\}$ for the set of nonnegative integers including zero, and 
$\mathbb N_+=\mathbb N\setminus \{0\}$.
Unless otherwise stated, monoid actions are applied on the right: we make this choice to simplify the constructions within the proofs below.
An \emph{alphabet} is a finite set. If $X$ is an alphabet, we let $X^*$ denote the set of words (finite length strings) of letters from $X$, and $X^+=X^*\setminus\{\varepsilon\}$.
We will sometimes speak of the word problem of a monoid.
Suppose that $M$ is a monoid with finite generating set $X$, and that $1_M$ is the monoid identity.
The word problem of $(M,X)$ is given by
\[
  \WP(M,X)
  \coloneqq
  \{
    w\in X^*
  \mid
    \overline{w}=1_M
  \}
\]
where $\overline\cdot \colon X^* \to M$ is the natural homomorphism from the words over $X$ to the monoid.

\section{Geodesics in finitely-generated groups}\label{sec:background}

The results in this paper combine topics in formal language theory and geometric group theory.
This section is provided for readers who are not familiar with some of the terms and basic notations of geometric group theory which are used in this paper.
In particular, this section explains what it means for a word to be \emph{geodesic} in a finitely generated group.
Readers who are familiar with this terminology can safely skip this section.

Suppose that $G$ is a group which is generated as a monoid by a finite subset $X$.
That is, we have a surjective homomorphism, which we denote as $\overline\cdot\colon X^* \to G$, from the free monoid over $X$ to the group $G$.
Thus, for each element $g\in G$, there exists at least one word $w\in X^*$ such that $\overline{w}=g$.
If the word $w\in X^*$ is a minimal-length word for which $\overline{w}=g$, then we say that $w$ is a \emph{geodesic} for $g$.
Notice that each element $g\in G$ must have at least one corresponding geodesic.

For example, suppose that $G$ is the group given by vectors in $\mathbb Z^2$ with group operator being vector addition.
Such a group is generated by $X = \{x,x^{-1}, y, y^{-1}\}$ where $\overline{x}=(1,0)$ and $\overline y=(0,1)$.
The element $(2,3)$ can be written as any of the following words over $X$:
\begin{align*}
  xxyyy &&
  yxyxy &&
  y^{-1}xxyyyy &&
  x^{-1} x x^{-1} yxyxyx.
\end{align*}
Notice here that the word $xxyyy$ and $yxyxy$ are both geodesics.

We have the following properties of geodesics which are used in this paper:
\begin{enumerate}
  \item any prefix or suffix of a geodesic is also a geodesic;
  \item the only geodesic for the group identity is the empty word; and
  \item if $G$ is an infinite group, then it has arbitrarily long geodesics.
\end{enumerate}
The above properties are straightforward to prove, and left as an exercise to the reader.

\section{EDT0L and EDT0L of finite index}\label{sec:EDT0L-languages}

In this section we provide background on the family of EDT0L languages and the subfamily of EDT0L languages of finite index.
Our aim here is to prove \cref{lem:shuffle-product-edt0l} due to Latteux~\cite{Latteux1980}, which provides a method to show that a given language is EDT0L of finite index.
We begin this section by defining the class of EDT0L grammars as follows.
We note that our definition of EDT0L grammar differs from their original definition given in \cite[Definition~1 and~2]{Rozenberg1973et0l}, in particular, we make additional assumptions of our \emph{tables} in (\ref{def:edt0l-grammar/3}) of \cref{def:edt0l-grammar}.
However, it can be seen from \cite[Lemma~2.2]{Ehrenfeucht1974} that our definition produces the same class of languages to that of \cite[Definition~1 and~2]{Rozenberg1973et0l}.

\begin{definition}
\label{def:edt0l-grammar}
	An \emph{EDT0L grammar} is a 4-tuple $E = (\Sigma, V, I, H)$ where
	\begin{enumerate}
		\item $\Sigma$ is an alphabet of \emph{terminal letters};
		\item $V$ is an alphabet of \emph{nonterminal letters} which is disjoint from $\Sigma$;
		\item $I \in V$ is an \emph{initial symbol}; and
		\item\label{def:edt0l-grammar/3}
        $H \subset \mathrm{End}((V\cup \Sigma)^*)$ is a finite set of monoid endomorphisms such that $\sigma\cdot h=\sigma$ for each $h\in H$ and each $\sigma\in \Sigma$.
        The endomorphisms $h\in H$ are called \emph{tables}.
	\end{enumerate}
	The language generated by such a grammar $E$ is given by
	\[
		L(E)
		\coloneq
		\{
    		I\cdot \alpha
		\mid
    		\alpha \in H^*
		\} \cap \Sigma^\ast
	\]
	where $I\cdot\alpha$ denotes the (right) action of $\overline\alpha\in \mathrm{End}((V\cup\Sigma)^*)$ on the word $I$.
	That is, $L(E)$ is the set of words in $\Sigma^*$ that can be obtained by applying any sequence of maps from $H$ to $I$.
\end{definition}

For example, the language 
\[
  L
  =
  \{
    a^{n_1}b a^{n_2} b a^{n_3}b \cdots a^{n_k}b
    \mid
    n_i,k\in \mathbb N
    \text{ with }
    0\leq n_1 \leq n_2 \leq \cdots \leq n_k
  \}
\]
is EDT0L \cite[Proposition 7]{CiobanuElderFerov} as it can be generated using the grammar $E = (\Sigma, V, I, H)$ where
\begin{enumerate}
  \item $\Sigma = \{a,b\}$;
  \item $V = \{I,A\}$; and
  \item $H = \{h_1,h_2,h_3\}$ with
    \begin{align*}
      v \cdot h_1 &=
      \begin{cases}
        I A b & \text{if }v=I,\\
        v & \text{otherwise},
      \end{cases}
      &
      v \cdot h_2 &=
      \begin{cases}
        Aa & \text{if }v=A,\\
        v & \text{otherwise},
      \end{cases}
      &
      v \cdot h_3 &=
      \varepsilon
    \end{align*}
    for each $v\in V$.
\end{enumerate}
In particular, the word $abaaab\in L$ can be generated as $I\cdot (h_1 h_2h_2 h_1 h_2 h_3)$.

Suppose that $E=(\Sigma, V,I,H)$ is an EDT0L grammar.
Then, we call $w \in (\Sigma\cup V)^*$ a \emph{sentential form} if there exists some sequence of tables $\alpha \in H^*$ such that $w = I \cdot \alpha$.
For example, in the grammar given above, the word $AabAaab$ is a sentential form which can be generated as $I\cdot(h_1 h_2h_1)$.
Notice that the initial word $I$ and every word generated by the grammar are examples of sentential forms.

Given a word $w\in X^*$ over a finite alphabet $X$, we write $|w|$ for its length.
For each letter $x\in X$, we write $|w|_x$ to denote the number of instances of the letter $x$ in $w$.
Moreover, for any subset $V\subset X$, we write $|w|_V$ to denote $\sum_{v\in V} |w|_v$.
We now define the EDT0L languages of finite index as follows.

\begin{definition}[See \cite{Rozenberg1978}]\label{def:edt0l-finite-index}
	An EDT0L grammar $E = (\Sigma,V,I, H)$ is of \emph{index $n$} if $|I\cdot \alpha|_{V} \leq n$ for each monoid endomorphism $\alpha \in H^*$.
	That is, if there are at most $n$ nonterminals in any sentential form.
	An EDT0L grammar is said to be of finite index if it is an EDT0L grammar of index $n$ for some $n\in \mathbb{N}$.
	A language is \emph{EDT0L of finite index} if it can be generated by an EDT0L grammar of finite index.
\end{definition}

From the formal grammar description of regular languages, we have the following result.

\begin{lemma}\label{lem:reg_is_edt0l}
	Regular languages are EDT0L of index 1.
\end{lemma}

In order to prove the results in this section, we define the family of \emph{LULT} grammars as follows.

\begin{definition}\label{def:lult}
	An EDT0L grammar $E = (\Sigma,V,I, H)$ is \emph{LULT} (which stands for \emph{EDT0{\bfseries L}-système {\bfseries ult}ralinéaire}, cf.~\cite[p.\,361]{Latteux1980}) if for each word $w\in L(E)$, there exists some $\alpha\in H^*$ with $w = I\cdot \alpha$ such that for each factorisation $\alpha = \alpha_1 \alpha_2$ with $\alpha_1,\alpha_2\in H^*$, and each $v \in V$, either $|I \cdot \alpha_1|_v \leq 1$ or $|v \cdot \alpha_2|\leq 1$.
\end{definition}

In the next result we show that if a language is generated by a LULT grammar, then the language is EDT0L of finite index. Throughout the proof, we write $f \colon A \rightharpoonup B$ to denote a partial function from $A$ to $B$, and we write $\mathrm{dom}(f)\subseteq A$ for its domain.
That is, $f\colon A \rightharpoonup B$ is a function from some subset of $A$ to a subset of $B$.
Moreover, we write $f = \emptyset$ if $\mathrm{dom}(f) = \emptyset$.

\begin{lemma}[Latteux~\cite{latteux1977}]\label{lem:lult-implies-finite-index}
	If $E = (\Sigma,V,I,H)$ is an EDT0L grammar which is LULT, then the language $L(E)$ is EDT0L of finite index.
\end{lemma}

\begin{proof}
	Let $E = (\Sigma,V,I,H)$ be a LULT grammar as in the lemma statement.
	In this proof, we construct an EDT0L grammar $E' = (\Sigma,V',I', H')$ with index $|V|+1$ such that $L(E) = L(E')$.
	In particular, we first construct the grammar $E'$, then prove that it generates precisely the language $L(E)$.

The construction given in this proof is quite involved.
In order to assist the reader's comprehension, we provide an informal overview as follows.
From the definition of LULT grammars, we see that if $w \in \Sigma^*$ is generated by the grammar $E$, then there must exist some sequence of tables $\alpha\in H^*$ with $w = I\cdot \alpha$ such that for each factorisation $\alpha = \alpha_1\alpha_2$, and each nonterminal $v\in V$, we have either $|v\cdot \alpha_1|_v\leq 1$ or $|v\cdot \alpha_2| \leq 1$.
Notice that we do not assume that $\alpha$ is uniquely defined.
The idea of our construction is to produce a grammar $E'$ which only considers such sequences of tables.
In particular, we construct our grammar $E'$ such that each sentential form $u\in (V'\cup \Sigma)^*$ of $E'$ is either some \textit{dead-end} word (see case~\ref{item:case2} in part~\hyperref[lem:lult-implies-finite-index/1.1]{1.1} of this proof) which cannot lead to an accepted word, or corresponds to some word $u' = I\cdot \alpha_1$ of $E$ (as before) where the nonterminals $v$ with $|u'|_v\geq 2$ have been replaced with the words $v\cdot \alpha_2\in \Sigma^*$.
Notice then that the remaining  nonterminals in $u$ each appear at most once, thus resulting in our grammar having finite index as desired.
In order to correctly emulate such productions of the grammar $E$, for technical reasons, each of the nonterminals of $E'$ must store information about which nonterminals are present in the sentential form: to do so, we introduce the following nonterminals.
Their properties are explained in more detail at the end of part~\hyperref[lem:lult-implies-finite-index/1.1]{1.1} of this proof.

	\medskip

	\noindent\underline{1.~Nonterminals\phantomsection\label{lem:lult-implies-finite-index/1}}:\nopagebreak

	\smallskip\nopagebreak\noindent
	We begin our construction by defining our set of nonterminals $V'$ as
	\begin{multline*}
		V'
		=
		\{
		\mathfrak d, I'
		\}
		\cup
		\{
		X_{a,A,f}
		\mid
		a\in A \subseteq V
		\text{ and }
		f\colon V \rightharpoonup (\Sigma\cup \{\varepsilon\})
		\text{ with }
		\mathrm{dom}(f)\cap A = \emptyset
		\}\\
		\cup
		\{
		Y_{A,f}
		\mid
		A \subseteq V
		\text{ and }
		f\colon V\rightharpoonup (\Sigma\cup\{\varepsilon\})
		\text{ with }
		\mathrm{dom}(f)\cap A = \emptyset
		\}.
	\end{multline*}
	We immediately notice that this set of nonterminals is finite, in particular, we have
	\[
		|V'|
		\leq
		2
		+ \left(|V|\cdot 2^{|V|}\cdot (|\Sigma|+2)^{|V|}\right)
		+ \left(2^{|V|}\cdot (|\Sigma|+2)^{|V|}\right).
	\]
	We note here that the symbol $I'$ is the starting symbol of the new grammar, and that $\mathfrak d$ is an additional nonterminal which we call the \emph{dead-end symbol}.
	In our construction, we ensure that each table maps $\mathfrak d$ to itself.
	Thus, once $\mathfrak d$ enters a sentential form, there is no way to continue to generate a word in the language associated to the grammar.
	We describe the remaining symbols of $V'$ in detail in part~\hyperref[lem:lult-implies-finite-index/1.1]{1.1}.

Informally, the nonterminals $X_{a,A,f}$ correspond to instances of the nonterminal $a\in A$ within sentential forms $w$ of $E$, where each nonterminal $v$ that appear in $w$ either has the property that (1) it occurs at most once in $w$ and $v\in A$, or (2) it goes on to produce the word $f(v)\in \Sigma^*$.
Similarly, the nonterminals $Y_{A,f}$ are included to describe such a sentential form: these nonterminals are required for the case where we are representing a sentential form $w$ of $E$ where each nonterminal in $w$ occurs more than once.

	\medskip

	\noindent\underline{1.1.~Properties of sentential forms of $E'$\phantomsection\label{lem:lult-implies-finite-index/1.1}}:\nopagebreak

	\smallskip\nopagebreak\noindent
	In our construction we ensure that, for each sequence of tables of the form $\alpha \in (H')^*$, the resulting sentential form $w = I' \cdot \alpha \in (\Sigma\cup V')^*$ is in one of the following four forms:
	\begin{enumerate}
		\item\label{item:case1} \textit{initial:}
        $w = I'$, that is, $w$ contains only the starting symbol;
		\item\label{item:case2} \textit{dead-end:}
        $\mathfrak d$ is the only nonterminal in $w$ and $|w|_\mathfrak{d}=1$;
		\item\label{item:case3} \textit{production:}
        there exists some $A\subseteq V$ and $f\colon V \rightharpoonup (\Sigma\cup\{\varepsilon\})$ with $\mathrm{dom}(f)\cap A = \emptyset$ such that
		      \[
			      \{
			      X_{a,A,f}\mid a\in A
			      \}
			      \cup
			      \{
			      Y_{A,f}
			      \}
		      \]
		      are the only nonterminals which appear in $w$, each such nonterminal appears exactly once in $w$, and the nonterminal $Y_{A,f}$ appears as the last letter of $w$; or
        \item\label{item:case4} \textit{final:}
            $w \in \Sigma^*$, that is, the word contains no nonterminals.
	\end{enumerate}
	Notice that, with such a property, the grammar would be EDT0L of index $|V|+1$.

    Notice that if a word $w$ is a production as in case~\ref{item:case3}, as described above, then $w = w' Y_{\emptyset,\emptyset}$ would imply that $w' \in \Sigma^*$.
    This is clear as, from the definition of the normal form, $Y_{\emptyset, \emptyset}$ would be the only nonterminal which may appear in such a normal form, i.e., there are no nonterminals of the form $X_{a,\emptyset,f}$.

	Suppose that $\beta_1,\beta_2 \in (H')^*$ with $I' \cdot (\beta_1 \beta_2) \in \Sigma^*$, and that $w = I'\cdot \beta_1 \in (\Sigma\cup V)^*$ is a word as in case~\ref{item:case3}, as above, with $A\subseteq V$ and $f\colon V \rightharpoonup (\Sigma\cup\{\varepsilon\})$.
	In our construction, these sequences $\beta_1,\beta_2 \in (H')^*$ correspond to some sequences $\alpha_1, \alpha_2 \in H^*$ in the grammar $E$ for which
	\begin{itemize}
		\item $I\cdot\alpha_1 \in (\Sigma\cup\mathrm{dom}(f) \cup A)^*$;
		\item for each $a \in A$, we have $|I \cdot \alpha_1|_a \leq 1$; and
		\item for each $b \in \mathrm{dom}(f)$, we have $|b\cdot \alpha_2|\leq 1$, in particular, we have $b\cdot \alpha_2 = f(b)$.
	\end{itemize}
	Compare the above with the definition of LULT grammars as in \cref{def:lult}.
	In the remainder of this proof, we construct the tables $H'$ of $E'$ such that the above properties hold and $L(E)=L(E')$.

	\medskip

	\noindent\underline{2.~Table $t_\mathrm{init}$}:\nopagebreak

	\smallskip\noindent\noindent
	We begin the productions of our grammar $E'$ with a table $t_\mathrm{init}\in H'$ defined as
	\[
		v\cdot t_\mathrm{init}
		=
		\begin{cases}
			X_{I, \{I\},\emptyset}\, Y_{\{I\}, \emptyset}
			  & \text{if }v=I',   \\
			v & \text{otherwise}.
		\end{cases}
	\]
	Notice that applying this table preserves the property of a word belonging to one of the four forms described in part~\hyperref[lem:lult-implies-finite-index/1.1]{1.1} of this proof.
	In the remainder of this proof, we notice that in each production of $E'$, we may assume that $t_\mathrm{init}$ is the first table we apply, and that it is applied exactly once.

	\medskip

	\noindent\underline{3.~Table $t_\mathrm{end}$}:\nopagebreak

	\smallskip\nopagebreak\noindent
	We end the production of our grammar with the table $t_\mathrm{end} \in H'$ defined as
	\[
		v\cdot t_\mathrm{end}
		=
		\begin{cases}
			v           & \text{if }v\in \Sigma,                                                           \\
			\mathfrak d & \text{if }v=Y_{A,f}\text{ where either }A\neq\emptyset\text{ or }f\neq\emptyset, \\
			\mathfrak d & \text{if }v\in \{I',\mathfrak d\},                                               \\
			\varepsilon & \text{otherwise}.
		\end{cases}
	\]
	Notice that applying this table preserves the property of a word being in one of the four forms described in part~\hyperref[lem:lult-implies-finite-index/1.1]{1.1} of this proof.
	Moreover, we notice that given a word $w\in (\Sigma\cup V')^*$ in one of the forms described in part~\hyperref[lem:lult-implies-finite-index/1.1]{1.1} of this proof, that $w\cdot t_\mathrm{end} \in \Sigma^*$ if and only if either $w\in \Sigma^*$ or $w=w'Y_{\emptyset,\emptyset}$ in which case we would have $w'\in \Sigma^*$ (see part~\hyperref[lem:lult-implies-finite-index/1.1]{1.1} of this proof).

	\medskip

	\noindent\underline{4.~Tables $t_{h,B,g}$}:\nopagebreak

	\smallskip\nopagebreak\noindent
	Let $h\in H$ be a table from the grammar $E$, let $B\subseteq V$ be a subset of nonterminals of $E$, and let $g\colon V\rightharpoonup (\Sigma\cup\{\varepsilon\})$ be a partial function with $\mathrm{dom}(g)\cap B = \emptyset$.
	We then introduce a new table $t_{h,B,g}\in H'$.
    Such a table corresponds to an application of the table $h$ of $E$, and is intended to produce words as in case~\ref{item:case3} of part~\hyperref[lem:lult-implies-finite-index/1.1]{1.1} which end in the nonterminal $Y_{B,g}$ (with some edge cases for invalid productions).

    We begin our construction of this table by specifying that 
	\begin{itemize}
		\item $\sigma\cdot t_{h,B,g} = \sigma$ for each $\sigma\in \Sigma$;
		\item $\mathfrak d \cdot t_{h,B,g} = \mathfrak d$; and
		\item $I'\cdot t_{h,B,g} = \mathfrak d$.
	\end{itemize}
    Thus, we see that if $w\in (\Sigma\cup V')^*$ is a word as in case~\ref{item:case1}, \ref{item:case2} or~\ref{item:case4} as in part~\hyperref[lem:lult-implies-finite-index/1.1]{1.1} of this proof, then $w\cdot t_{h,B,g}$ is also a word of such a form.
    Thus, it only remains to specify the table such that we understand how it acts on word in the form described in case~\ref{item:case3} of part~\hyperref[lem:lult-implies-finite-index/1.1]{1.1}.

    Let $A\subseteq V$ and $f\colon V \rightharpoonup (\Sigma\cup\{\varepsilon\})$ with $\mathrm{dom}(f)\cap A = \emptyset$; and
    let $w$ be a word as in case~\ref{item:case3} of part~\hyperref[lem:lult-implies-finite-index/1.1]{1.1} which ends in the nonterminal $Y_{A,f}$.
    In the remainder of this part, we describe the word $w \cdot t_{h,B,g}$, in particular, we specify the action of the table on the nonterminals of the form $X_{a,A,f}$ and $Y_{A,f}$.

    \medskip
    
    \noindent
    \textit{\underline{4.1.~Notation}.\@}
    For each $C\subseteq V$ and each partial function $k\colon V \rightharpoonup (\Sigma\cup\{\varepsilon\})$ with $\mathrm{dom}(k)\cap C = \emptyset$, we define a monoid homomorphism $\mathrm{Label}_{C,k}\colon (\Sigma\cup C\cup \mathrm{dom}(k))^* \to (\Sigma\cup V')^*$ such that
    \[
        \mathrm{Label}_{C,k}(v)
        \coloneqq
        \begin{cases}
            v & \text{if }v\in \Sigma\\
            X_{v,C,k} &\text{if }v\in A\\
            k(v) &\text{if }v\in \mathrm{dom}(k)
        \end{cases}
    \]
    for each $v\in \Sigma\cup C\cup \mathrm{dom}(k)$.
    From our choice of $C$ and $k$, we see that $\mathrm{Label}_{C,k}$ is well-defined.

    Let $k\colon V\rightharpoonup (\Sigma\cup\{\varepsilon\})$ be a partial function as above, then we write $\widetilde k \in \mathrm{End}(\Sigma\cup V)^*$ for the table defined such that
    \[
        v\cdot \widetilde k
        \coloneqq
        \begin{cases}
            k(v) &\text{if }v\in \mathrm{dom}(k)\\
            v &\text{otherwise}
        \end{cases}
    \]
    for each $v\in (\Sigma\cup V)$.

    \medskip
    
    \noindent
    \textit{\underline{4.2.~Intention}.\@}
    Our intention is for our construction to have the following property.
    If there exists some $\beta_1,\beta_2\in(H')^*$ such that $w = I' \cdot \beta_1$ and $I'\cdot(\beta_1\, t_{h,B,g}\, \beta_2) \in \Sigma^*$, then there exists some $\alpha_1,\alpha_2\in H^*$ where
    \begin{enumerate}[label=(I\arabic*),ref=\arabic*]
        \item\label{item:nextprop}
        $I'\cdot(\beta_1\,t_{h,B,g}\cdot \beta_2)= I\cdot (\alpha_1\,h\,\alpha_2)$;
        \item\label{item:nextprop2}
        $I' \cdot \beta_1 = \mathrm{Label}_{A,f}(I\cdot \alpha_1)\,Y_{A,f}$ is a word as in case~\ref{item:case3} of part~\hyperref[lem:lult-implies-finite-index/1.1]{1.1}; and
        \item\label{item:nextprop3}
        $I' \cdot (\beta_1 \, t_{h,B,g}) = \mathrm{Label}_{B,g}(I\cdot \alpha_1 h)\,Y_{B,g}$ is a word as in case~\ref{item:case3} of part~\hyperref[lem:lult-implies-finite-index/1.1]{1.1}.
    \end{enumerate}
    Thus, the table $t_{t,B,g}$ of $E'$ behaves similarly to the table $h$ of $E$ except it performs additional replacements, i.e.~the replacements given by $f$ and $g$ which are completely specified by the nonterminals used in the sentential form.
    Thus, our reason for including the nonterminals $Y_{A,f}$ and $Y_{B,g}$ in these words is so that this information is always present in our sentential form.

    \medskip
    
    \noindent
    \textit{\underline{4.3.~Technical details}.\@}
	We label the elements of $A$ as $A = \{x_1,x_2,\ldots,x_k\}\subseteq V$.
	We then define two sets $B^{(1)}, B^{(\geq2)}\subseteq V$ as
	\[
		B^{(1)}
		  \coloneqq
		\bigl\{ v\in V \bigm|
    	   | (x_1 x_2\cdots x_k) \cdot h |_v = 1
        \bigr\}
        \qquad\text{and}\qquad
		B^{(\geq2)}
		  \coloneqq
		\bigl\{ v\in V \bigm|
		| (x_1 x_2\cdots x_k) \cdot h |_v \geq 2
		\bigr\}
	\]
	where $h \in H$ is the table as in $t_{h,B,g}$.
	We are then interested in the case where the sets $B^{(1)}$ and $B^{(\geq2)}$ satisfy the following 4 additional properties:
	\begin{enumerate}[label=(T\arabic*),ref=\arabic*]
		\item\label{lem:lult-implies-finite-index/prop1} $B^{(\geq2)} \subseteq \mathrm{dom}(g)$;
		\item\label{lem:lult-implies-finite-index/prop2} $B = B^{(1)} \setminus \mathrm{dom}(g)$;
		\item\label{lem:lult-implies-finite-index/prop3} 
        for each $a\in A$, we have
        $a\cdot h \in (\Sigma\cup \mathrm{dom}(g)\cup B)^*$; and
		\item\label{lem:lult-implies-finite-index/prop4} 
        for each $v\in \mathrm{dom}(f)$, we have
        $v\cdot h \in (\Sigma\cup \mathrm{dom}(g))^*$ with $f(v) = \tilde g( v\cdot h )$.
	\end{enumerate}
	If any of these above properties (T\ref{lem:lult-implies-finite-index/prop1}\,--\,\ref{lem:lult-implies-finite-index/prop4}) are violated, then we define
	\begin{align*}
		X_{a,A,f}\cdot t_{h,B,g} & = \varepsilon
		                         &
		\text{and}
		                         &               &
		Y_{A,f} \cdot t_{h,B,g}  & = \mathfrak d
	\end{align*}
	for each $a\in A$, which produces a dead-end word as in case~\ref{item:case2} of part~\hyperref[lem:lult-implies-finite-index/1.1]{1.1}.
    Otherwise, if properties (T\ref{lem:lult-implies-finite-index/prop1}\,--\,\ref{lem:lult-implies-finite-index/prop4}) are satisfied, then we define
    \begin{align}\label{lem:lult-implies-finite-index/construct}
        X_{a,A,f}\cdot t_{h,B,g}
        &=\mathrm{Label}_{B,g}(a\cdot h)
        &\text{and}&&
        Y_{A,f} \cdot t_{h,B,g} &= Y_{B,g} 
    \end{align}
    for each $a\in A$.
    From the definition of $B^{(1)}$, $B^{(\geq 2)}$ and properties (T\ref{lem:lult-implies-finite-index/prop1}\,--\,\ref{lem:lult-implies-finite-index/prop4}), we see that the applications of $\mathrm{Label}_{B,g}$, as above in (\ref{lem:lult-implies-finite-index/construct}), are well-defined.

    Notice that we may then completely describe the action of the tables $t_{h,B,g}$ by repeating the above construction for any valid values of $A$ and $f$.

    \medskip
    
    \noindent
    \textit{\underline{4.4.~Property}.\@}
    Suppose that $w\in (\Sigma\cup V')^*$ is a word in the form described in case~\ref{item:case3} of part~\hyperref[lem:lult-implies-finite-index/1.1] {1.1}.
    Then, from the construction of $B^{(1)}$ and properties T\ref{lem:lult-implies-finite-index/prop1} and T\ref{lem:lult-implies-finite-index/prop2}, as above, we see that $w\cdot t_{h,B,g}$ is either a dead-end word as in case~\ref{item:case2} of part~\hyperref[lem:lult-implies-finite-index/1.1] {1.1}, or a production as in case~\ref{item:case3} of part~\hyperref[lem:lult-implies-finite-index/1.1] {1.1}.
    Further, suppose additionally that there exists some word $u \in (\Sigma\cup V)^*$ such that $w = \mathrm{Label}_{A,f}(u)\, Y_{A,f}$.
    Then either $w\cdot t_{h,B,g}$ is a dead-end word; otherwise it follows from property T\ref{lem:lult-implies-finite-index/prop4} that
    $
        w \cdot t_{h,B,g}
        =
        \mathrm{Label}_{B,g}(u\cdot h)
        \, Y_{B,g}
    $.
    
	\medskip

	\noindent\underline{5.~Soundness and Completeness}:\nopagebreak

	\smallskip\nopagebreak\noindent
	Suppose that $w\in L(E)$, then from the definition of LULT grammars, we see that there must exist some sequence of tables $\alpha = h_1 h_2 \cdots h_k \in H^*$ such that for each factorisation $\alpha=\alpha_1\alpha_2$ with $\alpha_1,\alpha_2\in H^*$, and each nonterminal $v\in V$, either $|I\cdot \alpha_1|_v \leq 1$ or $|v\cdot \alpha_2|\leq 1$.
	For each $i\in \{1,2,\ldots, k-1\}$, let
	\[
		C_i
		=
		\bigl\{
		v\in V
		\bigm|
		|I\cdot (h_1 h_2\cdots h_k)|_v\geq 1
		\bigr\}.
	\]
	We then see that
	\[
		I\cdot (h_1 h_2\cdots h_i) \in (\Sigma \cup C_i)^*
	\]
	for each $i\in \{1,2,\ldots, k-1\}$.
	From our choice of sequence $\alpha=h_1 h_2\cdots h_k$, we see that we can partition each set $C_i$ into two disjoint subsets
	\begin{align*}
		B_i       & = \bigl\{ v\in V \bigm| |I\cdot (h_1 h_2 \cdots h_{i})|_v = 1 \text{ and } |v\cdot (h_{i+1}h_{i+1}\cdots h_k)| \geq 2 \bigr\}\subseteq C_i\text{ and} \\
		C_i^{(*)} & = \bigl\{ v\in V \bigm| |I\cdot (h_1 h_2 \cdots h_{i})|_v \geq 1 \text{ and } |v\cdot(h_{i+1}h_{i+2}\cdots h_k)| \leq 1 \bigr\}\subseteq C_i.
	\end{align*}
	In particular, the two sets correspond to the two cases in the definition of LULT grammars.

	Notice then that, for each $i\in \{1,2,\ldots,k-1\}$, we can define a partial function $g_i\colon V \rightharpoonup (\Sigma\cup\{\varepsilon\})$ with $\mathrm{dom}(g_i) = C_i^{(*)}$ such that $g_i(v) = v\cdot(h_{i+1} h_{i+2}\cdots h_k)$ for each $v\in C_i^{(*)}$.

	We then see that the word $w$ is generated by the grammar $E'$ as
	\[
		w = I' \cdot
		\left(
		t_\mathrm{init}\
		t_{h_1, B_1, g_1}\
		t_{h_2, B_2, g_2}\
		\cdots
		t_{h_{k-1}, B_{k-1}, g_{k-1}}\
		t_{h_k, \emptyset, \emptyset}\
		t_\mathrm{end}
		\right).
	\]
	Thus, we see that $L(E)\subseteq L(E')$.

	Suppose then that $w\in L(E')$ where $w = I' \cdot (\tau_1 \tau_2 \cdots \tau_k)$ with each $\tau_i \in H'$.
	Then, from the definitions of the tables in $H'$, we may assume without loss of generality that
	\begin{itemize}
		\item $k\geq 3$ with $\tau_1 = t_\mathrm{init}$ and $\tau_k = t_\mathrm{end}$; and
		\item for each $i \in \{2,3,\ldots,k-1\}$, the table $t_i$ is of the form $\tau_i = t_{h_i, B_i, g_i}$ where each $h_i\in H$.
	\end{itemize}
	From our construction, we then see that
	$
		w = I\cdot ( h_2 h_3 \cdots h_{k-1} ).
	$
	Thus, we have $L(E')\subseteq L(E)$.

	We thus conclude that $L(E) = L(E')$ where $E'$ is an EDT0L grammar of index $|V|+1$.
\end{proof}

Using \cref{lem:lult-implies-finite-index}, we have the following result.

\begin{lemma}[Proposition 25 in \cite{Latteux1980}]\label{lem:shuffle-product-edt0l}
	Let $L \subseteq \Sigma^*$ and $c\notin \Sigma$.
	If the language
	\[\delimiterfactor=1000
		L\uparrow c^*
		\coloneqq
		\left\{
		c^{n_0} w_1 c^{n_1} w_2 c^{n_2} \cdots w_k c^{n_k} \in (\Sigma\cup \{c\})^*
		\ \middle|\
		\begin{aligned}
			w = w_1 w_2 \cdots w_k \in L\text{ and} \\
			n_0,n_1,n_2,\ldots,n_k\in \mathbb{N}
		\end{aligned}
		\right\}
	\]
	is EDT0L, then $L$ is EDT0L of finite index.
\end{lemma}

\begin{proof}
	Suppose that $E = (\Sigma\cup \{c\},V, I, H)$ is an EDT0L grammar for the language $L\uparrow c^*$.
	We then introduce a monoid homomorphism $h\colon (\Sigma\cup \{c\}\cup V)^*\to (\Sigma\cup V)^*$ such that
	\[
		h(c) = \varepsilon
		\quad
		\text{and}
		\quad
		h(v)=v
		\text{ for each }
		v\neq c.
	\]
	From this, we define an EDT0L grammar $E' = (\Sigma, V, I, H')$ where
	$
		H' = \left\{
		t h
		\mid
		t\in H
		\right\}.
	$
	Notice that since $c$ is a terminal letter of the grammar $E$, deleting the letter $c$ after every application of a table is equivalent to removing all instances of $c$ at the end of a production.
	Thus, we have $L(E') = L(E)\cdot h = L$.
	It only remains to be shown that $E'$ is a EDT0L language of finite index.
	We demonstrate this by proving that $E'$ is a LULT grammar, after which we apply \cref{lem:lult-implies-finite-index} to obtain our result.

	To simplify this proof, we define a (length-preserving) monoid homomorphism $\varphi\colon H^* \to (H')^* $ which we define such that $\varphi(t) = th\in H'$ for each $t \in H$.
	Notice also that $H' = \varphi(H)$.
	Moreover, since $c$ is a terminal letter, we see that $w \cdot \varphi(\alpha) = w\cdot (\alpha h)$ for each $\alpha \in H^*$ and each $w\in (\Sigma\cup\{c\}\cup V)^*$, that is, removing the letter $c$ after every application of a table is equivalent to removing every $c$ only at the end.

	Suppose we are given a word $w' = w_1 w_2 \cdots w_k\in L(E')$, then it follows that
	\[
		w = w_1 c w_2 c^2 w_3 c^3 \cdots w_k c^k \in L\uparrow c^*.
	\]
	We then see that there exists some $\alpha=h_1 h_2\cdots h_k\in H^*$ with $w = I \cdot\alpha$ and thus $w' = I\cdot(\alpha h) = I\cdot \varphi(\alpha)$.
	We now show that the sequence $\alpha' = h_1' h_2'\cdots h_k' \in H'$, where each $h'_i=\varphi(h_i)$, is a choice of a sequence of tables which generates $w'$ in $E'$ and satisfies the constraints of a LULT grammar (see \cref{def:lult}).

	Suppose for contradiction that there is a factorisation $\alpha' = \alpha_1'\alpha_2'$ and a nonterminal $v\in V$ such that both $|I\cdot \alpha_1'|_v\geq 2$ and $|v\cdot \alpha_2'|\geq 2$.
	Let $\alpha = \alpha_1 \alpha_2$ be the unique factorisation of $\alpha \in H^*$ with $\alpha_1' = \varphi(\alpha_1)$ and $\alpha'_2 = \varphi(\alpha_2)$.
	We then have the follow two observations:
	\begin{itemize}
		\item Since $|I\cdot \alpha_1'|_v\geq 2$, from the definition of $\varphi$, we see that $|I\cdot \alpha|_v\geq 2$, that is, $I\cdot \alpha_1$ contains at least 2 distinct instances of the variable $v$.
		\item Since $|v\cdot \alpha_2'|\geq 2$, from the definition of the word $w$ and the map $\varphi$, we see that $v\cdot \alpha_2$ must contain a factor of the form $\sigma_1 c^m \sigma_2$ where $\sigma_1,\sigma_2\in \Sigma$ and $m\in \mathbb N_+$.
	\end{itemize}
	From these observations, we see that the word $w = I\cdot (\alpha_1\alpha_2)$ must contain two distinct factors of the form $\sigma_1 c^m \sigma_2$.
	This contradicts our choice of word $w$.
	Hence, we conclude that either $|I\cdot \alpha'_1|_v \leq 1$ or $|v\cdot \alpha_2'|\leq 1$ holds.
	From this, we then see that $E'$ is a LULT grammar.

	From \cref{lem:lult-implies-finite-index}, we conclude that the language $L$ is EDT0L of finite index.
\end{proof}

\section{EDT0L is closed under application of string transducers}\label{sec:edt0l-closure-props}

In this section, we furnish a proof that the family of EDT0L languages is closed under mapping by a \emph{string transducer}, also known as a \emph{deterministic finite-state transducer}, or a \emph{deterministic generalised sequential machine (deterministic gsm)}.
We begin with the following
definition.

\begin{definition}\label{def:det-fst}
	A (deterministic) \emph{string transducer} is a tuple
	$M = (\Gamma, \Sigma, Q, A, q_0, \delta)$ where
	\begin{itemize}
    \item $\Gamma$ and $\Sigma$ are the \emph{input} and \emph{output alphabets}, respectively;
    \item $Q$ is a finite set of \emph{states};
    \item $A\subseteq Q$ is a finite set of \emph{accepting states};
    \item $q_0 \in Q$ is the \emph{initial state}; and
    \item $\delta\colon \Gamma \times Q \to \Sigma^*\times Q$ is a \emph{transition function}.
	\end{itemize}
	Given a language $L\subseteq \Gamma^*$, we may then define the language $M(L)\subseteq \Sigma^*$ as
	\[
		M(L)
		=
		\left\{
		u_1 u_2 \cdots u_k\in \Sigma^*
		\ \middle|\
		\begin{aligned}
			\text{there exists some word }
			w = w_1 w_2 \cdots w_k \in L\subseteq \Gamma^*
			\\\text{ such that }
			\delta(w_i, q_{i-1}) = (u_i, q_i)
			\text{ for each }i \in \{1,2,\ldots,k\} \\
			\text{ where }
			q_0\text{ is the initial state, and }
			q_1,q_2,\ldots,q_k\in Q
			\text{ with }
			q_k\in A
		\end{aligned}
		\right\}.
	\]
	We then say that $M(L)$ is the image of $L$ under mapping by the string transducer $M$.
\end{definition}

\begin{example}
Here is a simple example of a
 string transducer, which computes the successor of a non-negative integer.
A non-negative integer is represented as a word of the form $w\$$ where $w\in\{0,1\}^*$ is the minimum-length binary encoding of the integer (with the least significant digit appearing first), and $\$$ is an end-of-input symbol.
For example, we encode the numbers $0$, $4$, $10$ and $11$ as $\$$, $001\$$, $0101\$$ and $1101\$$, respectively.
Notice that $100\$$ would not be a valid string as it does not represent the number 1 with minimal length.
The string transducer 
 $M=(\Gamma,\Sigma,Q,A,q_0,\delta)$ where $\Gamma=\Sigma=\{0,1,\$\}$, $Q=\{q_0,q_1,q_2,q_3,q_4\}$, $A=\{q_4\}$, and $\delta\colon\Gamma\times Q\to \Sigma^*\times Q$ is described by
\begin{align*}
  \delta(0,q_0)&=(1,q_1), &
  \delta(0,q_1)&=(0,q_1), &
  \delta(0,q_2)&=(0,q_1), &
  \delta(0,q_3)&=(0,q_3), &
  \delta(0,q_4)&=(0,q_3),
  \\
  \delta(1,q_0)&=(0,q_0), &
  \delta(1,q_1)&=(1,q_2), &
  \delta(1,q_2)&=(1,q_2), &
  \delta(1,q_3)&=(1,q_3), &
  \delta(1,q_4)&=(1,q_3),
  \\
  \delta(\$,q_0)&=(1\$,q_4), &
  \delta(\$,q_1)&=(\$,q_3), &
  \delta(\$,q_2)&=(\$,q_4), &
  \delta(\$,q_3)&=(\$,q_3), &
  \delta(\$,q_4)&=(\$,q_3)
\end{align*}
computes the successor of a given number. 
For example, if $\$$, $001\$$, $0101\$$ and $1101\$$ are given to the transducer, then it will output the words $1\$$, $101\$$, $1101\$$ and $00101\$$, respectively. We can represent the string transducer by the
graph in \cref{fig:adding-machine} where the vertices are given by the state set $Q$, and for each transition $\delta(u,q)=(v,q')$ there is a labelled edge of the form $q\to^{u/v} q'$.

\begin{figure}[!hpt]
  \centering

  \begin{tikzpicture}[>=stealth,initial text={initial}]
    \node [initial,state] (q_0) {$q_0$};
    \node [state] (q_nocarry0) [below=6em of q_0] {$q_1$};

    \coordinate (midpoint) at ($(q_0.south)!0.5!(q_nocarry0.north)$); 

    \node [state] (q_nocarry1) [right=4em of midpoint] {$q_2$};
    \node [state] (q_fail) [right=8em of q_nocarry0] {$q_3$};
    \node [state,accepting] (q_accept) [right=8em of q_0] {$q_4$};

    \path[->]
      (q_0)
        edge [bend left] node [above] {$\$/1\$$} (q_accept)
        edge [bend right] node [left] {$0/1$} (q_nocarry0)
        edge [loop above] node [above] {$1/0$} ()
      (q_nocarry0)
        edge [bend right] node [below] {$\$/\$$} (q_fail)
        edge [loop left] node [left] {$0/0$} ()
        edge [bend left] node [above left=-0.5em] {$1/1$} (q_nocarry1)
      (q_nocarry1)
        edge node[below right=-0.25em] {$\$/\$$} (q_accept)
        edge [bend left] node [below right=-0.5em] {$0/0$} (q_nocarry0)
        edge [loop above] node [above] {$1/1$} ()
      (q_fail)
        edge [loop right] node [right] {$\$/\$,\ 0/0,\ 1/1$} ()
      (q_accept)
        edge[bend left] node [right]{$\$/\$,\ 0/0,\ 1/1$} (q_fail)
      ;

  \end{tikzpicture}

  \caption{Add one to a number encoded in binary with an end marker.}\label{fig:adding-machine}
\end{figure}

Note that the state
\begin{itemize}
  \item $q_0$ corresponds to prefixes of the form $1^n$ for some $n\in \mathbb N$;
  \item $q_1$ corresponds to prefixes of the form $w 0$ where $w\in \{0,1\}^*$;
  \item $q_2$ corresponds to prefixes which can be written as $w0v1$ where $w,v\in \{0,1\}^*$;
  \item $q_3$ corresponds to invalid input sequences; and
  \item $q_4$ corresponds to valid sequences of the form $w\$$ where $w\in \{0,1\}^*$.
\end{itemize}
For the interested reader, further examples of string transducers abound in the literature on self-similar groups, see for example \cite{grigorchuk2006}.
\end{example}

The following definition provides some useful notation when working with string transducers.
\begin{definition}\label{def:det-fst-maps}
	Let $M = (\Gamma,\Sigma,Q,A,q_0,\delta)$ be a string transducer.
	Then, for each pair of states $q,q'\in Q$, and words $w = w_1 w_2 \cdots w_k\in \Gamma^*$ and $w'\in \Sigma^*$, we write $q\to^{(w,w')}q'$ if there is a path from state $q$ to $q'$ which rewrites the word $w$ to $w'$; that is, if there is a sequence of states $q_1,q_2,\ldots,q_{k+1}\in Q$ such that
	\begin{itemize}
		\item $q = q_1$ and $q' = q_{k+1}$; and
		\item $\delta(w_i, q_i) = (u_i, q_{i+1})$ for each $i\in \{1,2,\ldots,k\}$ where $w' = u_1 u_2 \cdots u_k$.
	\end{itemize}
	Notice then that
	\[
		M(L)
		=
		\{
		w'\in \Sigma^*
		\mid
		q_0 \to^{(w,w')} q
		\text{ where }
		w\in L\text{ and }q\in A
		\}
	\]
	for each $L\subseteq \Gamma^*$.
\end{definition}

We then have the following. 

\begin{lemma}[Corollary 4.7 in \cite{Engelfriet1976}]\label{lem:edt0l-closed-under-fs-transduction}
	The family of EDT0L languages is closed under applying a string transducer.
  That is, if $L$ is an EDT0L language, and $M$ is a string transducer, then $M(L)$ is also an EDT0L language.
\end{lemma}

\begin{proof}
	Let $L$ be an EDT0L language with EDT0L grammar $E = (\Gamma, V, I, H)$, and $M = (\Gamma, \Sigma,Q, A, q_0, \delta)$ be a string transducer.
	In this proof, we construct an EDT0L grammar $E' = (\Sigma, V', I', H')$ for the language $M(L)$, thus showing that $M(L)$ is EDT0L.

In this paragraph we give an informal description of our construction.
We construct our grammar $E'$ such that it contains a nonterminal for every choice of nonterminal $v\in V$ of $E$, and every choice of states $q,q'\in Q$ of $M$.
That is,  $E'$ will have nonterminals of the form $X_{v,q,q'}$ where $v\in V$ and $q,q'\in Q$.
We construct $E'$ such that it has the property that, if
\begin{equation}\label{eq:transformed-word}
    w' =
    w_0'
    X_{v_1,q_1,q'_1} w_1'
    X_{v_2,q_2,q'_2} w_2'
    \cdots
    X_{v_k,q_k,q'_k} w_k'
\end{equation}
is a sentential form that appears in some production of $E'$ where each $w_i'\in \Sigma^*$, then there exists some sentential form
\[
    w
    =
    w_0
    v_1 w_1
    v_2 w_2
    \cdots
    v_k w_k
\]
which appears in a production of $E$ where each $w_i\in \Gamma^*$ such that
\begin{itemize}
    \item $q_0 \to^{(w_0,w'_0)} q_1$,
    \item $q_i \to^{(w_i, w_i')} q_{i+1}$ for each $i\in \{1,2,\ldots,k-1\}$ and
    \item  $q_k \to^{(w_k, w'_k)} q'$ for some $q'\in A$.
\end{itemize}
That is, the words $w_i'$ in (\ref{eq:transformed-word}) correspond to words in $\Gamma^*$ which have been rewritten by $M$, and the nonterminals $X_{v_i,q_i,q'_i}$ are placeholders for words which are yet to be rewritten by $M$. Moreover, the states $q_i,q_i'$ are included to ensure that the sentential form $w'$ in (\ref{eq:transformed-word}) corresponds to a path in the automaton $M$.
We achieve this by constructing the tables of $E'$ such that, for every replacement
\begin{equation}\label{eq:transformed-word2-before}
    v\cdot h = w_0 u_1 w_1 u_2 w_2 \cdots u_k w_k,
\end{equation}
where $v\in V$, $h\in H$, each $w_i\in \Gamma^*$ and each $u_i\in V$, we have a table $h' \in H'$ with the property that
\begin{equation}\label{eq:transformed-word2}
    X_{v,q,q'}
    \cdot h'
    =
    w_0'
    X_{u_1, q_1,q_1'} w_1'
    X_{u_2, q_2,q_2'} w_2'
    \cdots
    X_{u_k, q_k,q_k'} w_k'
\end{equation}
where
\begin{itemize}
    \item $q \to^{(w_0,w'_0)} q_1$,
    \item $q_i \to^{(w_i, w'_i)} q_{i+1}$ for each $i\in \{1,2,\ldots,k-1\}$ and
    \item  $q_k \to^{(w_k, w'_k)} q'$.
\end{itemize}
Thus, the application of the table $h'$ as in (\ref{eq:transformed-word2}) corresponds to the application of $h$ in (\ref{eq:transformed-word2-before}) where each word in $\Gamma^*$ has been rewritten by $M$, and we are representing a path in $M$ from $q$ to $q'$.
It then follows that the language of all words in $M(L)$ corresponds to the union of the languages obtained from EDT0L grammars with initial nonterminals of the form $X_{I,q_0,a}$ where $a\in A$: we construct our grammar to correspond to this union by introducing a new start nonterminal $I'$ along with some additional tables $t_{\mathrm{init},a}\in H'$.
We finish this informal description of our construction by noting that there may be cases where there are no choices of states $q_i,q_i'\in Q$ for which we can define an expansion as in (\ref{eq:transformed-word2}), thus we introduce an additional nonterminal $\mathfrak d$ which we call a \emph{dead-end symbol} which stops  the production in such situations.

	\medskip

	\noindent\underline{1.~Nonterminals}:\nopagebreak

	\smallskip\nopagebreak\noindent
	The set of nonterminals of $E'$ is given as
	\[
		V' = \{
		X_{v,q,q'}
		\mid
		v\in V
		\text{ and }
		q,q' \in Q
		\}\cup \{I', \mathfrak{d}\}.
	\]
	Here, the nonterminals $X_{v,q,q'}$  of $E'$ correspond to nonterminals $v\in V$ of $E$ which generate words which are read by the string transducer $M$ during a path from state $q$ to $q'$;
	the nonterminals $I'$ and $\mathfrak{d}$ are disjoint symbols which are used as the starting symbol and a \emph{dead-end symbol} respectively.
	In our grammar, we ensure that every table maps $\mathfrak{d}$ to itself.
	Thus, if a $\mathfrak{d}$ appear in a sentential form, then there is no way of removing it to continue to generate a word in the language.

	\medskip

	\noindent\underline{2.~Initial tables $t_\mathrm{init,a} \in H'$}:\nopagebreak

	\smallskip\nopagebreak\noindent
	For each accepting state $a\in A$ of the string transducer $M$, we introduce a table $t_{\mathrm{init},a} \in H'$ such that
	\[
		v \cdot t_{\mathrm{init},a}
		=
		\begin{cases}
			v            & \text{if }v\in \Sigma, \\
			X_{I,q_0,a}  & \text{if }v=I',        \\
			\mathfrak{d} & \text{otherwise}.
		\end{cases}
	\]
(Note that we write $v \cdot t_{\mathrm{init},a}=v$ for $v\in\Sigma$ only to comply with \cref{def:edt0l-grammar}, in our grammar to produce a word in the language such tables will only be applied to the sentential form $I'$.)

	The remaining tables in $H'$ are modified versions of tables in $H$, described as follows.

	\medskip

	\noindent\underline{3.~Tables $t_{h,r} \in H'$}:\nopagebreak

	\smallskip\nopagebreak\noindent
	For each $h\in H$ and each $v\in V$, we have
	\[
		v \cdot h = w_0 x_1 w_1 x_2 w_2 \cdots x_{k_{v}} w_{k_{v}}
	\]
	for some $k_{v}\in \mathbb{N}$ where each $w_i\in \Sigma^*$ and each $x_i \in V$.
  Then for each pair of states $q,q' \in Q$, we define a finite set $C_{h,v,q,q'}\subset (\Sigma\cup V')^*$ as
	\[
		C_{h,v,q,q'}
		=
		\left\{
		\begin{aligned}
			u_0 \,
			X_{x_1,q_1,q_1'} \, u_1\,
			X_{x_2,q_2, q_2'}\, u_2\quad \\
			\cdots
			X_{x_{k_v},q_{k_v},q_{k_v}'}\, u_{k_v}
		\end{aligned}
		\ \middle|\
		\begin{aligned}
			 &
			q \to^{(w_0,u_0)} q_1,                                         \\
			 &
			q_i' \to^{(w_i, u_i)} q_{i+1} \text{ for each }1 \leq i < k_v, \\
			 &
			\text{and }q_{k_v}'\to^{(w_{k_v}, u_{k_v})} q'
		\end{aligned}
		\right\}
	\]
	Notice that each set $C_{h,v,q,q'}$ is finite as there are only finitely many choices for each state $q_i'\in Q$;
    the state $q_1$ is completely determined from $q$ and $w_0$;
    for each $i\geq 1$, that state $q_{i+1}$ is completely determined from $q_i'$ and $w_i$; and the words  $u_i\in \Sigma^*$ are completely determined from the state $q$ and states $q_i$.

	From the sets $C_{h,v,q,q'}$, we define a set of functions $R_h\subset((\Sigma\cup V')^*\cup\{\mathfrak{d}\})^{V\times Q\times Q}$ as
	\[
		R_h
		=
		\left\{
		r\colon V\times Q\times Q
		\to
		((\Sigma\cup V')^*\cup\{\mathfrak{d}\})
		\ \middle|\
		\begin{aligned}
			\text{where }r(v,q,q')\in C_{h,v,q,q'}\cup \{\mathfrak{d}\} \\
			\text{ for each }v\in V\text{ and }q,q'\in Q
		\end{aligned}
		\right\}.
	\]
	Notice that $R_h$ is finite, in particular,
	\[
		|R_h|
		\leq \prod_{v\in V}\prod_{q\in Q}\prod_{q'\in Q} (|C_{h,v,q,q'}|+1).
	\]
	Notice that we added $\mathfrak{d}$ as an option in $R_h$ so that there is always a choice for the value of $r(v,q,q')$, in particular, so that there is a choice when $C_{h,v,q,q'}$ is empty.

	For each $r\in R_h$, we then define a table $t_{h,r}\in H'$ such that
	\[
		X_{v,q,q'}
		\cdot t_{h,r}
		=
		r(v,q,q')
	\]
	for each $v\in V$, and each $q,q'\in Q$.

	\medskip

	\noindent\underline{4.~Soundness and completeness}:\nopagebreak

	\smallskip\nopagebreak\noindent
	Suppose that $w \in L(E')$, then there must exist some $\tau_1,\tau_2,\ldots,\tau_k\in H'$ such that
	\[
		w = I' \cdot (\tau_1 \tau_2 \cdots \tau_k).
	\]
	From the definition of the initial tables $t_{\mathrm{init},a}$, we see that
	\begin{itemize}
		\item $\tau_1 = t_{\mathrm{init},a}$ for some $a\in A$; and
		\item $\tau_i = t_{h_i, r_i}$ for each $i\neq 1$.
	\end{itemize}
	From the construction of the tables of the form $t_{h,r}$, the word
	\[
		u = I\cdot (h_2 h_3\cdots h_k) \in L(E)
	\]
	has the property that $q_0 \to^{(u,w)} a$, and thus $w \in M(L(E))$.
	Hence, $M(L(E))\subseteq L(E')$.

	Suppose that $w\in M(L(E))$, then there exists some word $u\in L(E)$ and an accepting state $a\in A$ such that $q_0\to^{(u,w)} a$.
	Thus, there exists a sequence $h_1h_2\cdots h_k \in H^*$ such that
	\[
		u = I \cdot (h_1 h_2 \cdots h_k)\in L(E).
	\]
	From the construction of our tables, we then see that
	\[
		w = I' \cdot ( t_{\mathrm{init},a} t_{h_1,r_1} t_{h_2, r_2} \cdots t_{h_k, r_k} )
	\]
	for some choice of functions $r_1,r_2,\ldots,r_k$.
	Hence, we see that $L(E') \subseteq M(L(E))$.

	We now conclude that $L(E')= M(L(E))$, and thus the family of EDT0L languages is closed under mapping by string transducer.
\end{proof}

The following lemma, which makes use of string transducers, will be used in the next section. Recall that a subset $W\subset \Gamma^+$ is an \emph{antichain} with respect to prefix order if for each choice of words $u,v\in W$, the word $u$ is not a proper prefix of $v$.

\begin{lemma}\label{lem:f_is_detfsm}
	Let $W\subset \Gamma^+$ be a finite antichain with respect to prefix order.
	For each word $w\in W$, we fix a word $x_w\in \Sigma^*$.
	Define a map $f\colon \mathcal{P}(\Gamma^*) \to \mathcal{P}(\Sigma^*)$ as
	\[
		f(L)
		=
		\{
		x_{w_1} x_{w_2} \cdots x_{w_k} \in \Sigma^*
		\mid
		w_1 w_2 \cdots w_k \in L
		\text{ where each }
		w_i\in W
		\}.
	\]
	Then, there is a string transducer $M = (\Gamma,\Sigma, Q, A, q_0, \delta)$ such that $f(L) = f(L\cap W^*) = M(L)$.
\end{lemma}

\begin{proof}
	Firstly we notice that if $W = \emptyset$, then $f(L) = \emptyset$ for each language $L\subseteq \Gamma^*$.
	In this case, any such string transducer with $A = \emptyset$ satisfies the lemma statement.
	Thus, in the remainder of this proof, we assume that $W\neq \emptyset$.

	Let $w\in \Gamma^*$ be a word for which $w\in W^*$.
	Then, since $W$ is a finite antichain with respect to the prefix order, there is a unique factorisation of $w$ as
	$
		w =w_1 w_2 \cdots w_k
	$
	where each $w_i\in W$.

	We construct a string transducer $M = (\Gamma,\Sigma, Q, A, q_0, \delta)$ as follows.
	For each proper prefix $u \in \Gamma^*$ of a word $w\in W$, we introduce a state $q_u \in Q$.
	The initial state is $q_0 =q_\varepsilon$, and the set of accepting states is $A = \{q_\varepsilon\}$.
	Further, our automaton has one additional state $q_\mathrm{fail}$ which is a fail state; that is,
	\[
		\delta(g, q_\mathrm{fail}) = (\varepsilon, q_\mathrm{fail})
	\]
	for each $g\in \Gamma$.
	We then specify the remaining transitions as follows.

	For each state $q_u$ with $u\in \Gamma^*$, and each $g\in \Gamma$, we define the transition
	\[
		\delta(g, q_u)
		=
		\begin{cases}
			(x_w, q_\varepsilon)           & \text{if } w = ug \in W,                                \\
			(\varepsilon, q_{ug})          & \text{if } ug\text{ is a proper prefix of some }w\in W, \\
			(\varepsilon, q_\mathrm{fail}) & \text{otherwise}.
		\end{cases}
	\]
	We then see that the string transducer $M$ is now completely specified.
	It is clear from our construction that $f(L)=M(L)$ for each $L\subseteq\Gamma^*$.
\end{proof}

\section{EDT0L word problem}\label{sec:edt0l-wp}

A standard approach to showing that having word problem in a given formal class $\mathcal C$
is invariant under change of finite generating set is to rely on the fact that $\mathcal C$ is closed under inverse monoid homomorphism. For example, this holds when $\mathcal C$ 
is any class in the Chomsky hierarchy, indexed, ET0L, and finite-index EDT0L (for the class of finite-index EDT0L languages, see \cite[Theorem~5]{Rozenberg1978}).

As noted in the introduction, 
EDT0L languages are not closed under taking inverse monoid homomorphism.
In particular, the language $L_1= \{a^{2^n}\mid n\in \mathbb N\}$ can easily be shown to be EDT0L, however, the language $L_2=\{w\in \{a,b\}^*\mid |w|=2^n \text{ for some }n\in \mathbb{N}\}$, which can be written as an inverse monoid homomorphism of $L_1$, is known to not be EDT0L (see either the proof of Theorem~1 on p.\,22 of \cite{Rozenberg1973}, or Corollary~2 on p.\,22 of \cite{Ehrenfeucht1975}).
Instead, we prove \cref{lem:taking-a-submonoid-of-EDT0L} via the following steps.

Recall from \cref{sec:background} that if $G$ is infinite then it has geodesics of arbitrary length.

\begin{lemma}\label{lem:example-antichain}
	Suppose that $G$ is an infinite group with finite monoid generating set $X$.
	Fix a finite number of words $u_1,u_2,\ldots,u_k\in X^*$.
	Then there exists a choice of non-empty words $w_1,w_2,\ldots,w_k\in X^*\setminus \{\varepsilon\}$, such that each $\overline{w_i}=1$ and
	\[
		W = \{
		w_1 u_1, w_2 u_2,\ldots, w_k u_k
		\}
	\]
	is an antichain in the sense of \cref{lem:f_is_detfsm}; that is, for each choice of words $x,y\in W$, the word $x$ is not a proper prefix of $y$.
\end{lemma}

\begin{proof}
	We begin by constructing the words $w_1,w_2,\ldots,w_k$ as follows.

	Let $\alpha_1 \in X$ be a nontrivial generator, that is, $\overline{\alpha_1}\neq 1$; then let $\beta_1 \in X^*$ be a geodesic with $\overline{\alpha_1\beta_1} = 1$ (Notice that if $X$ is a symmetric generating set, then we may choose $\beta_1 = \alpha_1^{-1}$).
	From this selection, we then define $w_1 = \alpha_1\beta_1$.
	We now choose the words $w_2,w_3,\ldots,w_k$ sequentially as follows.

	For each $i\geq 2$, we choose a geodesic $\alpha_i\in X^*$ with length $|\alpha_i|= |\alpha_{i-1}\beta_{i-1}|+1$.
	We then choose a geodesic $\beta_i\in X^*$ such that $\overline{\alpha_i\beta_i}=1$ (Notice that if $X$ is a symmetric generating set, then we may choose $\beta_i = \alpha_i^{-1}$).
	Then, define $w_i = \alpha_i\beta_i$.

	We have now selected a sequence of words $w_1,w_2,\ldots,w_k$.
	For each word $w_i$, let $\gamma_i$ denote the longest prefix which is a geodesic.
	We then notice that
	\[
		|\gamma_{i-1}| < |w_{i-1}| < |\alpha_{i}| \leq |\gamma_{i}|
	\]
	for each $i\in\{2,3,\ldots,k\}$.
	Thus, $|\gamma_1|< |\gamma_2|< \cdots< |\gamma_k|$ and $|\gamma_i| < |w_iu_i|$ for each $i$.

	We now see that, if $w_i u_i$ is a proper prefix of some word $v\in X^*$, then $\gamma_i$ is also the longest prefix of $v$ which is a geodesic.
	Hence, we conclude that the set
	\[
		W = \{
		w_1 u_1, w_2 u_2,\ldots, w_k u_k
		\}
	\]
	is an antichain as required.
\end{proof}

Recall from the introduction 
that the word and coword problem for a group $G$ with respect to a generating set $X$ are denoted as 
\begin{align*}
	\mathrm{WP}(G,X)
	 & \coloneq
	\{
	w \in X^*
	\mid
	\overline{w} = 1
	\}
	 & \text{and} &  &
	\mathrm{coWP}(G,X)
	\coloneq
	X^*\setminus \mathrm{WP}(G,X).
\end{align*}
We have the following result for groups with EDT0L (co-)word problem.

\begin{proposition}\label{prop:wp-edt0l-fi}
	Let $G$ be a group with finite monoid generating set $X$.
	If the word problem $\mathrm{WP}(G,X)$ is EDT0L, then it is EDT0L of finite index.
	Moreover, if the coword problem $\mathrm{coWP}(G,X)$ is EDT0L, then it is EDT0L of finite index.
\end{proposition}

\begin{proof}
	Notice that if $G$ is finite, then this result follows from the fact that the word problem of a finite group is a regular language (see~\cite[Theorem 1]{Anisimov1971}) and \cref{lem:reg_is_edt0l}.
	Thus, in the remainder of this proof we may assume that $G$ is an infinite group.

	Let $c$ be a letter which is disjoint from the alphabet $X = \{x_1,x_2,\ldots,x_n\}$.
	From \cref{lem:example-antichain}, there exists a choice of words $w_1,w_2,\ldots,w_n,w_{n+1}\in X^*$ such that
	\[
		W
		=
		\{
		w_1 x_1, w_2 x_2,\ldots,w_n x_n, w_{n+1}
		\}
	\]
	is an antichain with respect to the prefix order where each $\overline{w_i}=1$.

	We define a map $f\colon \mathcal{P}(X^*)\to\mathcal{P}((X\cup \{c\})^*)$ as
	\[
		f(L)
		=
		\left\{
		y_{i_1} y_{i_2} \cdots y_{i_k} \in 
		{(X\cup \{c\})^*}
		\ \middle|\
		\begin{aligned}
			(w_{i_1} u_{i_1})(w_{i_2} u_{i_2})\cdots(w_{i_k} u_{i_k}) \in L
			\\\text{ where each }
			i_j \in \{1,2,\ldots,n+1\}
		\end{aligned}
		\right\}
	\]
	where $u_i = y_i = x_i$ for each $i\in \{1,2,\ldots,n\}$, $u_{n+1}=\varepsilon$, and $y_{n+1}=c$.
	From \cref{lem:f_is_detfsm}, $f$ is a mapping by string transducer.

	Notice that $f(\mathrm{WP}(G,X)) = \mathrm{WP}(G,X)\uparrow c^*$ and $f(\mathrm{coWP}(G,X)) = \mathrm{coWP}(G,X)\uparrow c^*$.
  In particular, given some word $v= v_1v_2\cdots v_k\in X^*$, if $v\in \mathrm{WP}(G,X)$ then
  \[
    v'\coloneqq
    (w_{n+1})^{i_1} v_1
    (w_{n+1})^{i_2} v_2
    (w_{n+1})^{i_3} v_3
    \cdots
    (w_{n+1})^{i_k} v_k
    (w_{n+1})^{i_{k+1}}
    \in \mathrm{WP}(G,X)
  \]
  for each $i_j\in \mathbb N$ where $w_{n+1}\in X^*$ is the word as in the set $W$ given above.
  Then $f(v')$ is the word $c^{i_1} v_1 c^{i_2} v_2\cdots c^{i_k} v_k c^{i_{k+1}} \in \mathrm{WP}(G,X)\uparrow c^*$.
  Hence, we can construct any word in $\mathrm{WP}(G,X)\uparrow c^*$ in this way.
  A similar statement also holds for $\mathrm{coWP}(G,X)$ and $\mathrm{coWP}(G,X)\uparrow c^*$.

	From \cref{lem:edt0l-closed-under-fs-transduction}, we see that $\mathrm{WP}(G,X)\uparrow c^*$ and $\mathrm{coWP}(G,X)\uparrow c^*$ are EDT0L languages.
	Thus, from \cref{lem:shuffle-product-edt0l}, we conclude that $\mathrm{WP}(G,X)$ and $\mathrm{coWP}(G,X)$ are both EDT0L of finite index.
\end{proof}

We then obtain the following result as a corollary to the above Proposition.

\CorInvariantsDFSM*

\begin{proof}
Let $X$ be a finite monoid generating set for $G$, let $Y$ be a finite monoid generating set for a submonoid of $G$, and suppose that $\mathrm{WP}(G,X)$ (resp.~$\mathrm{coWP}(G,X)$) is EDT0L. Then by \cref{prop:wp-edt0l-fi},
$\mathrm{WP}(G,X)$ (resp.~$\mathrm{coWP}(G,X)$) is EDT0L of finite index. Let $\psi\colon Y^*\to X^*$ be a monoid homomorphism such that $\psi(y)=_G y$ for each $y\in Y$. Then $\mathrm{WP}(G,Y)=\psi^{-1}\left(\mathrm{WP}(G,X)\right)$ (resp.~$\mathrm{coWP}(G,Y)=\psi^{-1}\left(\mathrm{coWP}(G,X)\right)$).
From~\cite[Theorem~5]{Rozenberg1978}, it is known that the class of EDT0L languages of finite index is closed under inverse monoid homomorphisms, thus we conclude that the language $\mathrm{WP}(G,Y)$ (resp.~$\mathrm{coWP}(G,Y)$) is EDT0L of finite index, so in particular the language is EDT0L.
\end{proof}

\section{Non-branching multiple context-free languages}\label{sec:multiple-context-free}

In this section, we introduce non-branching multiple context-free grammars, and define what it means for such a grammar to be non-permuting and non-erasing.
We then show that every finite-index EDT0L language can be generated by a non-branching multiple context-free grammar which is both non-permuting and non-erasing.
We conclude this section by providing a normal form for grammars of this type which is then used throughout the proof of our main result in \cref{sec:main-theorem}.

\begin{definition}[Non-branching multiple context-free grammar]\label{def:nonbranching-mcf}
	Let $\Sigma$ be an alphabet, and $Q$ be a finite set of symbols called \emph{nonterminals} where each $H\in Q$ has a \emph{rank} some positive integer.
    We define an \emph{initiating rule} to be an expression of the form
	\[
		H(u_1,u_2, \ldots, u_k) \leftarrow
	\]
	where $H\in Q$ has rank $k$, and each $u_i\in \Sigma^*$. A \emph{propagating rule} is an expression of the form
	\[
		K(u_1, u_2,\ldots,u_m) \leftarrow H(x_1,x_2,\ldots,x_n)
	\]
	where
	\begin{itemize}
		\item $K\in Q$ has rank $m$,
		\item $H\in Q$ has rank $n$,
		\item $x_1, x_2,\ldots,x_n$ are abstract variables,
		\item each $u_i \in (\Sigma\cup\{x_1,x_2,\ldots,x_n\})^*$, and
		\item for each $i\in\{1,2,\ldots,n\}$, the word $u_1 u_2 \cdots u_m$ contains at most one instance of $x_i$.
	\end{itemize}
	A \emph{non-branching multiple context-free grammar} is a tuple $M = (\Sigma, Q, S, P)$ where $\Sigma$ is an alphabet, $Q$ is a finite set of nonterminals, $S\in Q$ is the \emph{starting nonterminal} which has rank 1, and $P$ is a finite set of rewrite rules each of which is either initiating or propagating.

	A word $w\in \Sigma^*$ is \emph{generated} by the grammar if there is a sequence of rules from $M$ starting with an initiating rule, then followed by propagating rules, and finishing at $S(w)$, that is, if there exists a sequence
  \[
		S(w)
		\leftarrow H_1(v_{1,1}, v_{1,2},\ldots, v_{1,m_1})
		\leftarrow H_2(v_{2,1}, v_{2,2},\ldots, v_{2,m_2})
		\leftarrow\cdots
		\leftarrow H_k(v_{k,1}, v_{k,2},\ldots, v_{k,m_k})
		\leftarrow
  \]
	where each $H_i \in Q$ with $H_k(v_{k,1}, v_{k,2},\ldots, v_{k,m_k}) \leftarrow$ initiating, each $v_{i,j}\in \Sigma^*$, and each replacement follows from $P$.
	We refer to the symbols $x_i$ used to define a propagating rule as the \emph{variables} of the rule.
  We refer to a sequence, as above, as a \emph{derivation} of the word $w$.
The language \emph{generated} by $M$ is the set of all words generated by $M$, that is, the set of all words which have a derivation as above.
\end{definition}

A non-branching multiple context-free grammar is \emph{non-permuting} if in each propagating replacement rule, the variables $x_i$ which appear in $u_1u_2\cdots u_m$, appear in the same order as on the right-hand side of the rule.
Moreover, a grammar is \emph{non-erasing} if for each propagating replacement rule, each variable $x_i$ which appears in the right-hand side also appears in the word $u_1u_2\cdots u_m$.
Notice then that if a non-branching multiple context-free grammar is both non-permuting and non-erasing, then we have
\[
	u_1u_2\cdots u_m \in \Sigma^* x_1 \Sigma^* x_2 \Sigma^* \cdots x_n \Sigma^*
\]
for each propagating rule.

As noted in the introduction, we call a non-permuting non-erasing non-branching multiple context-free grammar a \emph{restricted \MCF\ grammar}, abbreviated as \RMCFG.

\begin{proposition}\label{prop:equiv-langs}
	Let $L$ be EDT0L of finite index. Then $L$ is generated by an \RMCFG.
\end{proposition}

\begin{proof}
	Let $E = (\Sigma, V, I, H)$ be an EDT0L grammar of index $n$.
	In this proof, we construct an \RMCFG\ $M = (\Sigma,Q,S,P)$.

	For each sequence $v_1 v_2\cdots v_k \in V^*$ with length $k\leq n$, we introduce a nonterminal $H_{v_1v_2\cdots v_k}\in Q$ with rank $k+1$ to our non-branching multiple context-free grammar.
	Notice then that there are finitely many such nonterminals, in particular, $|Q| \leq (|V|+1)^n$.

	In our construction, a configuration of the form
	\[
		H_{v_1v_2\cdots v_k}(u_0, u_1,u_2,\ldots,u_k),
	\]
	where each $u_i\in \Sigma^*$, corresponds to a sentential form $u_0 v_1 u_1 v_2 u_2 \ldots v_k u_k$ generated by the EDT0L grammar $E$.
	Thus, the starting nonterminal of our non-branching multiple context-free grammar is $S = H_\varepsilon$.

	We begin our production by introducing an initiating rule of the form
	\[
		H_{I}(\varepsilon, \varepsilon) \leftarrow
	\]
	which corresponds to a word containing only the starting symbol $I$ of $E$.
	For each table $h \in H$ of our EDT0L grammar, we introduce a monoid endomorphism
	\[
		\overline{h}
		\colon (\Sigma\cup V\cup \{x_0,x_1,x_2,\ldots,x_n\})^*
		\to
		(\Sigma\cup V\cup \{x_0,x_1,x_2,\ldots,x_n\})^*
	\]
	such that $\overline{h}(u) = h(u)$ for each $u\in \Sigma\cup V$, and $\overline{h}(x_i) = x_i$ for each $i\in \{0,1,2,\ldots,n\}$.
	We note here that these additional disjoint symbols $x_0,x_1,\ldots,x_n$ will correspond to words in $\Sigma^*$.

	For each nonterminal $H_{v_1 v_2 \cdots v_k}\in Q$, we consider the word
	\[
		W_{h,H,v_1 v_2 \cdots v_k} \coloneqq \overline{h}(x_0 v_1 x_1 v_2 x_2 v_3 x_3 \cdots v_k x_k) \in (\Sigma\cup V\cup \{x_0,x_1,x_2,\ldots,x_k\})^*.
	\]
	Notice that $W_{h,H,v_1 v_2 \cdots v_k}$ contains exactly one instance of each symbol $x_0,x_2,\ldots,x_k$, and that these symbols appear in the same order in which they were given to the map $\overline{h}$.

	If the word $W_{h,H,v_1 v_2 \cdots v_k}$ contains at most $n$ instances of letters in $V$, then we may decompose it as
	\[
		W_{h,H,v_1 v_2 \cdots v_k}=w_0 v_1' w_1 v_2' w_2\cdots v_m' w_m
	\]
	where $m\leq n$, each $w_i\in (\Sigma\cup \{x_0,x_1,x_2,\ldots,x_k\})^*$, and each $v_i'\in V$.
	We introduce a propagating rule
	\[
		H_{v_1' v_2'\cdots v_m'}
		(w_0,w_1,\ldots,w_m)
		\leftarrow
		H_{v_1v_2\cdots v_k}(x_1,x_2,\ldots,x_k).
	\]
	Notice that this new rule is non-erasing and non-permuting from our earlier observations on $W_{h,H,v_1 v_2\cdots v_k}$.

	From our construction, we see that the non-branching multiple context-free language described in this proof is both non-permuting and non-erasing, so is an \RMCFG, and generates the finite-index EDT0L language $L(E)$.
\end{proof}

\begin{remark}
The converse of \cref{prop:equiv-langs} also holds, that is, every \RMCFG\ produces an EDT0L language of finite index.
In particular, given an \RMCFG\ one may construct an equivalent EDT0L language with nonterminals of the form $X_{H,k}$ where $H$ is a predicate of the \RMCFG\ and $k\in \{0,1,\ldots, r_H\}$ where $r_H$ is the rank of $H$.
One would then construct the EDT0L grammar such that productions of the \RMCFG\ of the form $H(u_1,u_2,\ldots,u_k)$ correspond to sentential forms $X_{H,0}\, u_1\, X_{H,1}\, u_2\, X_{H,2} \cdots u_k\, X_{H,k}$ of the EDT0L grammar.
We do not prove this result as it is not required by the results in this paper.
\end{remark}

We now describe three types of replacement rules for non-branching multiple context-free grammars as follows.
These types of rules will be used to describe a normal form for \RMCFG s.
We begin with insertions rules as follows.

\begin{definition}\label{def:rule-insertion}
	A replacement rule is a \emph{left} or \emph{right insertion} if it is of the form
	\[
		K(x_1,x_2,\ldots,x_{i-1}, \sigma x_{i}, x_{i+1},\ldots,x_k)
		\leftarrow
		H(x_1,x_2,\ldots,x_k)
	\]
	or
	\[
		K(x_1,x_2,\ldots,x_{i-1}, x_{i} \sigma, x_{i+1},\ldots,x_k)
		\leftarrow
		H(x_1,x_2,\ldots,x_k),
	\]
	respectively, for some $\sigma\in \Sigma$ and $i\in \{1,2,\ldots,k\}$.
\end{definition}

We then require a type of replacement rule which allows us to move the variables $x_i$ around.
We do so using merge rules defined as follows.

\begin{definition}\label{def:rule-merge}
	A replacement rule is a \emph{left} or \emph{right merge} if it is of the form
	\[
		K(x_1,x_2,\ldots,x_{i-1}, x_{i}x_{i+1}, \varepsilon, x_{i+2},\ldots,x_k)
		\leftarrow
		H(x_1,x_2,\ldots,x_k)
	\]
	or
	\[
		K(x_1,x_2,\ldots,x_{i-1}, \varepsilon, x_{i}x_{i+1}, x_{i+2},\ldots,x_k)
		\leftarrow
		H(x_1,x_2,\ldots,x_k),
	\]
	respectively, for some $i\in \{1,2,\ldots,k-1\}$.
\end{definition}

Further, we require a type of replacement rule which generate words as follows.

\begin{definition}\label{def:rule-accept}
	An \emph{accepting replacement rule} is one of the form
	\[
		S(x_1 x_2 \cdots x_k) \leftarrow H(x_1,x_2,\ldots, x_k)
	\]
	where $S$ is the starting nonterminal of the non-branching multiple context-free grammar.
\end{definition}

From these types of replacement rules as described in the previous definitions, we may now define normal forms for \RMCFG s as follows.

\begin{definition}\label{def:normal-form}
	An \RMCFG\ $M = (\Sigma,Q,S,P)$ is in \emph{normal form} if there exists some $k\geq 2$ such that every nonterminal $H\in Q$, except for the starting nonterminal $S$, has rank $k$, and each replacement rule in $P$ is either initiating of the form
	\[
		H(\underbrace{\varepsilon, \varepsilon,\varepsilon,\varepsilon, \ldots,\varepsilon}_{k\text{ components}})\leftarrow,
	\]
	a left or right insertion (as in \cref{def:rule-insertion}), a left or right merge (as in \cref{def:rule-merge}), or an accepting rule (as in \cref{def:rule-accept}).
	Moreover, every production of the grammar has the form
	\[
		S(w) \leftarrow H(w,\varepsilon,\varepsilon,\ldots, \varepsilon) \leftarrow\cdots \leftarrow K(\varepsilon,\varepsilon,\ldots,\varepsilon)\leftarrow
	\]
	where $w\in \Sigma^*$ and $H,K\in Q$.
\end{definition}

The use of the term `normal form' in the above definition is justified by the following lemma.

\begin{lemma}\label{lem:normal-form}
	Suppose that $L\subseteq \Sigma^*$ is the language of an \RMCFG.
	Then, $L$ can be generated by an \RMCFG\ in normal form.
\end{lemma}

\begin{proof}
	Let $M = (\Sigma,Q,S,P)$ be an \RMCFG.
	Let $k\in \mathbb N$ be the smallest value for which both $k\geq 2$ and for each replacement rule
	\[
		H(w_1,w_2,\ldots,w_\ell)\leftarrow
		\qquad
		\text{and}
		\qquad
		H(w_1,w_2,\ldots,w_n)\leftarrow K(x_1,x_2,\ldots,x_m)
	\]
	in $P$, we have $k\geq \max\{\ell,m,n\}$.
	Such a constant $k$ exists as there are finitely many replacement rules.

	In this proof, we construct three \RMCFG s $M'$, $M''$ and $M'''$.
	At the end of this proof, we have a grammar $M'''$ in normal form which generates the same language as $M$.

    The techniques used in this proof are quite straightforward.
    In step~1, we ensure that all nonterminals have the same rank by simply padding each predicate with empty components, as required;
    in step~2, we unfold each production of the grammar into a sequence of rules each of the form \cref{def:rule-accept,def:rule-insertion,def:rule-merge} or a trivial rule of the form
    \[
        H(x_1,x_2, \cdots, x_k) \leftarrow K(x_1,x_2,\ldots,x_k);
    \]
    then finally, in step~3, we remove these additional trivial rules to produce a grammar in normal form.

	\medskip

	\noindent
	\underline{Step 1}:
	We construct an \RMCFG, denoted as $M' = (\Sigma,Q',S', P')$, from $M$ as follows.

	For each nonterminal $A\in Q$ of the grammar $M$, we introduce the nonterminals
	\[
		A_0, A_1, A_2,\ldots,A_k\in Q'.
	\]
	In our construction of $M'$, we ensure that if
	\[
		A_i(w_1,w_2,\ldots,w_k)
	\]
	appears in a derivation, then $w_j=\varepsilon$ for each $j > i$. Thus, the subscript of these nonterminals count the number of non-empty components.
	Moreover, we introduce the nonterminals
	\[
		S', F' \in Q'
	\]
	to our grammar $M'$.
	These additional nonterminals are used in the grammar $M'$ as follows.

	The grammar $M'$ has the replacement rules
	\[
		S'(x_1 x_2 \ldots x_k) \leftarrow S_1(x_1,x_2,\ldots,x_k)
		\qquad
		\text{and}
		\qquad
		F'(\underbrace{\varepsilon, \varepsilon,\varepsilon,\varepsilon, \ldots,\varepsilon}_{k\text{ components.}})\leftarrow
	\]
	For each replacement rule
	\[
		H(w_1,w_2,\ldots,w_\ell)\leftarrow
	\]
	in $P$, we introduce a replacement rule
	\[
		H_\ell(w_1 x_1,w_2 x_2,\ldots,w_\ell x_\ell,x_{\ell+1},\ldots, x_k)\leftarrow F'(x_1,x_2,\ldots,x_k)\quad
	\]
	to $P'$.
	Moreover, for each replacement rule of the form
	\[
		H(w_1,w_2,\ldots,w_n)\leftarrow K(x_1,x_2,\ldots,x_m)
	\]
	in $P$, we introduce a replacement rule
	\[
		H_n(w_1,w_2,\ldots,(w_n x_{m+1}x_{m+2}\cdots x_k), \varepsilon,\varepsilon,\ldots,\varepsilon)
		\leftarrow
		K_m(x_1,x_2,\ldots,x_k)
	\]
	to $P'$. This completes our construction of $M'$.

	\medskip

	\noindent
	\textit{Properties of $M'$}:
	We see that $M'$ generates exactly the same language as $M$ and that all derivations in the grammar $M'$ have the form
	\[
		S'(w) \leftarrow S_1(w,\varepsilon,\varepsilon,\ldots, \varepsilon) \leftarrow\cdots \leftarrow F'(\varepsilon,\varepsilon,\ldots,\varepsilon)\leftarrow.
	\]
	The nonterminal $S'$ only appears as the leftmost nonterminal of a derivation, and each nonterminal $A \in Q'\setminus\{S'\}$ has rank $k$.

	\medskip

	\noindent
	\underline{Step 2}:
	We construct an \RMCFG, which we denote as $M'' = (\Sigma,Q'',S',P'')$, from $M'$ by unfolding each replacement of $M'$ into a sequence of more basic replacements as follows.

	The nonterminals of the grammar $M''$ contain the nonterminals of $M'$, that is, $Q'\subseteq Q''$.
	In this stage of the proof, we show how to decompose the replacement rules of $M'$ into sequences of finitely many replacement rules, each of which of the form as described in \cref{def:normal-form}, or of the form
	\begin{equation}\label{eq:useless-rule}
		H(x_1,x_2,\ldots,x_k)\leftarrow K(x_1,x_2,\ldots,x_k).
	\end{equation}
	In Step 3 of this proof, we complete our construction by removing the replacement rules of the form (\ref{eq:useless-rule}).

	Firstly, our grammar $M''$ contains the replacement rules
	\[
		S'(x_1x_2 \cdots x_k)\leftarrow S_1(x_1,x_2,\ldots,x_k)
		\qquad
		\text{and}
		\qquad
		F'(x_1,x_2,\ldots,x_k) \leftarrow
	\]
	which are both a part of the grammar $M'$.

	Suppose that the grammar $M'$ has a replacement rule $p \in P'$ of the form
	\[
		p\colon A(w_1,w_2,\ldots,w_k)\leftarrow B(x_1,x_2,\ldots,x_k)
	\]
	where each
	\[
		w_i = t_{i} \, z_{i,1}\, u[z_{i,1}]\, z_{i,2}\, u[z_{i,2}] \cdots z_{i,\ell_i}\, u[z_{i,\ell_i}]
	\]
	for some $\ell_i\in \mathbb{N}$, each $z_{i,j}\in \{x_1,x_2,\ldots,x_k\}$, $t_{i}\in \Sigma^*$ and each $u[x_i] \in \Sigma^*$.

	We then introduce a nonterminal of the form
	\[
		C_{p,n,m}\in Q''
	\]
	for each $n\in \{1,2,\ldots,k\}$ and each $m \in \{0,1,2,\ldots,|u(x_{n})|\}$.
	Let each
	\[
		u[x_n] = u_{n,1} u_{n,2}\cdots u_{n,|u[x_n]|}\in \Sigma^*.
	\]
	We then introduce replacement rules to $M''$ as follows.
	\begin{itemize}
		\item $C_{p,1,0}(x_1,x_2,\ldots,x_k) \leftarrow B(x_1,x_2,\ldots,x_k)$;
		\item for each $n\in \{1,2,\ldots,k\}$ and each $m \in \bigl\{0,1,2,\ldots,\bigl|u[x_{n}]\bigr|-1\bigr\}$, we have
		      \[C_{p,n,m+1}(x_1,x_2,\ldots,x_{n-1},(x_n u_{n,m+1}),x_{n+1},\ldots,x_k)\leftarrow C_{p,n,m}(x_1,x_2,\ldots,x_k);\]
		\item for each $n\in \{1,2,\ldots,k-1\}$, we have
		      \[
			      C_{p,n+1,0}(x_1,x_2,\ldots,x_k)\leftarrow C_{p,n,|u[x_n]|}(x_1,x_2,\ldots,x_k).
		      \]
	\end{itemize}
	Notice then that we have the following production
	\[
		C_{p,k,|u[x_k]|}(x_1 u[x_1],x_2 u[x_2],\ldots,x_k u[x_k])
		\leftarrow
		\cdots
		\leftarrow
		B(x_1,x_2,\ldots,x_k)
	\]
	in the grammar $M''$.

	We now introduce a finite number of nonterminals of the form
	\[
		D_{p,\vec v} \in Q''
	\]
	where $\vec v = (v_1,v_2,\ldots,v_k)\in (\Sigma^*)^k$ is a vector with each \[v_i\in \{x_1,x_2,\ldots, x_k\}^*\qquad\text{and}\qquad v_1v_2\cdots v_k = x_1 x_2\cdots x_k.\]
	For the grammar $M'$, we then introduce all replacements of the form
	\[
		D_{p,(x_1,x_2,\ldots,x_k)}(x_1,x_2,\ldots,x_k)
		\leftarrow C_{p,k,|u[x_k]|}(x_1,x_2,\ldots,x_k);
	\]
	\begin{multline*}
		D_{p,(v_1,v_2,\ldots, v_{i-1}, v_{i}v_{i+1}, \varepsilon, v_{i+2},\ldots,v_k)}
		(x_1,x_2,\ldots,x_{i-1},x_{i}x_{i+1},\varepsilon,x_{i+2},\ldots,x_k)
		\\
		\leftarrow
		D_{p,(v_1,v_2,\ldots,v_k)} (x_1,x_2,\ldots,x_k)
	\end{multline*}
	for each $\vec v = (v_1,v_2,\ldots,v_k)$ as before, and each $i\in \{1,2,\ldots,k-1\}$; and
	\begin{multline*}
		D_{p,(v_1,v_2,\ldots, v_{i-1}, \varepsilon, v_{i}v_{i+1}, v_{i+2},\ldots,v_k)}
		(x_1,x_2,\ldots,x_{i-1},\varepsilon, x_{i}x_{i+1},x_{i+2},\ldots,x_k)
		\\
		\leftarrow
		D_{p,(v_1,v_2,\ldots,v_k)} (x_1,x_2,\ldots,x_k)
	\end{multline*}
	for each $\vec v = (v_1,v_2,\ldots,v_k)$ as before, and each $i\in \{1,2,\ldots,k-1\}$.

	Notice then that
	\[
		D_{p,\vec v}
		(v_1,v_2,\ldots,v_k)
		\leftarrow^*
		C_{p,k, |u[x_k]|}(x_1,x_2,\ldots,x_k)
	\]
	for each $D_{p,\vec v}$ as above.

	Let
	\[
		\vec Z =
		(
		(z_{1,1}z_{1,2}\cdots z_{2,\ell_1}),\
		(z_{2,1}z_{2,2}\cdots z_{2,\ell_2}),\
		\ldots,\
		(z_{k,1}z_{k,2}\cdots z_{k,\ell_k})
		).
	\]
	We then see that in the grammar $M''$, we have
	\[
		D_{p,\vec Z}(w_1',w_2',\ldots,w_k')
		\leftarrow\cdots\leftarrow B(x_1,x_2,\ldots,x_k)
	\]
	where each
	\[
		w_i'
		=
		z_{i,1}\, u(z_{i,1})\, z_{i,2}\, u(z_{i,2}) \cdots z_{i,\ell_i}\, u(z_{i,\ell_i}),
	\]
	that is, $w_i = t_{i} w_i'$ where $t_i\in \Sigma^*$.
	We then introduce a nonterminal of the form
	\[
		G_{p,n,m}\in Q''
	\]
	for each $n\in \{1,2,\ldots,k\}$ and each $m \in \{0,1,2,\ldots,|t_i|\}$.
	Let each
	\[
		t_i = t_{i,1} t_{i,2}\cdots t_{i,|t_i|}\in \Sigma^*.
	\]
	We then introduce replacement rules to $M''$ as follows.
	\begin{itemize}
		\item $G_{p,1,|t_1|}(x_1,x_2,\ldots,x_k)\leftarrow D_{p,\vec Z}(x_1,x_2,\ldots,x_k)$;
		\item for each $n\in \{1,2,\ldots,k\}$ and each $m \in \{1,2,\ldots,|t_n|\}$, we have
		      \[
			      G_{p,n,m-1}(x_1,x_2,\ldots,x_{n-1}, (t_{n,m} x_n) ,x_{n+1},\ldots,x_k)
			      \leftarrow
			      G_{p,n,m}(x_1,x_2,\ldots,x_k);
		      \]
		\item for each $n\in \{1,2,\ldots,k-1\}$, we have
		      \[
			      G_{p,n+1,|t_{n+1}|}(x_1,x_2,\ldots,x_k)
			      \leftarrow
			      G_{p,n,0}(x_1,x_2,\ldots,x_k);\text{ and}
		      \]
		\item $A(x_1,x_2,\ldots,x_k)\leftarrow G_{p,k,0}(x_1,x_2,\ldots,x_k)$.
	\end{itemize}
	We then see that in the grammar $M''$ we now have the derivation
	\[
		A(w_1,w_2,\ldots,w_k)
		\leftarrow \cdots\leftarrow
		B(x_1,x_2,\ldots, x_k)
	\]
	where each rule in the sequence is either of the form as described in \cref{def:normal-form}, or the form as in~(\ref{eq:useless-rule}).

	\medskip

	\noindent
	\textit{Properties of $M''$}:
	Observe that the grammar $M''$ generates exactly the same language as $M'$, and thus as $M$.
	Moreover, the only thing preventing $M''$ from being in normal form is that it contains rules of the form given in (\ref{eq:useless-rule}) which we remove in the following step.

	\medskip

	\noindent
	\underline{Step 3}:
	\nopagebreak

	\smallskip
	\nopagebreak
	\noindent
	Suppose that $M''$ contains a replacement rule as in (\ref{eq:useless-rule}), that is, a replacement rule of the form
	\[
		p \colon H(x_1,x_2,\ldots,x_k)\leftarrow K(x_1,x_2,\ldots,x_k).
	\]
	Then, we modify $M''$ by
	\begin{itemize}
		\item removing the nonterminal $K$ from $Q''$;
		\item removing the replacement rule $p$ from $P''$; and
		\item replacing each instance of the nonterminal $K$ with $H$ in each rule contained in $P''$.
	\end{itemize}
	Notice that after this modification, the number of rules of the form (\ref{eq:useless-rule}) is reduced by one, and that the grammar generates the same language.
	Moreover, we notice by induction that we can remove all such rules.
	We refer to the resulting grammar as $M''' = (\Sigma,Q''',S',P''')$.

	\medskip

	\noindent
	\underline{Conclusion}:
	\nopagebreak

	\smallskip
	\noindent
	From the properties of the grammar $M''$ in step 2, and the procedure described in step 3, we see that we may generate a grammar $M'''$ in normal form for the language $L$.
\end{proof}

\section{Proof of the main theorem}\label{sec:main-theorem}

Our aim in this section is to prove \cref{thm:main}.
From \cref{lem:taking-a-submonoid-of-EDT0L}, it suffices to prove this theorem for a specific generating set.
Thus, we introduce the following language.

\begin{definition}\label{def:wp-for-Z}
	Let $\Sigma = \{a,b\}$, and let $\MonoidHom\colon \Sigma^* \to \mathbb{Z}$ to be the monoid homomorphism defined such that $\MonoidHom(a) = 1$ and $\MonoidHom(b) = -1$.
	 Let $L\subseteq \Sigma^*$ as
	\[
		L
		=
		\{
		w\in \Sigma^*
		\mid
		\MonoidHom(w) = 0
		\}
	\]
   represent the word problem for $\mathbb{Z}$ with respect to the generating set $X=\{1,-1\}$.
\end{definition}

The remainder of this section is organised as follows.
We begin by introducing the constant $C$ (in \cref{lem:derivation_bound}) and the \emph{counterexample word} $\mathcal W$ (see \cref{def:counterexample} in \cref{sec:main-proof/w}).
We note here that the word $\mathcal W$ belongs to the language $L$, as in \cref{def:wp-for-Z}, as is used to show that $L$ is not EDT0L by contradiction.
\Cref{sec:main-proof/constants,sec:main-proof/functions} introduce several constants and functions on words which are used throughout the proofs of technical lemmas in this section.
These constants correspond to the lengths of certain factors of $\mathcal W$.
Our motivation for introducing these constants and functions is so that we can define and study the properties of \emph{$n$-decompositions}.
In \cref{sec:main-proof/decomp}, we introduce the important concept of an \emph{$n$-decomposition} (in \cref{def:decomposition}), we then spend the rest of this subsection proving technical lemmas which are used later in this section.
In \cref{sec:main-proof/maintaining-decomp}, we prove three propositions (\cref{prop:decomp1,prop:decomp2,prop:decomp3}) which are used to show that, in some way, having a decomposition would be maintained during the production of $\mathcal W$ of an \RMCFG\ for $L$ (from this we obtain a contradiction).
These three propositions are summarised and generalised in \cref{prop:decomp-full} in \cref{sec:main-proof/mirror}.
In \cref{sec:main-proof/result}, we use \cref{prop:decomp-full} to prove \cref{thm:main}, then obtain \cref{thm:cormain} as a corollary to \cref{thm:main}.

In the remainder of this paper $L$, $\Sigma$, $\MonoidHom$, $a$ and $b$ shall refer to the values as in \cref{def:wp-for-Z}.
To simplify the proofs in this section, we extend $\MonoidHom\colon \Sigma^* \to \mathbb{Z}$ to the domain of all $k$-tuples as follows.

\begin{definition}\label{def:monoid-hom-tuple}
	Let $\vec w = (w_1,w_2,\ldots,w_k) \in (\Sigma^*)^k$ be a tuple of words, then
	$
		\MonoidHom(\vec w)
		\coloneqq
		\MonoidHom(w_1 w_2 \cdots w_k).
	$
\end{definition}

From \cref{prop:wp-edt0l-fi,prop:equiv-langs,lem:normal-form}, we see that in order to prove \cref{thm:main}, it is sufficient to show that $L$ cannot be generated by an \RMCFG\ in normal form.
We have the following lemma and corollary which place restrictions on the structure of such a grammar.

\begin{lemma}\label{lem:derivation_bound_basic}
	If $L$ is generated by some \RMCFG\ $M = (\Sigma,Q,S,P)$,
	then for each nonterminal $H\in Q$, there exists a constant $C_H \in \mathbb{Z}$ such that if
	\[
		S(w)\leftarrow\cdots\leftarrow H(u_1, u_2,\ldots,u_k) \leftarrow\cdots\leftarrow
	\]
	is a derivation in the grammar $M$, then $\MonoidHom(u_1 u_2 \cdots u_k) = C_H$.
\end{lemma}

For example, we have $C_S=0$ where $S$ is the start non-terminal of the grammar.

\begin{proof}
	Let some nonterminal $H\in Q$ be chosen and suppose that there exists at least one derivation of the form
	\begin{equation}\label{eq:lem:derivation_bound_basic}
		S(g)\leftarrow\cdots\leftarrow H(v_1, v_2,\ldots,v_k) \leftarrow\cdots\leftarrow
	\end{equation}
	where $g, v_1,v_2,\ldots,v_k\in \Sigma^*$.
	To simplify the explanations in this proof, we write $\beta\in P^*$ for the sequence of replacement rules which generate the configuration $H(v_1, v_2,\ldots,v_k)$, and $\alpha\in P^*$ for the sequence of replacement rules which generates $S(g)$ from $H(v_1, v_2,\ldots,v_k)$.
	We then write
	\[
		S(g)\leftarrow^\alpha H(v_1, v_2,\ldots,v_k) \leftarrow^\beta
	\]
	to denote the sequence in (\ref{eq:lem:derivation_bound_basic}).

	From the definition of an \RMCFG, we then see that there exist words $m_0, m_1, m_2,\ldots,m_k \in \Sigma^*$ such that
	\[
		S(m_0 p_1 m_1 p_2 m_2 \cdots p_k m_k)
		\leftarrow^\alpha
		H(p_1,p_2,\ldots,p_k)
	\]
	for each $p_1,p_2,\ldots,p_k \in \Sigma^*$.
	In particular, if $g$ and $v_i$ are as in (\ref{eq:lem:derivation_bound_basic}), then $g = m_0 v_1 m_1 v_2 m_2\cdots v_k m_k$.

	Suppose that $u_1,u_2,\ldots,u_k\in \Sigma^*$ are words such that there exists some derivation
	\[
		S(w)\leftarrow\cdots\leftarrow H(u_1, u_2,\ldots,u_k) \leftarrow\cdots\leftarrow
	\]
	in the grammar $M$, and write $\gamma\in P^*$ for the sequence of replacements which generates $H(u_1, u_2,\ldots,u_k)$.
	That is,
	\[
		H(u_1, u_2,\ldots,u_k) \leftarrow^\gamma.
	\]
	We then see that
	\[
		S(m_0 u_1 m_1 u_2 m_2 \cdots u_k m_k)
		\leftarrow^\alpha
		H(u_1, u_2,\ldots,u_k) \leftarrow^\gamma
	\]
	is a derivation in the grammar $L$ and thus
	\[
		m_0 u_1 m_1 u_2 m_2 \cdots u_k m_k\in L.
	\]
	Thus, from the definition of the language $L$, we see that
	\[
		\MonoidHom(m_0 u_1 m_1 u_2 m_2 \cdots u_k m_k) = 0.
	\]
	Since $\MonoidHom$ is a monoid homomorphism onto an abelian group, we then see that
	\[
		\MonoidHom(u_1u_2\cdots u_k)+\MonoidHom(m_0 m_1 \cdots m_k)=0.
	\]
	We thus conclude by setting $C_H = -\MonoidHom(m_0 m_1 \cdots m_k)$.
\end{proof}

From this lemma, we have the following immediate corollary.

\begin{corollary}\label{lem:derivation_bound}
	Let $L$ be generated by some \RMCFG\ $M = (\Sigma,Q,S,P)$.
	Then there exists a constant $C=C_M\geq 1$ such that, if
	\[
		S(w)\leftarrow\cdots\leftarrow H(u_1, u_2,\ldots,u_k) \leftarrow\cdots\leftarrow
	\]
	is a derivation in the grammar $M$, then $|\MonoidHom(u_1 u_2\cdots u_k)| \leq C$.
\end{corollary}

\begin{proof}
	For each nonterminal $H\in Q$, let $C_H \in \mathbb{Z}$ be the constant as in \cref{lem:derivation_bound_basic}.
	We then see that our result follows with the constant
	$
		C_M
		=
		\max\{|C_H|\mid H\in Q\}
	$
	which is 
    well-defined since $Q$ is finite.
\end{proof}

\subsection{Counterexample word}\label{sec:main-proof/w}
In the remainder of this section, suppose for contradiction that there exists an \RMCFG\ $M = (\Sigma,Q,S,P)$  in normal form recognising $L$.
Moreover, write $C$ for the constant 
in \cref{lem:derivation_bound}, and write $k$ for the  rank of the nonterminals in $M$ that are not the starting nonterminal.

We show that such a grammar $M$ cannot exist by constructing a word $\mathcal{W}\in L$ which cannot be generated by any derivation in $M$. 

\begin{definition}\label{def:counterexample}
	Let $m = 24k+2C+5$. We define the word $\mathcal W \in\Sigma^*$ as
	\[
		\mathcal W = a^{m^{2k}}\, \mathcal T_{2k}\, a^{m^{2k}}
	\]
	where $\mathcal T_{2k}$ is defined recursively such that
	\[
		\mathcal T_{0} = b^2
		\quad
		\text{and}
		\quad
		\mathcal T_{n} = a^{m^{n}} (\mathcal T_{n-1})^{2m} a^{m^{n}}
	\]
	for each $n\in \{ 1,2,3,\ldots,2k \}$.
\end{definition}

\begin{remark}
\cref{fig:placeholder} illustrates the word $\mathcal W$ for the values $m=4$, $k=2$ with the letter $a$ represented by
an up-step $(1, 1)$ and $b$ by a down-step $(1,-1)$.
Note that \cref{def:counterexample} requires $m = 24k+2C+5$ so $m=4$ is not a realistic value, but larger values of $m$ would make the figure difficult to draw, and one can already start to see the general behaviour with these values.
Here is a calculation of the word in this case. We have
\[
    \mathcal T_0=b^2,\quad
    \mathcal T_1=a^4 b^{16} a^4,\quad
    \mathcal T_2=a^{16} \bigl(a^4 b^{16} a^4\bigr)^8 a^{16},\quad
    \mathcal T_3=a^{64} \bigl(a^{16}\bigl(a^4 b^{16} a^4\bigr)^8 a^{16}\bigr)^8 a^{64},
\]\[
    \mathcal T_4=a^{256} \bigl(a^{64} \bigl( a^{16} \bigl(a^4 b^{16} a^4\bigr)^8 a^{16}\bigr)^8 a^{64}\bigr)^8 a^{256},
\]
so
\[\mathcal W=a^{512} \bigl( a^{64}\bigl(a^{16}\bigl(a^4 b^{16} a^4\bigr)^8 a^{16}\bigr)^8a^{64}\bigr)^8 a^{512}.\]
Notice that in the word $\mathcal W$, every factor of $b$ is of length exactly 16.
\end{remark}

In the remainder of this section, $m$ denotes the constant in \cref{def:counterexample}.
We then have the following observation of the words $\mathcal{W}$ and $\mathcal{T}_n$ from \cref{def:counterexample}.

\begin{lemma}\label{lem:m-hom-recursive-definition}
	For each $n\in \{0,1,2,3,\ldots,2k\}$, we have $\MonoidHom(\mathcal{T}_n) = -2 m^n$ with
	\[
		|\mathcal{T}_n|_a = 2(2 m)^n - 2m^n
		\quad\text{and}\quad
		|\mathcal{T}_n|_b = 2(2 m)^n.
	\]
	Moreover, we have $\MonoidHom(\mathcal W)=0$ and thus $\mathcal W \in L$.
\end{lemma}
\begin{proof}
	From the recursive definition of $\mathcal T_n$ in \cref{def:counterexample}, we have
	\[
		\MonoidHom(\mathcal{T}_0) = -2
		\quad\text{and}\quad
		\MonoidHom(\mathcal{T}_{n}) = 2m\,\MonoidHom(\mathcal{T}_{n-1}) + 2 \, m^{n}
	\]
	for each $n \geq 1$.
	From this recurrence, we find that
	\[
		\MonoidHom(\mathcal{T}_n)
		=
		-2\,m^n
	\]
	for each $n\geq 0$.
	Thus, $\MonoidHom(\mathcal{T}_{2k}) = -2m^{2k}$.
	Again from \cref{def:counterexample}, we then have
	\[
		\MonoidHom(\mathcal{W})
		=
		\MonoidHom(\mathcal{T}_{2k})
		+ 2\, m^{2k} = 0.
	\]
	From the recursive definition of $\mathcal T_n$, we also see that
	\[
		|\mathcal{T}_0|_b = 2
		\quad\text{and}\quad
		|\mathcal{T}_{n}|_b = 2m\,|\mathcal{T}_{n-1}|_b.
	\]
	Thus, we see that $|\mathcal{T}_n|_b = 2(2m)^n$ for each $n\in \{1,2,3,\ldots,2k\}$; and we find that
	\[
		|\mathcal{T}_n|_a = |\mathcal{T}_n|_b+\MonoidHom(\mathcal T_n) = 2(2m)^n - 2m^n
	\]
	as desired.
\end{proof}

\subsection{Constants}\label{sec:main-proof/constants}
We now introduce some constants which are used in the proofs below.

\begin{definition}\label{def:misc-constants}
	For $n \in \{1,2,3,\ldots,2k\}$, we define
	\[
		\Lambda_n = m^{2k+1-n}
		\qquad\text{and}\qquad
		B_n = \frac{m^{2k+2-n}-m}{m-1}.
	\]
	Moreover, for each $n\in \{1,2,3,\ldots,2k\}$, we define
	$
		\sigma_n = 2 \Lambda_{n} + 2 B_{n}.
	$
\end{definition}

The constants are chosen so that they satisfy the following relations.

\begin{lemma}\label{lem:bound1}
	For each $n\in \{1,2,3,\ldots,2k\}$, we have
	$12 \Lambda_n > 6 B_n \geq 6 \Lambda_{n} > \sigma_n>0$.
\end{lemma}

\begin{proof}
	We observe that $\Lambda_n, B_n>0$ for each $n\in \{1,2,3,\ldots,2k\}$, and that
	\[
		B_n = \frac{m^{2k+2-n}-m}{m-1}
		=
		m + m^2 + \cdots + m^{2k+1-n}
		\qquad
		\text{and}
		\qquad
		\Lambda_n = m^{2k+1-n}.
	\]
	We thus see that $B_n \geq \Lambda_n$.
	Moreover, since $m\geq 2$, we see that $2\Lambda_n > B_n$.

	Then, since $\sigma_n = 2\Lambda_{n}+2 B_n$, we see that
	\[
		\sigma_n < 2\Lambda_n + 4\Lambda_n= 6\Lambda_n.
	\]
	Thus, we have our desired inequalities.
\end{proof}

\subsection{Functions on words and word weights}\label{sec:main-proof/functions}
In this section, we require a function to extract the $a^*$  factors  which appear as part of the prefix and suffix of a given word in $\Sigma^*$.
To accomplish this, we introduce the following function.

\begin{definition}\label{def:affix}
	We define $\Affix\colon \Sigma^* \times \{\leftSide,\rightSide\} \to \Sigma^*$ such that for each word $w\in \Sigma^*$, the word $\Affix(w,\leftSide)$ is the longest prefix of $w$ of the form $a^*$;
	and $\Affix(w,\rightSide)$ is the longest suffix of $w$ of the form $a^*$.
\end{definition}
For example, if $w=aaababbabbba$ then $\Affix(w,\leftSide)=aaa$ and $\Affix(w,\rightSide)=a$.
In the next lemma we compute $\Affix$ for some words of particular interest.

\begin{lemma}\label{lem:affix-of-T-recursive}
	For each $n\in \{1,2,3,\ldots,2k\}$ and each $s\in \{\leftSide, \rightSide\}$, we have
	\[
		|\Affix(\mathcal T_n, s)|
		=
		m + m^2 +\cdots + m^n = \frac{m^{n+1}-m}{m-1}=B_{2k+1-n}
	\]
	where $\mathcal{T}_n$ is as in \cref{def:counterexample}.
\end{lemma}

\begin{proof}
	Let $s\in \{\leftSide, \rightSide\}$, then from \cref{def:counterexample}, we see that
	\[
		|\Affix(\mathcal T_0, s)| = 0
		\quad\text{and}\quad
		|\Affix(\mathcal T_{n}, s)| = m^{n} + \Affix(\mathcal T_{n-1}, s)
	\]
	for each $n \geq 1$.
	From this recurrence, it follows immediately that
	\[
		|\Affix(\mathcal T_n, s)|
		=
		m + m^2 +\cdots + m^n = \frac{m^{n+1}-m}{m-1}
	\]
	for each $n\in \{1,2,3,\ldots,2k\}$.
\end{proof}

From this, we then have the following result on the affixes of $\mathcal W$.

\begin{corollary}\label{cor:affix-of-w}
	We have $|\Affix(\mathcal W,s)| = B_1 + m^{2k}$ for each $s\in \{\leftSide,\rightSide\}$.
\end{corollary}
\begin{proof}
	From the definition of $\mathcal W$ in \cref{def:counterexample}, we see that
	\[
		|\Affix(\mathcal W, s)| = |\Affix(\mathcal{T}_{2k},s)| + m^{2k}
	\]
	for each $s \in \{\leftSide,\rightSide\}$.
	Thus, from \cref{lem:affix-of-T-recursive}, we see that
	\[
		|\Affix(\mathcal W, s)| = B_1 + m^{2k}
	\]
	for each $s\in \{\leftSide,\rightSide\}$ as desired.
\end{proof}

We now define what it means for a word to be \emph{heavy} or \emph{light} as follows.
\begin{definition}\label{def:heavy}
	Let $n \in \{1,2,3,\ldots,2k\}$, we say that a word $u\in \Sigma^*$ is \emph{$n$-heavy} if it contains a factor of the form $a^{2\Lambda_{n}}$, otherwise we say that $u$ is \emph{$n$-light}.
	Given a tuple $\vec w = (w_1,w_2,\ldots,w_k)$, we define the set of indices $L_n(\vec w)\subseteq \{1,2,\ldots,k\}$ such that $i\in L_n(\vec w)$ if and only if $w_i$ is an $n$-light word.
\end{definition}

We then define the following function on $n$-heavy words.

\begin{definition}\label{def:segment}
	Let $u\in \Sigma^*$ be an $n$-heavy word, we then define
	\begin{itemize}
		\item $\Segment_n(u,\leftSide) = x$ to be the shortest prefix of $u$ such that $u = x a^{2\Lambda_{n}} u'$; and
		\item $\Segment_n(u,\rightSide) = x$ to be the shortest suffix of $u$ such that $u = u'a^{2\Lambda_{n}} x$.
	\end{itemize}
	In the above, it should be understood that $u' \in \Sigma^*$.
\end{definition}

We now define the functions $\Rem_n\colon \Sigma^*\to \Sigma^*$ as follows.

\begin{definition}\label{def:rem}
	For $n \in \{1,2,3,\ldots,2k\}$, we define $\Rem_n\colon \Sigma^*\to \Sigma^*$ such that
	\begin{itemize}
		\item if $u\in \Sigma^*$ is $n$-light, then $\Rem_n(u) = u$; otherwise
		\item if $u\in \Sigma^*$ is $n$-heavy, then $\Rem_n(u)=\Segment_n(u, \leftSide)\, \Segment(u,\rightSide)$.
	\end{itemize}
	Notice that in the second case, the word $u$ is $n$-heavy, and thus it is well-defined.
\end{definition}

We now define an additional function as follows.

\begin{definition}\label{def:strip}
	We define the function $\Strip\colon \Sigma^*\to\Sigma^*$ such that
	\[
		\Strip(u)
		\coloneqq
		\begin{cases}
			\varepsilon & \mathrm{if}\ u = \Affix(u,\leftSide) = \Affix(u,\rightSide),                           \\
			u'          & \mathrm{where}\ u=\Affix(u,\leftSide)\, u'\, \Affix(u,\rightSide)\ \mathrm{otherwise}.
		\end{cases}
	\]
	That is, $\Strip(u)$ is the word obtained by removing all leading and trailing instances of $a$ from $u$.
\end{definition}
For example, $\Strip(aaababbabbba)=babbabbb$.
We observe the following.

\begin{lemma}\label{lem:strip-bound}
	For each $n\in \{1,2,3,\ldots,2k\}$, we have
	\[
		\MonoidHom(\Strip(\mathcal{T}_n))
		=
		-2m^{n} - 2\frac{m^{n+1}-m}{m-1}.
	\]
	Moreover, for $n\in \{1,2,3,\ldots,2k\}$, let $w_n$ be the word
	\[
		w_n = \Affix(\mathcal T_{n}, \rightSide) \, \mathcal T_{n} \, \Affix(\mathcal T_{n}, \leftSide),
	\]
	then, for each factor $u$ of $w_n$, we have $|\MonoidHom(u)|\leq |\MonoidHom(\Strip(\mathcal T_n))|$.
\end{lemma}

\begin{proof}
	From the recursive formula given in \cref{def:counterexample}, we see that
	\[
		\MonoidHom(\Strip(\mathcal{T}_n))
		=
		2m \MonoidHom(\mathcal{T}_{n-1}) - \MonoidHom(\Affix(\mathcal{T}_{n-1}, \leftSide)) - \MonoidHom(\Affix(\mathcal{T}_{n-1},\rightSide))
	\]
	for each $n\geq 1$.
	Thus, from \cref{lem:m-hom-recursive-definition,lem:affix-of-T-recursive}, we see that
	\begin{align*}
		\MonoidHom(\Strip(\mathcal{T}_n))
		 & =
		2m (-2m^{n-1})
		-
		2\left(\frac{m^{n}-m}{m-1}\right)
		\\&=
		-4m^{n} - 2\frac{m^{n}-m}{m-1}
		\\&=
		-2m^{n} - 2\frac{m^{n+1}-m}{m-1},
	\end{align*}
	as desired.

	Considering the factors $u$ of $w_1 = a^{2m} b^{4m} a^{2m} = \Affix(\mathcal T_1, \rightSide)\, \mathcal T_1\, \Affix(\mathcal T_1, \leftSide)$, we see that
	\[
		-4m = \MonoidHom(b^{4m}) \leq \MonoidHom(u) \leq \MonoidHom(a^{2m}) = 2m.
	\]
	Thus, $|\MonoidHom(u)|\leq 4m = |\Delta(\Strip(\mathcal{T}_1))|$.

	Suppose, for induction, that $|\MonoidHom(v)|\leq |\Delta(\Strip(\mathcal{T}_{n-1}))|$ for every factor $v$ of the word \[w_{n-1} = \Affix(\mathcal T_{n-1}, \rightSide) \, \mathcal{T}_{n-1} \, \Affix(\mathcal T_{n-1}, \leftSide)\] for some value $n \in \{2,3,\ldots,2k\}$.
	Let $u$ be a factor of \[w_n = \Affix(\mathcal T_{n}, \rightSide) \, \mathcal{T}_{n} \, \Affix(\mathcal T_{n}, \leftSide).\]
	Notice from \cref{lem:affix-of-T-recursive,lem:m-hom-recursive-definition} that
	\begin{align*}
		\Delta(w_n) & = \MonoidHom(\mathcal{T}_n) + \MonoidHom(\Affix(\mathcal T_n, \leftSide)) + \MonoidHom(\Affix(\mathcal T_n, \rightSide)) \\
		            & = -2m^n + 2\left(\frac{m^{n+1}-m}{m-1}\right)                                                                            \\
		            & = 2\frac{m^n-m}{m-1}.
	\end{align*}
	From \cref{def:counterexample,lem:affix-of-T-recursive}, we see
	\begin{equation}\label{eq:strip-ineq-eq1}
		w_{n}
		=
		a^{B_{2k+1-n}}\, \mathcal{T}_{n} \, a^{B_{2k+1-n}}
		=
		a^{m^n+B_{2k+1-n}} \, (\mathcal{T}_{n-1})^{2m} \, a^{m^n+B_{2k+1-n}}
	\end{equation}
	and thus, the factor $u$ falls into one of the following 4 cases:

	\medskip

	\noindent
	\underline{Case 1}: The factor $u$ has the form
	\[
		u = a^i (\mathcal{T}_{n-1})^{2m} a^j
	\]
	where $i,j \in \{0,1,2,\ldots,m^n + B_{2k+1-n}\}$.
	In this case, we see that
	\[
		-4m^n=
		2m \MonoidHom(\mathcal T_{n-1})=
		\MonoidHom((\mathcal{T}_{n-1})^{2m})
		\leq
		\MonoidHom(u)
		\leq
		\MonoidHom(w_n)
		= 2\frac{m^n-m}{m-1} \leq 2m^n.
	\]
	From this, we then see that
	$
		|\MonoidHom(u)|\leq 4m^n \leq |\Delta(\Strip(\mathcal{T}_n))|
	$
	as desired.

	\noindent
	\underline{Case 2}: The factor $u$ has the form
	\[
		u = a^i (\mathcal{T}_{n-1})^{j} v
	\]
	where $i\in \{0,1,2,\ldots,m^n+B_{2k+1-n}\}$, $j \in \{0,1,2,\ldots,2m-1\}$ and $v$ is a factor of $\mathcal{T}_{n-1}$.

	From our inductive hypothesis we see that
	\[
		(2m-1)\MonoidHom(\mathcal{T}_{n-1}) - |\Delta(\Strip(\mathcal{T}_{n-1}))|
		\leq
		\MonoidHom(u)
		\leq
		\left(m^n+B_{2k+1-n}\right) +|\Delta(\Strip(\mathcal{T}_{n-1}))|.
	\]
	From \cref{lem:m-hom-recursive-definition,lem:affix-of-T-recursive}, we see that
	\begin{align*}
		\MonoidHom(u)
		 & \geq
		(2m-1)\MonoidHom(\mathcal{T}_{n-1}) - |\Strip(\mathcal{T}_{n-1})|
		\\
		 & \qquad\qquad=
		-(2m-1)2m^{n-1} - 2m^{n-1} - 2\frac{m^{n}-m}{m-1}
		\\
		 & \qquad\qquad
		=
		-4m^{n} - 2\frac{m^{n}-m}{m-1}
		\\
		 & \qquad\qquad
		=
		-2m^{n} - 2\frac{m^{n+1}-m}{m-1}
		=
		-|\Delta(\Strip(\mathcal{T}_n))|
	\end{align*}
	and
	\begin{align*}
		\MonoidHom(u)
		 & \leq\left(m^n+B_{2k+1-n}\right) + |\Delta(\Strip(\mathcal{T}_{n-1}))|
		\\&\qquad\qquad=
		\left(m^n+\frac{m^{n+1}-m}{m-1}\right) + \left(2m^{n-1} + 2\frac{m^{n}-m}{m-1}\right)
		\\&\qquad\qquad=
		\left(2m^n+\frac{m^{n}-m}{m-1}\right) + \left(2m^{n-1} + 2\frac{m^{n}-m}{m-1}\right)
		\\&\qquad\qquad\leq
		\left(2m^n+2m^{n-1}\right) + \left(2m^{n-1} + 2\frac{m^{n}-m}{m-1}\right)
		\\&\qquad\qquad\leq
		2m^n + 4m^{n-1} + 2\frac{m^{n}-m}{m-1}
		\\&\qquad\qquad\leq
		2m^n + 2m^{n} + 2\frac{m^{n}-m}{m-1}
		=
		2m^n + 2\frac{m^{n+1}-m}{m-1}= |\Delta(\Strip(\mathcal{T}_n))|
	\end{align*}
	where some of the above inequalities hold since $m\geq 2$.

	From this, we then see that
	$
		|\MonoidHom(u)| \leq |\Delta(\Strip(\mathcal{T}_n))|
	$
	as desired.

	\noindent
	\underline{Case 3}: The factor $u$ has the form
	\[
		u = v (\mathcal{T}_{n-1})^{j} a^i
	\]
	where $i\in \{0,1,2,\ldots,m^n+B_{2k+1-n}\}$, $j \in \{0,1,2,\ldots,2m-1\}$ and $v$ is a factor of $\mathcal{T}_{n-1}$.

	This case is symmetric to the previous case. In particular, we see that
	$
		|\MonoidHom(u)| \leq |\Delta(\Strip(\mathcal{T}_n))|
	$
	as desired.

	\noindent
	\underline{Case 4}: The factor $u$ has the form
	\[
		u = \alpha (\mathcal{T}_{n-1})^{i} \beta
	\]
	where $i\in \{0,1,2,\ldots,2m-2\}$, and $\alpha$ and $\beta$ are factors of $\mathcal T_{n-1}$.

	From our inductive hypothesis, we then see that
	\[
		-2|\Delta(\Strip(\mathcal T_{n-1}))| + (2m-2)\MonoidHom(\mathcal T_n)
		\leq
		\MonoidHom(u)
		\leq
		2|\Delta(\Strip(\mathcal T_{n-1}))|.
	\]
	Thus, from \cref{lem:m-hom-recursive-definition} and the fact that $m\geq 2$, we then see that
	\begin{align*}
		|\MonoidHom(u)|
		 & \leq
		2\left( 2m^{n-1} + 2\frac{m^n-m}{m-1} \right) + (2m-2)m^{n-1}
		\\&\qquad\qquad=
		2m^n + 2m^{n-1} + 4\frac{m^n-m}{m-1}
		\\&\qquad\qquad\leq
		2m^n + 2\frac{m^{n+1}-m}{m-1}
		=|\Delta(\Strip(\mathcal T_n))|
	\end{align*}
	where the last inequality holds since $m > 3$, in particular,
	\begin{align*}
		2m^n + 2m^{n-1} + 4\frac{m^n-m}{m-1}
		\leq
		2m^n + 2\frac{m^{n+1}-m}{m-1}
		 & \Longleftarrow
		m^{n-1} + 2\frac{m^n-m}{m-1}
		\leq
		\frac{m^{n+1}-m}{m-1}
		\\
		 & \Longleftarrow
		m^{n-1}(m-1) + 2(m^n-m)
		\leq
		m^{n+1}-m
		\\
		 & \Longleftarrow
		m^{n+1} -3m^n + m^{n-1}+m \geq 0
		\\
		 & \Longleftarrow
		m^{n+1} -3m^n \geq 0
		\\
		 & \Longleftarrow
		m^{n}(m- 3) \geq 0
		\Longleftarrow
		m\geq 3.
	\end{align*}
	Thus, $|\MonoidHom(u)| \leq |\Delta(\Strip(\mathcal T_n))|$ as desired.

	\noindent
	\underline{Conclusion}:
	\nopagebreak

	\smallskip
	\noindent
	In all cases, we have our desired bound on $|\MonoidHom(u)|$.
\end{proof}

We then have the following lemma which we use to characterise the $n$-light factors of $\mathcal{W}$.

\begin{lemma}\label{lem:bound-on-small-subwords}
	If $u$ is a factor of $\mathcal{W}$ that does not contain $a^{2B_{n}}$ as a factor, with $n\in \{1,2,3,\ldots,2k\}$, then $|\MonoidHom(u)|\leq \sigma_n$.
\end{lemma}

\begin{proof}
	For each $n\in \{1,2,3,\ldots,2k\}$, notice from the definition of $\mathcal{W}$ that
	\[
		\mathcal{W} \in a^* \mathcal{T}_{2k+1-n} a^* \mathcal{T}_{2k+1-n} a^* \mathcal{T}_{2k+1-n} a^*\cdots a^* \mathcal{T}_{2k+1-n} a^*.
	\]
	From \cref{lem:affix-of-T-recursive}, we see that
	\[
		\Affix(\mathcal{T}_{2k+1-n},\leftSide) = \Affix(\mathcal{T}_{2k+1-n},\rightSide) = a^{B_n}.
	\]
	From this we see that if $u$ does not contain $a^{2B_n}$ as a factor,
  then it must also be a factor of
	\[
		w=\Affix(\mathcal T_{2k+1-n}, \rightSide) \, \mathcal T_{2k+1-n} \, \Affix(\mathcal T_{2k+1-n}, \leftSide).
	\]
	Thus, from \cref{lem:strip-bound}, we see that
	\[
		|\MonoidHom(u)|
		\leq
		2m^{2k+1-n} + 2\frac{m^{2k+2-n}-m}{m-1}
		=
		\sigma_n,
	\]
	as desired.
\end{proof}

\begin{corollary}\label{cor:bound-on-rem}
	If $u$ is a factor of $\mathcal{W}$, then
	\[|\MonoidHom(\Rem_n(u))|\leq 2\sigma_n\]
	for all $n\in \{1,2,3,\ldots,2k\}$.
	If additionally $u$ is $n$-light, then $|\MonoidHom(\Rem_n(u))|\leq \sigma_n$.
\end{corollary}
\begin{proof}
	From the definition of $\Rem_n\colon \Sigma^*\to\Sigma^*$, we have two cases as follows.
	\begin{enumerate}[leftmargin=*]
		\item The word $u$ is $n$-light and thus $\Rem_n(u) = u$ does not contain $a^{2\Lambda_n}$ as a factor.
		      From \cref{lem:bound-on-small-subwords,lem:bound1}, we see that $|\MonoidHom(\Rem_n(u))|\leq \sigma_n$.
		\item The word $u$ is $n$-heavy and thus $\Rem_n(u) = \Segment_n(u,\leftSide)\,  \Segment_n(u,\rightSide)$ where both $\Segment_n(u,\leftSide)$ and $\Segment_n(u,\rightSide)$ are $n$-light and thus do not contain $a^{2\Lambda_n}$ as a factor.
		      From \cref{lem:bound-on-small-subwords,lem:bound1}, we then see that $|\MonoidHom(\Rem_n(u))|\leq 2\sigma_n$.
	\end{enumerate}
	We thus have our desired bounds.
\end{proof}

\subsection{Decompositions}\label{sec:main-proof/decomp}
In the following we define what it means for a $k$-tuple of words in $\Sigma^*$ to have a decomposition.
Later in this proof, we show that the property of having a decomposition is preserved, in some way, by the grammar $M$ for $L$.

\begin{definition}\label{def:decomposition}
	Let $n\in \{2,3,\ldots,2k\}$, then an \emph{$n$-decomposition} of $\vec w = (w_1,w_2,\ldots,w_k) \in (\Sigma^*)^k$ is a nonempty subset
	\[
		E \subseteq \{1,2,\ldots,k\} \times \{\leftSide, \rightSide\}
	\]
	with $|E|=n$ such that, for each $(i,s) \in E$, the word $w_i$ is $n$-heavy and
	\begin{equation}\label{eq:main-equation}
		|\Affix(w_i, s)|
		\geq
		\Lambda_{n-1}
		-
		(C - \MonoidHom(\vec w))
		-
		(2k - n - |L_n(\vec w)|)\sigma_n
		-
		\sum_{j=1}^k\MonoidHom( \Rem_n(w_j) ).
	\end{equation}
	Notice that $|L_n(\vec w)|$ is the number of $n$-light components in $\vec w$.
	We say that $\vec w$ \emph{has a decomposition} if it has an $n$-decomposition for some $n\in \{2,3,\ldots,2k\}$.
\end{definition}

\begin{remark}\label{remark:Compute2decompExplicitly}
    Notice that if $\vec w = (\mathcal W, \varepsilon,\varepsilon, \ldots,\varepsilon)$, then $E = \{(1,\leftSide), (1,\rightSide)\}$ is a $2$-decomposition of $\vec w$.
    In particular, we see that $w_1 = \mathcal W$ is $2$-heavy as it contains $a^{B_1+m^{2k}}$ as a prefix (see~\cref{cor:affix-of-w}) where $B_1+m^{2k}\geq 2\Lambda_1$ follows since we have $2\Lambda_1 = 2m^{2k}$ and $B_1 \geq m^{2k}$ from \cref{def:misc-constants}; and moreover, for each $(i,s)\in E$, we see that the left-hand side of (\ref{eq:main-equation}) evaluates to $B_1+m^{2k}$ (from \cref{cor:affix-of-w}) while the right-hand side of (\ref{eq:main-equation}) evaluates to
    \begin{align*}
        \Lambda_{1} - C - (2k - 2 - (k-1))\sigma_2 - 0
        =
        \Lambda_1 - C - (k-1)\sigma_2
        =
        \Lambda_1 - C - (k-1)(2\Lambda_2 + 2B_2)
    \end{align*}
    where the last equality follows from the definition of $\sigma_n$ in \cref{def:misc-constants}, thus (\ref{eq:main-equation}) is equivalent to
    \[
        B_1+m^{2k}
        \geq
        \Lambda_1 - C - (k-1)(2\Lambda_2 + 2B_2)
    \]
    which follows immediately as we have $B_1> \Lambda_1$ from \cref{lem:bound1}.
\end{remark}

To simplify the proofs of later lemmas in this section, we often assume without loss of generality that a given $n$-decomposition is maximal as follows.

\begin{definition}
	We say that an $n$-decomposition $E$ of $\vec w = (w_1,w_2,\ldots,w_k)$ is \emph{maximal} if there is no $n'$-decomposition $E'$ of $\vec w$ with $|E'|=n' >n= |E|$.
\end{definition}

To simplify notation, we define a set of tuples which we call \emph{$\mathcal W$-sentential tuples} as follows.

\begin{definition}\label{def:sentential-tuple}
	A \emph{$\mathcal W$-sentential tuple} is a $k$-tuple of the form $\vec w = (w_1,w_2,\ldots,w_k)\in (\Sigma^*)^k$ where
	\[
		\mathcal{W} = u_0 w_1 u_1 w_2 u_2 \cdots w_k u_k
	\]
	for some words $u_0,u_1,u_2,\ldots,u_k \in \Sigma^*$, and $|\MonoidHom(\vec w)|\leq C$.
  Thus, each word vector which appears in the derivation of $\mathcal W$, in the grammar $M$, is an example of a $\mathcal W$-sentential tuple.
\end{definition}

We have the following property of $n$-decompositions.

\begin{lemma}\label{lem:affix-bound-from-decomposition}
	Let $\vec w = (w_1,w_2,\ldots,w_k)$ be a $\mathcal W$-sentential tuple, and let $E$ be an $n$-decomposition of $\vec w$.
	Then, $|\Affix(w_i,s)| \geq 5\Lambda_{n}$ for every $(i,s)\in E$.
\end{lemma}

\begin{proof}
	Since $\vec w = (w_1,w_2,\ldots,w_k)$ is a $\mathcal W$-sentential tuple, we see that
	\[
		|\Affix(w_i, s)|
		\geq
		\Lambda_{n-1}
		-
		C
		-
		2k\sigma_n
		-
		\sum_{j=1}^k\MonoidHom( \Rem_n(w_j) ).
	\]
	From \cref{cor:bound-on-rem}, we then see that
	\[
		|\Affix(w_i, s)|
		\geq
		\Lambda_{n-1}
		-
		C
		-
		4k\sigma_n.
	\]
	From the definition of the constant $\Lambda_{n}$, we see that
	\begin{align*}
		|\Affix(w_i, s)|
		\geq
		\Lambda_{n-1}
		-
		C
		-
		4k\sigma_n
		 & =
		m\Lambda_{n}
		-
		C
		-
		4k\sigma_n
		\\
		 & =
		5\Lambda_n
		+ ( m - 5)\Lambda_n - C - 4k\sigma_n.
	\end{align*}
	From \cref{lem:bound1}, we then see that
	\[
		|\Affix(w_i, s)|
		\geq
		5\Lambda_n
		+ ( m - 5)\Lambda_n - C - 24k \Lambda_n.
	\]
	Thus,
	\[
		|\Affix(w_i, s)|
		\geq
		5\Lambda_n
		+ ( m - 5 -C - 24 k)\Lambda_n.
	\]
	From the value of $m$ given in \cref{def:counterexample}, we see that $|\Affix(w_i,s)|\geq 5\Lambda_n$.
\end{proof}

We now have the following lemma which we use to construct decompositions.

\begin{lemma}\label{lem:affix-bound-to-decomposition}
	Let $\vec w = (w_1,w_2,\ldots,w_k)$ be a $\mathcal W$-sentential tuple, and suppose that \[E \subseteq \{1,2,\ldots,k\}\times\{\leftSide,\rightSide\}\] is a set of size $|E|=n+1\geq 2$ such that $|\Affix(w_i,s)|\geq \Lambda_n$ for each $(i,s)\in E$.
	Then, the set $E$ is an $(n+1)$-decomposition of $\vec w$.
\end{lemma}
\begin{proof}
	From the definition of $m$ and $\Lambda_n$ in \cref{def:counterexample,def:misc-constants}, we see that $\Lambda_n \geq 2\Lambda_{n+1}$ and thus, for each $(i,s)\in E$, the word $w_i$ is $(n+1)$-heavy.
	From \cref{cor:bound-on-rem}, we see that
	\[|\MonoidHom(\Rem_{n+1}(w_i))|\leq \sigma_{n+1}\]
	for each $(n+1)$-light word $w_i$.
	From the definition of $\Rem_{n}\colon \Sigma^*\to\Sigma^*$, if $w_i$ is $(n+1)$-heavy, then
	\[
		|\MonoidHom(\Rem_{n+1}(w_i))|
		\leq
		|\MonoidHom(\Segment_{n+1}(w_i,\leftSide))|
		+
		|\MonoidHom(\Segment_{n+1}(w_i,\rightSide))|.
	\]
	Notice also that each $\Segment_{n+1}(w_i,\rightSide))$, as above, is $(n+1)$-light and thus
	\[
		|\MonoidHom(\Segment_{n+1}(w_i,\leftSide))|\leq \sigma_{n+1}
	\]
	from \cref{cor:bound-on-rem}.
	Moreover, if $(i,s)\in E$, then $\Segment_{n+1}(w_i,s)= \varepsilon$ and thus $|\MonoidHom(\Segment_{n+1}(w_i,s))|=0$.

	Let $H = \{1,2,\ldots,k\}\setminus L_{n+1}(\vec w)$ be the indices of all the $(n+1)$-heavy words in $\vec w$.
	From the previous paragraph and the triangle inequality, we then see that
	\[
		\left|
		\sum_{i\in H}
		\MonoidHom(\Rem_{n+1}(w_i))
		\right|
		\leq
		(2|H|-|E|)\sigma_{n+1}.
	\]
	Notice that $2|H|-|E|$ counts the sides of the $(n+1)$-heavy words which are not in $E$.

	From the above, we see that
	\begin{align*}
		\left|\sum_{i=1}^k \MonoidHom(\Rem_{n+1}(w_i)) \right|
		 & \leq
		(|L_{n+1}(\vec w)|+2|H|-|E|)\sigma_{n+1}
		\\
		 & \qquad=(|L_{n+1}(\vec w)| + 2(k-|L_{n+1}(\vec w)|) - (n+1))\sigma_{n+1}.
		\\
		 & \qquad=(2k - (n+1) - |L_{n+1}(\vec w)|)\sigma_{n+1}.
	\end{align*}
	From this bound, we then see that
	\[
		(2k - (n+1) - |L_{n+1}(\vec w)|)\sigma_{n+1}
		+
		\sum_{j=1}^k\MonoidHom( \Rem_{n+1}(w_j) )
		\geq 0.
	\]
	Thus, we see that
	\[
		\Lambda_n
		-
		(C - \MonoidHom(\vec w))
		-
		(2k - (n+1) - |L_{n+1}(\vec w)|)\sigma_{n+1}
		-
		\sum_{j=1}^k\MonoidHom( \Rem_{n+1}(w_j) )
		\leq
		\Lambda_n  - (C-\MonoidHom(\vec w)).
	\]
	Since $\vec w$ is a $\mathcal W$-sentential tuple, by definition, we have $|\MonoidHom(\vec w)|\leq C$, and thus
	\[
		\Lambda_n
		-
		(C - \MonoidHom(\vec w))
		-
		(2k - (n+1) - |L_{n+1}(\vec w)|)\sigma_{n+1}
		-
		\sum_{j=1}^k\MonoidHom( \Rem_{n+1}(w_j) )
		\leq
		\Lambda_n.
	\]

	From the assumptions in our lemma statement, we then see that
	\[
		|\Affix_{n+1}(w_i,s)|\geq \Lambda_n
		\geq \Lambda_n
		-
		(C - \MonoidHom(\vec w))
		-
		(2k - (n+1) - |L_{n+1}(\vec w)|)\sigma_{n+1}
		-
		\sum_{j=1}^k\MonoidHom( \Rem_{n+1}(w_j) )
	\]
	for each $(i,s)\in E$.
	Hence, $E$ is an $(n+1)$-decomposition as required.
\end{proof}

We have the following property of maximal decompositions.

\begin{lemma}\label{lem:maximal-form}
	Let $E$ be a maximal $n$-decomposition of a $\mathcal W$-sentential tuple $\vec w$.
	Then, for each $n$-heavy component $i\notin L_n(\vec w)$ and each $s \in \{\leftSide,\rightSide\}$ we have $\Segment_n(w_i,s)= \varepsilon$ if and only if $(i,s) \in E$.
\end{lemma}
\begin{proof}
	Let $E$ be a maximal $n$-decomposition as in the lemma statement.

	If $(i,s)\in E$, then from \cref{lem:affix-bound-from-decomposition} we see that $|\Affix(w_i,s)|\geq 5\Lambda_n$ and thus $\Segment_n(w_i,s)=\varepsilon$.
	Thus, if $(i,s)\in E$, then $\Segment_n(w_i,s) = \varepsilon$.
	All that remains is to show the converse of this statement as follows.

	Suppose for contradiction that $w_i$ is an $n$-heavy word with $\Segment_n(w_i,s)=\varepsilon$ and $(i,s)\notin E$.
	Then, from the definition of $\Segment_n\colon \Sigma^*\times \{\leftSide,\rightSide\}\to\Sigma^*$ in \cref{def:segment}, we see that $|\Affix(w_i,s)|\geq 2\Lambda_n$.
	Thus, from \cref{lem:affix-bound-to-decomposition,lem:affix-bound-from-decomposition}, we see that $E' = E\cup\{(i,s)\}$ is an $(n+1)$-decomposition of $\vec w$ which contradicts the maximality of $E$.
	Thus, if $\Segment_n(w_i,s) = \varepsilon$, then $(i,s)\in E$.
\end{proof}

We have the following lemma which is used in the proof of \cref{prop:decomp3}.

\begin{lemma}\label{lem:technical}
	Let $w\in \Sigma^*$ be an $n$-heavy factor of $\mathcal{W}$, and suppose that $w = uv$ with $u,v\in \Sigma^*$.
	Then, at least one of the following five conditions holds.
	\begin{enumerate}
		\item\label{lem:technical/1} $|\Affix(u,\rightSide)| \geq \Lambda_n$;
		\item\label{lem:technical/2} $|\Affix(v,\leftSide)| \geq \Lambda_n$;
		\item\label{lem:technical/3} $u$ and $v$ are both $n$-heavy, and $\MonoidHom(\Rem_n(u)) + \MonoidHom(\Rem_n(v)) \geq \MonoidHom(\Rem_n(w))-\sigma_n$;
		\item\label{lem:technical/4} $u$ is $n$-light, $v$ is $n$-heavy, and $\MonoidHom(\Rem_n(u)) + \MonoidHom(\Rem_n(v))  = \MonoidHom(\Rem_n(w))$;
		\item\label{lem:technical/5} $u$ is $n$-heavy, $v$ is $n$-light, and $\MonoidHom(\Rem_n(u)) + \MonoidHom(\Rem_n(v))  = \MonoidHom(\Rem_n(w))$.
	\end{enumerate}
\end{lemma}
\begin{proof}
	Suppose that $w$ is an $n$-heavy factor of $\mathcal{W}$ with $w = uv$ as in the lemma statement.
	We separate our proof into two parts as follows.

	\medskip

	\noindent
	\underline{Case 1}: both $u$ and $v$ are $n$-light.
	\nopagebreak

	\smallskip
	\noindent
	Here we see that $u$ and $v$ must be of the form
	\[
		u = \Segment_n(w,\leftSide) a^p
		\qquad
		\text{and}
		\qquad
		v = a^q \Segment(w,\rightSide),
	\]
	respectively, where $p+q \geq 2\Lambda_n$. Thus, we find that either $p\geq \Lambda_n$ or $q\geq \Lambda_n$, and thus we are either in case (\ref{lem:technical/1}) or (\ref{lem:technical/2}), respectively.

	\medskip

	\noindent
	\underline{Case 2}: at least one of $u$ or $v$ is $n$-heavy.
	\nopagebreak

	\smallskip
	\noindent
	Notice that if
	\[
		|\Affix(u,\rightSide)|\geq \Lambda_n
		\qquad
		\text{or}
		\qquad
		|\Affix(v,\leftSide)| \geq \Lambda_n,
	\]
	then we are in case (\ref{lem:technical/1}) or (\ref{lem:technical/2}), respectively.
	Thus, in the remainder of this proof, we assume that both
	\[
		|\Affix(u,\rightSide)|< \Lambda_n
		\qquad
		\text{and}
		\qquad
		|\Affix(v,\leftSide)| < \Lambda_n,
	\]
	and thus the factor 
	\[
		w'=\Segment_n(u,\rightSide)\Segment_n(v,\leftSide)
	\]
	of $w$ does not contain $a^{2\Lambda_n}$ as a factor, that is, $w'$ is $n$-light.

	\medskip

	\noindent
	\underline{Case 2.1}: both $u$ and $v$ are $n$-heavy.
	\nopagebreak

	\smallskip
	\noindent
	From \cref{lem:bound1,lem:bound-on-small-subwords}, we then see that
	\[
		\MonoidHom(\Segment_n(u,\rightSide)) + \MonoidHom(\Segment_n(u,\leftSide))
		=
		\MonoidHom(w')
		\geq -\sigma_n.
	\]
	From the definition of $\Rem_n\colon \Sigma^*\to \Sigma^*$, we see that
	\begin{align*}
		\MonoidHom(\Rem_n(u)) + \MonoidHom(\Rem_n(v))
		 & =
		\MonoidHom(\Segment_n(u,\leftSide)) + \MonoidHom(\Segment_n(v,\rightSide))
		+
		\MonoidHom(\Segment_n(u,\rightSide)) + \MonoidHom(\Segment_n(v,\leftSide))
		\\
		 & =
		\MonoidHom(\Segment_n(w,\leftSide)) + \MonoidHom(\Segment_n(w,\rightSide))
		+
		\MonoidHom(\Segment_n(u,\rightSide)) + \MonoidHom(\Segment_n(v,\leftSide))
		\\
		 & = \MonoidHom(\Rem_n(w)) + \MonoidHom(w')
		\geq \MonoidHom(\Rem_n(w)) -\sigma_n.
	\end{align*}
	Thus, we are in case (\ref{lem:technical/3}) of the lemma.

	\medskip

	\noindent
	\underline{Case 2.2}: $u$ is $n$-light, and $v$ is $n$-heavy.
	\nopagebreak

	\smallskip
	\noindent
	The only way for $u$ to be $n$-light is if it is either a factor
    of $\Segment_n(w,\leftSide)$ or a word of the form $\Segment_n(w,\leftSide)a^p$ where $p< 2\Lambda_n$. We consider these cases as follows.
	\begin{itemize}
		\item If $u$ is a factor of $\Segment_n(w,\leftSide)$, then we see that
		      \begin{align*}
			      \Rem_n(w) & = \Segment_n(w,\leftSide)\,\Segment_n(w,\rightSide)                                \\
			                & = u\,\Segment_n(v,\leftSide)\,\Segment_n(v,\rightSide) = \Rem_n(u) \, \Rem_n(v).
		      \end{align*}
		      Thus, \[\MonoidHom(\Rem_n(u))+\MonoidHom(\Rem_n(v)) = \MonoidHom(\Rem_n(w))\] and we are in case (\ref{lem:technical/4}) of the lemma.

        \item If $u = \Segment_n(w,\leftSide)a^p$ where $p< 2\Lambda_n$, then $v = a^q \, v'$ with $p+q \geq 2\Lambda_n$. So either $p\geq \Lambda_n$ or $q\geq \Lambda_n$, which correspond to cases (\ref{lem:technical/1}) and (\ref{lem:technical/2}) of the lemma.
	\end{itemize}

	\medskip

	\noindent
	\underline{Case 2.3}: $u$ is $n$-heavy, and $v$ is $n$-light.
	\nopagebreak

	\smallskip
	\noindent
	The proof of this case is symmetric to Case~2.2.

	\medskip

	\noindent
	\underline{Conclusion}: We see that in all cases, our words fall into one of the five cases listed in the lemma.
	Moreover, we see that these cases cover all possible situations.
\end{proof}

\subsection{Maintaining a decomposition}\label{sec:main-proof/maintaining-decomp}
In this subsection we show that, for $\mathcal W$-sentential tuples, the property of having a decomposition is maintained by the operations known as left letter deletion and left split (both of which are the reverse of some of the operations presented in \cref{def:normal-form}).

\begin{proposition}\label{prop:decomp1}
	Let $\vec w = (w_1,w_2,\ldots,w_k)$ and $\vec v = (v_1,v_2,\ldots,v_k)$ be $\mathcal W$-sentential tuples as in \cref{def:sentential-tuple}, and suppose that $x\in \{1,2,\ldots,k\}$ with
	\[
		(w_1,w_2,\ldots,w_k) = (v_1,v_2,\ldots,v_{x-1},a v_{x}, v_{x+1},\ldots,v_k).
	\]
	If $\vec w$ has an $n$-decomposition, then $\vec v$ has an $n'$-decomposition where $n'\geq n$.
	That is, given a $\mathcal W$-sentential tuple with an $n$-decomposition, if we still have a $\mathcal W$-sentential tuple after deleting an instance of the letter $a$ from the left-hand side of some component, then there exists an $n'$-decomposition, with $n'\geq n$, for the new tuple.
	(Recall that if $\vec w$ is a $\mathcal W$-sentential tuple, then $|\MonoidHom(\vec w)|\leq C$. Thus, we cannot directly use this process to remove sequences of $a$'s of arbitrarily large length and still obtain a decomposition.)
\end{proposition}

\begin{proof}
	Without loss of generality, we assume that $E$ is a maximal $n$-decomposition for the sentential tuple $\vec w$.
	We note here that $\MonoidHom(\vec v) = \MonoidHom(\vec w)-1$, that $\Rem_n(v_i) = \Rem_n(w_i)$ for each $i\neq x$; and that $\Affix(v_i,s) = \Affix(w_i,s)$ for each $i\neq x$ and $s\in \{\leftSide, \rightSide\}$.
	The remainder of this proof is separated into three cases as follows.

	\medskip

	\noindent
	\underline{Case 1}: $(x,\leftSide) \in E$.
	\nopagebreak

	\smallskip
	\noindent
	From \cref{lem:affix-bound-from-decomposition}, we see that
	\[
		|\Affix(v_x, \leftSide)| = |\Affix(w_x, \leftSide)| - 1 \geq 5\Lambda_n - 1\geq 2\Lambda_n.
	\]
	Further, we see that
	\[
		|\Affix(v_x,\rightSide)| \in \{ |\Affix(w_x,\rightSide)|,\, |\Affix(w_x,\rightSide)|-1 \},
	\]
	where the $|\Affix(w_x,\rightSide)|-1$ corresponds to the case where $w_x = a^{|w_x|}$.

	Thus,
	\[
		|\Affix(v_i, s)| \geq |\Affix(w_i,s)|-1
	\]
	for each $i\in \{1,2,\ldots,k\}$ and $s\in \{\leftSide,\rightSide\}$.

	We see that the words $v_x$ and $w_x$ are both $n$-heavy, and we see that
	\[
		L_n(\vec v) = L_n(\vec w)
		\qquad
		\text{and}
		\qquad
		\Rem_n(v_x) = \Rem_n(w_x).
	\]
	We are now ready to prove this case as follows.

	From the above observations, we see that
	\begin{align*}
		|\Affix(v_i,s)|
		 & \geq
		|\Affix(w_i,s)|-1
		\\&\geq
		\Lambda_{n-1}
		-
		(C - (\MonoidHom(\vec w) - 1))
		-
		(2k - n - |L_n(\vec w)|)\sigma_n
		-
		\sum_{j=1}^k\MonoidHom( \Rem_n(w_j) )
		\\&=
		\Lambda_{n-1}
		-
		(C - \MonoidHom(\vec v))
		-
		(2k - n - |L_n(\vec v)|)\sigma_n
		-
		\sum_{j=1}^k\MonoidHom( \Rem_n(v_j) )
	\end{align*}
	for each $(i,s)\in E$.
	Hence, we find that $E$ is an $n$-decomposition for $\vec v$.

	\medskip

	\noindent
	\underline{Case 2}: $(x,\leftSide) \notin E$ and $w_x$ is $n$-heavy.
	\nopagebreak

	\smallskip
	\noindent
	From \cref{lem:maximal-form}, we see that
	$
		\Segment_n(w_x, \leftSide) \neq \varepsilon
	$
	and thus removing an $a$ from the left-hand side of $w_x$ does not affect its status as being $n$-heavy, and furthermore does not affect $\Affix(w_x,\rightSide)$.
	That is, we have
	\[
		\Affix(v_x, \rightSide) = \Affix(w_x,\rightSide).
	\]
	Moreover, from the definition of $\Rem_n\colon \Sigma^*\to\Sigma^*$ on $n$-heavy words, we see that
	\[
		\MonoidHom(\Rem_n(v_x)) = \MonoidHom(\Rem_n(w_x))-1.
	\]

	From the above observations, we see that
	\begin{align*}
		|\Affix(v_i,s)|
		 & \geq
		|\Affix(w_i,s)|
		\\&\geq
		\Lambda_{n-1}
		-
		(C - (\MonoidHom(\vec w) - 1))
		-
		(2k - n - |L_n(\vec w)|)\sigma_n
		-
		\sum_{j=1}^{x-1}\MonoidHom( \Rem_n(w_j) )
		\\&
		\phantom{\qquad\qquad\qquad\geq
			\Lambda_{n-1}
			-
			(C - (}
		-(\MonoidHom( \Rem_n(w_x) )-1)
		-\sum_{j={x+1}}^k\MonoidHom( \Rem_n(w_j) )
		\\&=
		\Lambda_{n-1}
		-
		(C - \MonoidHom(\vec v))
		-
		(2k - n - |L_n(\vec v)|)\sigma_n
		-
		\sum_{j=1}^{k}\MonoidHom( \Rem_n(v_j) )
	\end{align*}
	for each $(i,s)\in E$.
	Hence, we find that $E$ is an $n$-decomposition for $\vec v$.

	\medskip

	\noindent
	\underline{Case 3}: $w_x$ is $n$-light.
	\nopagebreak

	\smallskip
	\noindent
	Notice here that $(x,\rightSide)\notin E$ since $w_x$ and thus $v_x$ are $n$-light.

	We then see that
	$
		a\,\Rem_n(v_x) = \Rem_n(w_x),
	$
	and hence,
	\[
		\MonoidHom(\Rem_n(v_x)) = \MonoidHom(\Rem_n(w_x))-1.
	\]

	From the above observations, we see that
	\begin{align*}
		|\Affix(v_i,s)|
		 & \geq
		|\Affix(w_i,s)|
		\\&\geq
		\Lambda_{n-1}
		-
		(C - (\MonoidHom(\vec w) - 1))
		-
		(2k - n - |L_n(\vec w)|)\sigma_n
		-
		\sum_{j=1}^{x-1}\MonoidHom( \Rem_n(w_j) )
		\\&
		\phantom{\qquad\qquad\qquad\geq
			\Lambda_{n-1}
			-
			(C - (}
		-(\MonoidHom( \Rem_n(w_x) )-1)
		-\sum_{j={x+1}}^k\MonoidHom( \Rem_n(w_j) )
		\\&=
		\Lambda_{n-1}
		-
		(C - \MonoidHom(\vec v))
		-
		(2k - n - |L_n(\vec v)|)\sigma_n
		-
		\sum_{j=1}^{k}\MonoidHom( \Rem_n(v_j) )
	\end{align*}
	for each $(i,s)\in E$.
	Hence, we find that $E$ is an $n$-decomposition for $\vec v$.

	\medskip

	\noindent
	\underline{Conclusion}:
	\nopagebreak

	\smallskip
	\noindent
	We see that in all cases $E$ is an $n$-decomposition for $\vec v$.
\end{proof}

\begin{proposition}\label{prop:decomp2}
	Let $\vec w = (w_1,w_2,\ldots,w_k)$ and $\vec v = (v_1,v_2,\ldots,v_k)$ be $\mathcal W$-sentential tuples as in \cref{def:sentential-tuple}, and suppose that $x\in \{1,2,\ldots,k\}$ with
	\[
		(w_1,w_2,\ldots,w_k) = (v_1,v_2,\ldots,v_{x-1},b v_{x}, v_{x+1},\ldots,v_k).
	\]
	If $\vec w$ has an $n$-decomposition, then $\vec v$ has an $n$-decomposition.
	That is, given a $\mathcal W$-sentential tuple with an $n$-decomposition, if we still have a $\mathcal W$-sentential tuple after deleting an instance of the letter $b$ from the left-hand side of some component, then there exists an $n$-decomposition for the new tuple.
\end{proposition}

\begin{proof}
	Let $E$ be an $n$-decomposition for $\vec w$.
	In this proof, we show that $E$ is also an $n$-decomposition for $\vec v$.
	Notice that since the leftmost letter of $w_x$ is $b$, we have $(x,\leftSide)\notin E$.
	In addition, removing an occurrence of $b$ from the left of $w_x$ does not change its $n$-heavy/$n$-light status, and does not affect $\Affix(w_x,\rightSide)$.
	Thus,
	\[
		L_n(\vec v) = L_n(\vec w)
		\qquad
		\text{and}
		\qquad
		\Affix(v_x,\rightSide) = \Affix(w_x,\rightSide).
	\]
	Notice that $\MonoidHom(\vec v) = \MonoidHom(\vec w)+1$,
	that $\Rem_n(v_i) = \Rem_n(w_i)$ for each $i\neq x$, and $\Affix(v_i,s) = \Affix(w_i,s)$ for each $i\neq x$ and $s\in \{\leftSide, \rightSide\}$.
	We also see that
	$
		b\, \Rem_n(v_x) = \Rem_n(w_x)
	$
	and thus
	\[
		\MonoidHom(\Rem_n(v_x)) = \MonoidHom(\Rem_n(w_x))+1.
	\]

	From our observations as above, we see that
	\begin{align*}
		|\Affix(v_i,s)|
		 & \geq
		|\Affix(w_i,s)|
		\\&\geq
		\Lambda_{n-1}
		-
		(C - (\MonoidHom(\vec w) + 1))
		-
		(2k - n - |L_n(\vec w)|)\sigma_n
		-
		\sum_{j=1}^{x-1}\MonoidHom( \Rem_n(w_j) )
		\\&
		\phantom{\qquad\qquad\geq
			\Lambda_{n-1}}
		-(\MonoidHom( \Rem_n(w_x) )+1)
		-\sum_{j={x+1}}^k\MonoidHom( \Rem_n(w_j) )
		\\&=
		\Lambda_{n-1}
		-
		(C - \MonoidHom(\vec v))
		-
		(2k - n - |L_n(\vec v)|)\sigma_n
		-
		\sum_{j=1}^{k}\MonoidHom( \Rem_n(v_j) )
	\end{align*}
	for each $(i,s)\in E$.
	Hence, we see that $E$ is an $n$-decomposition for $\vec v$.
\end{proof}

\begin{proposition}\label{prop:decomp3}
	Let $\vec w = (w_1,w_2,\ldots,w_k)$ and $\vec v = (v_1,v_2,\ldots,v_k)$ be $\mathcal W$-sentential tuples as in \cref{def:sentential-tuple}, and suppose that $x\in \{1,2,\ldots,k-1\}$ with $w_x = \varepsilon$, $w_{x+1} = v_{x}v_{x+1}$, and
	\[
		\vec w = (w_1,w_2,\ldots,w_k) = (v_1,v_2,\ldots,v_{x-1}, \varepsilon, v_x v_{x+1}, v_{x+2},\ldots,v_k).
	\]
	If $\vec w$ has an $n$-decomposition, then $\vec v$ has an $n'$-decomposition where $n'\geq n$.
	That is, given a $\mathcal W$-sentential tuple $\vec w$ with an $n$-decomposition, if we split one of its components into two adjacent components to obtain another $\mathcal W$-sentential tuple, then the resulting tuple will also have an $n'$-decomposition with $n'\geq n$.

	(Notice that if $\vec w$ and $\vec v$ are different vectors, then $\vec v$ will have one fewer component equal to $\varepsilon$ then $\vec w$. Thus, this proposition cannot be directly repeated arbitrarily many times.)
\end{proposition}

\begin{proof}
	Without loss of generality, we assume that $E$ is a maximal $n$-decomposition for the sentential tuple $\vec w$.
	We note here that $\MonoidHom(\vec v) = \MonoidHom(\vec w)$, and $v_i = w_i$ for each $i\notin\{x,x+1\}$.

	The remainder of this proof is separated into cases based on whether the word $w_{x+1}$ is $n$-heavy, and the memberships of its sides to the set $E$.

	\medskip

	\noindent
	\underline{Case 1}: The word $w_{x+1}$ is $n$-light.
	\nopagebreak

	\smallskip
	\noindent
	Since $w_{x+1} = v_x v_{x+1}$, we then see that both $v_x $ and $v_{x+1}$ are also $n$-light.
	Hence,
	\[
		L_n(\vec v) = L_n(\vec w),
		\qquad
		\MonoidHom(w_{x+1}) = \MonoidHom(v_x) + \MonoidHom(v_{x+1})
	\]
	and thus
	\begin{align*}
		\MonoidHom(\Rem_n(w_x))+\MonoidHom(\Rem_n(w_{x+1}))
		 & =
		\MonoidHom(\Rem_n(w_{x+1}))
		\\&=
		\MonoidHom(\Rem_n(v_x))+\MonoidHom(\Rem_n(v_{x+1})).
	\end{align*}
	Since $w_{x+1}$ is $n$-light and $w_x = \varepsilon$, it then follows that
	\[
		(x,\leftSide),(x,\rightSide),(x+1,\leftSide),(x+1,\rightSide) \notin E.
	\]
	Thus, we have $v_i = w_i$ if $(i,s)\in E$ for some $s\in \{\leftSide, \rightSide\}$.
	Hence, we find that
	\begin{align*}
		|\Affix(v_i,s)|
		 & =
		|\Affix(w_i,s)|
		\\&\geq
		\Lambda_{n-1}
		-
		(C - \MonoidHom(\vec w))
		-
		(2k - n - |L_n(\vec w)|)\sigma_n
		-
		\sum_{j=1}^{k}\MonoidHom( \Rem_n(w_j) )
		\\&=
		\Lambda_{n-1}
		-
		(C - \MonoidHom(\vec v))
		-
		(2k - n - |L_n(\vec v)|)\sigma_n
		-
		\sum_{j=1}^{k}\MonoidHom( \Rem_n(v_j) )
	\end{align*}
	for each $(i,s) \in E$. Thus, $E$ is an $n$-decomposition for $\vec v$.

	\medskip

	\noindent
	\underline{Case 2}: The word $w_{x+1}$ is $n$-heavy with $(x+1,\leftSide),(x+1,\rightSide) \notin E$.
	\nopagebreak

	\smallskip
	\noindent
	Notice that since $w_x = \varepsilon$, we have that
	\[
		(x,\leftSide),(x,\rightSide),(x+1,\leftSide),(x+1,\rightSide) \notin E.
	\]
	Thus, we have $v_i = w_i$ if $(i,s)\in E$ for some $s\in \{\leftSide, \rightSide\}$.

	From \cref{lem:technical}, we are in one of the following five cases.
	\begin{enumerate}[leftmargin=*]
		\item We have $|\Affix(v_x,\rightSide)| \geq \Lambda_n$.

		      From \cref{lem:affix-bound-from-decomposition,lem:affix-bound-to-decomposition}, we see that
		      $
			      E' = E\cup\{ (x,\rightSide) \}
		      $
		      is an $(n+1)$-decomposition for $\vec v$.

		\item We have $|\Affix(v_{x+1},\leftSide)| \geq \Lambda_n$.

		      From \cref{lem:affix-bound-from-decomposition,lem:affix-bound-to-decomposition}, we see that
		      $
			      E' = E\cup\{ (x+1,\leftSide) \}
		      $
		      is an $(n+1)$-decomposition for $\vec v$.

		\item We have that $v_x$ and $v_{x+1}$ are both $n$-heavy, and
		      \begin{equation}\label{eq:case2.3}
			      \MonoidHom(\Rem_n(v_x)) + \MonoidHom(\Rem_n(v_{x+1})) \geq \MonoidHom(\Rem_n(w_{x+1}))-\sigma_n.
		      \end{equation}
		      It then follows that $L_n(\vec v) = L_n(\vec w)\setminus \{x\}$, in particular, $|L_n(\vec v)| = |L_n(\vec w)|-1$.

		      Recall that $w_i = u_i$ if $(i,s)\in E$ for some $s\in \{\leftSide,\rightSide\}$.
		      Hence, for each $(i,s)\in E$, we have
		      \begin{align*}
			      |\Affix(v_i,s)|
			       & =
			      |\Affix(w_i,s)|
			      \\&\geq
			      \Lambda_{n-1}
			      -
			      (C - \MonoidHom(\vec w))
			      -
			      (2k - n - |L_n(\vec w)|)\sigma_n
			      -
			      \sum_{j=1}^{k}\MonoidHom( \Rem_n(w_j) )
			      \\&=
			      \Lambda_{n-1}
			      -
			      (C - \MonoidHom(\vec w))
			      -
			      (2k - n - (|L_n(\vec w)|-1))\sigma_n
			      +\sigma_n
			      -
			      \sum_{j=1}^{k}\MonoidHom( \Rem_n(w_j) )
			      \\&\geq
			      \Lambda_{n-1}
			      -
			      (C - \MonoidHom(\vec v))
			      -
			      (2k - n - |L_n(\vec v)|)\sigma_n
			      -
			      \sum_{j=1}^{k}\MonoidHom( \Rem_n(v_j) ).
		      \end{align*}
		      Notice that the last inequality, as above, follows from our case assumption in (\ref{eq:case2.3}).

		      We then see that $E$ is an $n$-decomposition for $\vec v$.

		\item We have that $v_x$ is $n$-light, $v_{x+1}$ is $n$-heavy, and \[\MonoidHom(\Rem_n(v_x)) + \MonoidHom(\Rem_n(v_{x+1})) = \MonoidHom(\Rem_n(w_{x+1})).\]

		      Notice that we have $L_n(\vec v) = L_n(\vec w)$, in particular, $|L_n(\vec v)| = |L_n(\vec w)|$.

		      Recall that $w_i = u_i$ if $(i,s)\in E$ for some $s\in \{\leftSide,\rightSide\}$.
		      Hence, for each $(i,s)\in E$, we have
		      \begin{align*}
			      |\Affix(v_i,s)|
			       & =
			      |\Affix(w_i,s)|
			      \\&\geq
			      \Lambda_{n-1}
			      -
			      (C - \MonoidHom(\vec w))
			      -
			      (2k - n - |L_n(\vec w)|)\sigma_n
			      -
			      \sum_{j=1}^{k}\MonoidHom( \Rem_n(w_j) )
			      \\&=
			      \Lambda_{n-1}
			      -
			      (C - \MonoidHom(\vec v))
			      -
			      (2k - n - |L_n(\vec v)|)\sigma_n
			      -
			      \sum_{j=1}^{k}\MonoidHom( \Rem_n(v_j) ).
		      \end{align*}
		      Notice that this follows from our case assumption since $\MonoidHom(\Rem_n(w_{x})) = \MonoidHom(\Rem_n(\varepsilon)) = 0$.

		      Thus, we see that $E$ is an $n$-decomposition for $\vec v$.

		\item We have that $v_x$ is $n$-heavy, $v_{x+1}$ is $n$-light, and \[\MonoidHom(\Rem_n(v_x)) + \MonoidHom(\Rem_n(v_{x+1})) = \MonoidHom(\Rem_n(w_{x+1})).\]

		      We have $L_n(\vec v) = (L_n(\vec w)\setminus\{x\})\cup\{x+1\}$, in particular, $|L_n(\vec v)| = |L_n(\vec w)|$.

		      Recall that $w_i = u_i$ if $(i,s)\in E$ for some $s\in \{\leftSide,\rightSide\}$.
		      Hence, for each $(i,s)\in E$, we have
		      \begin{align*}
			      |\Affix(v_i,s)|
			       & =
			      |\Affix(w_i,s)|
			      \\&\geq
			      \Lambda_{n-1}
			      -
			      (C - \MonoidHom(\vec w))
			      -
			      (2k - n - |L_n(\vec w)|)\sigma_n
			      -
			      \sum_{j=1}^{k}\MonoidHom( \Rem_n(w_j) )
			      \\&=
			      \Lambda_{n-1}
			      -
			      (C - \MonoidHom(\vec v))
			      -
			      (2k - n - |L_n(\vec v)|)\sigma_n
			      -
			      \sum_{j=1}^{k}\MonoidHom( \Rem_n(v_j) ).
		      \end{align*}
		      Notice that this follows from our case assumption since $\MonoidHom(\Rem_n(w_{x})) = \MonoidHom(\Rem_n(\varepsilon)) = 0$.

		      Thus, we see that $E$ is an $n$-decomposition for $\vec v$.
	\end{enumerate}

	\medskip

	\noindent
	\underline{Case 3}: The word $w_{x+1}$ is $n$-heavy with $(x+1,\leftSide) \in E$ and $(x+1,\rightSide) \notin E$.
	\nopagebreak

	\smallskip
	\noindent
	Notice that either $v_x$ is a prefix of $\Affix(w_{x+1},\leftSide)$, or $\Affix(w_{x+1},\leftSide)$ is a prefix of $v_x$.
	We consider these two cases separately as follows.

	\medskip

	\noindent
	\underline{Case 3.1}: The word $v_x$ is a prefix of $\Affix(w_{x+1},\leftSide)$.
	\nopagebreak

	\smallskip
	\noindent
	From \cref{lem:affix-bound-from-decomposition}, we see that $v_x$ and $v_{x+1}$ are words of the form
	\[
		v_x = a^p
		\qquad
		\text{and}
		\qquad
		v_{x+1} = a^q u
	\]
	with
	\[
		p+q = |\Affix(w_{x+1},\leftSide)|\geq 5\Lambda_n
		\qquad
		\text{and}
		\qquad
		u\in \Sigma^*.
	\]
	Notice that if $p\geq \Lambda_n$, then from \cref{lem:affix-bound-from-decomposition,lem:affix-bound-to-decomposition}, we see that
	\[
		E' = (E\setminus\{(x+1,\leftSide)\})\cup \{(x,\leftSide), (x,\rightSide)\}
	\]
	is an $(n+1)$-decomposition for $\vec v$.
	We consider the case where $p< \Lambda_n$ as follows.

	Since $p< \Lambda_n$, we see that $v_x$ is $n$-light, and that $v_{x+1}$ is $n$-heavy.
	From this, we see that $L_n(\vec v)= L_n(\vec w)$.
	Further, we then see that
	\[
		|\Affix_n(v_{x+1},\leftSide)| = |\Affix_n(w_{x+1}, \leftSide)| - p.
	\]
	Moreover, since $q\geq 2\Lambda_n$, we have $\Segment_n(v_{x+1},\leftSide) = \varepsilon$, and
	\[
		\MonoidHom(\Rem_n(v_{x}))+
		\MonoidHom(\Rem_n(v_{x+1}))
		=
		p + \MonoidHom(\Rem_n(w_{x+1})).
	\]
	We then see that
	\begin{align*}
		|\Affix(v_i,s)|
		 & \geq
		|\Affix(w_i,s)| - p
		\\&\geq
		\Lambda_{n-1}
		-
		(C - \MonoidHom(\vec w))
		\\&\phantom{\geq
			\Lambda_{n-1}
			-}
		-
		(2k - n - |L_n(\vec w)|)\sigma_n
		-
		\sum_{j=1}^{x}\MonoidHom( \Rem_n(w_j) )
		\\&\phantom{\geq
			\Lambda_{n-1}
			-}
		-(p+\MonoidHom( \Rem_n(w_{x+1}) ))
		-
		\sum_{j=x+2}^{k}\MonoidHom( \Rem_n(w_j) )
		\\&=
		\Lambda_{n-1}
		-
		(C - \MonoidHom(\vec v))
		-
		(2k - n - |L_n(\vec v)|)\sigma_n
		-
		\sum_{j=1}^{k}\MonoidHom( \Rem_n(v_j) )
	\end{align*}
	for each $(i,s)\in E$.
	Thus, we see that $E$ is an $n$-decomposition for $\vec v$.

	\medskip

	\noindent
	\underline{Case 3.2}: The word $\Affix(w_{x+1},\leftSide)$ is a prefix of $v_x$.
	\nopagebreak

	\smallskip
	\noindent
	In this case, the words $v_x$ and $v_{x+1}$ are of the form
	\[
		v_x
		=
		\Affix(w_x,\leftSide)\,u
		\qquad
		\text{and}
		\qquad
		v_{x+1}\in \Sigma^*
	\]
	where $u\in \Sigma^*$.
	From \cref{lem:affix-bound-from-decomposition}, $|\Affix(w_x,\leftSide)|\geq 5\Lambda_n$ and thus $v_x$ is $n$-heavy.

	We may apply \cref{lem:technical} to the factorisation $w_{x+1} = v_{x}v_{x+1}$ to obtain the following five cases.
	\begin{enumerate}[leftmargin=*]
		\item We have $|\Affix(v_x,\rightSide)| \geq \Lambda_n$.

		      Then, from \cref{lem:affix-bound-from-decomposition,lem:affix-bound-to-decomposition}, we see that
		      \[
			      E' = (E\setminus\{(x+1,\leftSide)\})\cup\{ (x,\leftSide), (x,\rightSide) \}
		      \]
		      is an $(n+1)$-decomposition for $\vec v$.

		\item We have $|\Affix(v_{x+1},\leftSide)| \geq \Lambda_n$.

		      Then, from \cref{lem:affix-bound-from-decomposition,lem:affix-bound-to-decomposition}, we see that
		      \[
			      E' = E\cup\{ (x,\leftSide) \}
		      \]
		      is an $(n+1)$-decomposition for $\vec v$.

		\item We have that $v_x$ and $v_{x+1}$ are both $n$-heavy, and \[\MonoidHom(\Rem_n(v_x)) + \MonoidHom(\Rem_n(v_{x+1})) \geq \MonoidHom(\Rem_n(w_{x+1}))-\sigma_n.\]

		      We then see that $L_n(\vec v) = L_n(\vec w)\setminus \{x\}$, in particular, $|L_n(\vec v)| = |L_n(\vec w)|-1$.
		      Let
		      \[
			      E' = (E\setminus\{(x+1,\leftSide)\})\cup\{(x,\leftSide)\}
		      \]
		      and notice that
		      \begin{align*}
			      |\Affix(v_i,s)|
			       & \geq
			      \Lambda_{n-1}
			      -
			      (C - \MonoidHom(\vec w))
			      -
			      (2k - n - |L_n(\vec w)|)\sigma_n
			      -
			      \sum_{j=1}^{k}\MonoidHom( \Rem_n(w_j) )
			      \\&=
			      \Lambda_{n-1}
			      -
			      (C - \MonoidHom(\vec w))
			      -
			      (2k - n - (|L_n(\vec w)|-1))\sigma_n
			      +\sigma_n
			      -
			      \sum_{j=1}^{k}\MonoidHom( \Rem_n(w_j) )
			      \\&\geq
			      \Lambda_{n-1}
			      -
			      (C - \MonoidHom(\vec v))
			      -
			      (2k - n - |L_n(\vec v)|)\sigma_n
			      -
			      \sum_{j=1}^{k}\MonoidHom( \Rem_n(v_j) )
		      \end{align*}
		      for each $(i,s)\in E'$.
		      Thus, we see that $E'$ is an $n$-decomposition for $\vec v$.

		\item We have that $v_x$ is $n$-light, $v_{x+1}$ is $n$-heavy, and \[\MonoidHom(\Rem_n(v_x)) + \MonoidHom(\Rem_n(v_{x+1})) = \MonoidHom(\Rem_n(w_{x+1})).\]
		      This case does not need to be considered as it contradicts our earlier case assumption that $v_x$ is $n$-heavy.

		\item We have that $v_x$ is $n$-heavy, $v_{x+1}$ is $n$-light, and \[\MonoidHom(\Rem_n(v_x)) + \MonoidHom(\Rem_n(v_{x+1})) = \MonoidHom(\Rem_n(w_{x+1})).\]

		      We have $L_n(\vec v) = (L_n(\vec w)\setminus\{x\})\cup\{x+1\}$, in particular, $|L_n(\vec v)| = |L_n(\vec w)|$.
		      Let
		      \[
			      E' = (E\setminus\{(x+1,\leftSide)\})\cup\{(x,\leftSide)\}
		      \]
		      and notice that
		      \begin{align*}
			      |\Affix(v_i,s)|
			       & \geq
			      \Lambda_{n-1}
			      -
			      (C - \MonoidHom(\vec w))
			      -
			      (2k - n - |L_n(\vec w)|)\sigma_n
			      -
			      \sum_{j=1}^{k}\MonoidHom( \Rem_n(w_j) )
			      \\&=
			      \Lambda_{n-1}
			      -
			      (C - \MonoidHom(\vec v))
			      -
			      (2k - n - |L_n(\vec v)|)\sigma_n
			      -
			      \sum_{j=1}^{k}\MonoidHom( \Rem_n(v_j) )
		      \end{align*}
		      for each $(i,s)\in E'$.
		      Thus, we see that $E'$ is an $n$-decomposition for $\vec v$.
	\end{enumerate}

	\medskip

	\noindent
	\underline{Case 4}: The word $w_{x+1}$ is $n$-heavy with $(x+1,\leftSide) \notin E$ and $(x+1,\rightSide) \in E$.
	\nopagebreak

	\smallskip
	\noindent
	The proof of this case is symmetric to the proof of Case~3 as above.

	\medskip

	\noindent
	\underline{Case 5}: The word $w_{x+1}$ is $n$-heavy with $(x+1,\leftSide), (x+1,\rightSide) \in E$.
	\nopagebreak

	\smallskip
	\noindent
	This case is separated into two subcases:
	\begin{itemize}
		\item[(5.1)] $w_{x+1}$ does not contain the letter $b$, in particular, \[w_{x+1} = \Affix(w_{x+1},\leftSide) = \Affix(w_{x+1},\rightSide)\text{; and}\]
		\item[(5.2)] $w_{x+1}$ contains at least one $b$, in particular, \[w_{x+1} = \Affix(w_{x+1},\leftSide)\, u\, \Affix(w_{x+1},\rightSide)\] where $u\in \Sigma^*$ with $|u|\geq 1$.
	\end{itemize}
	We consider these cases as follows.

	\medskip

	\noindent
	\underline{Case 5.1}: $w_{x+1} = \Affix(w_{x+1},\leftSide) = \Affix(w_{x+1},\rightSide)$.
	\nopagebreak

	\smallskip
	\noindent
	Here we see that
	\[
		v_{x} = a^p
		\qquad
		\text{and}
		\qquad
		v_{x+1} = a^q
	\]
	where $p+q = |w_{x+1}|$.
	We then separate this into three subcases based on the value of $p$ as follows.

	\medskip

	\noindent
	\underline{Case 5.1.1}: $\Lambda_n \leq p\leq |w_{x+1}|-\Lambda_n$ and thus $\Lambda_n \leq q\leq |w_{x+1}|-\Lambda_n$.
	\nopagebreak

	\smallskip
	\noindent
	From \cref{lem:affix-bound-from-decomposition,lem:affix-bound-to-decomposition}, we see that
	\[
		E' = E\cup \{(x,\leftSide)\}
		\qquad
		\text{and}
		\qquad
		E'' = E\cup \{(x,\rightSide)\}
	\]
	are both $(n+1)$-decompositions of $\vec v$.

	\medskip

	\noindent
	\underline{Case 5.1.2}: $p< \Lambda_n$ and thus $q> |w_{x+1}|-\Lambda_n$.
	\nopagebreak

	\smallskip
	\noindent
	We then see that $v_x$ is $n$-light.
	Moreover, from \cref{lem:affix-bound-from-decomposition}, we see that $|w_{x+1}|\geq 5\Lambda_n$ and thus $v_{x+1}$ is $n$-heavy.
	We then notice that $L_n(\vec v) = L_n(\vec w)$,
	\[
		\MonoidHom(\Rem_n(v_{x})) = p
		\qquad
		\text{and}
		\qquad
		\MonoidHom(\Rem_n(v_{x+1})) = \MonoidHom(\Rem_n(w_{x+1})) = 0
	\]
	For each $(i,s) \in E$, we then see that
	\begin{align*}
		|\Affix(v_i,s)| & \geq |\Affix(w_i,s)| - p
		\\
		                & \geq
		\Lambda_{n-1}
		-
		(C - \MonoidHom(\vec w))
		-
		(2k - n - |L_n(\vec w)|)\sigma_n
		-p-
		\sum_{j=1}^{k}\MonoidHom( \Rem_n(w_j) )
		\\&=
		\Lambda_{n-1}
		-
		(C - \MonoidHom(\vec v))
		-
		(2k - n - |L_n(\vec v)|)\sigma_n
		-
		\sum_{j=1}^{k}\MonoidHom( \Rem_n(v_j) ).
	\end{align*}
	Thus, $E$ is an $n$-decomposition for $\vec v$.

	\medskip

	\noindent
	\underline{Case 5.1.3}: $p> |w_{x+1}| - \Lambda_n$ and thus $q< \Lambda_n$.
	\nopagebreak

	\smallskip
	\noindent
	We see that $v_{x+1}$ is $n$-light.
	Moreover, from \cref{lem:affix-bound-from-decomposition} we see that $|w_{x+1}|\geq 5\Lambda_n$ and thus $v_{x}$ is $n$-heavy.
	We notice that $L_n(\vec v) = (L_n(\vec w)\setminus\{x\})\cup\{x+1\}$, and thus
	\[
		|L_n(\vec v)| = |L_n(\vec w)|.
	\]
	Moreover,
	\[
		\MonoidHom(\Rem_n(v_{x})) = \MonoidHom(\Rem_n(w_{x+1}))=0
		\qquad
		\text{and}
		\qquad
		\MonoidHom(\Rem_n(v_{x+1})) = q
	\]
	For each $(i,s) \in E\setminus\{(x+1,\leftSide),(x+1,\rightSide)\}$, we then see that
	\begin{align*}
		|\Affix(v_i,s)|
		 & \geq |\Affix(w_i,s)|      \\
		 & \geq |\Affix(w_i,s)|  - q
		\\
		 & \geq
		\Lambda_{n-1}
		-
		(C - \MonoidHom(\vec w))
		-
		(2k - n - |L_n(\vec w)|)\sigma_n
		-q-
		\sum_{j=1}^{k}\MonoidHom( \Rem_n(w_j) )
		\\&=
		\Lambda_{n-1}
		-
		(C - \MonoidHom(\vec v))
		-
		(2k - n - |L_n(\vec v)|)\sigma_n
		-
		\sum_{j=1}^{k}\MonoidHom( \Rem_n(v_j) ).
	\end{align*}
	Moreover, we see that for each $s\in \{\leftSide,\rightSide\}$,
	\begin{align*}
		|\Affix(v_{x},s)| & = |\Affix(w_{x+1},s)| - q
		\\
		                  & \geq
		\Lambda_{n-1}
		-
		(C - \MonoidHom(\vec w))
		-
		(2k - n - |L_n(\vec w)|)\sigma_n
		-q-
		\sum_{j=1}^{k}\MonoidHom( \Rem_n(w_j) )
		\\&r
		\Lambda_{n-1}
		-
		(C - \MonoidHom(\vec v))
		-
		(2k - n - |L_n(\vec v)|)\sigma_n
		-
		\sum_{j=1}^{k}\MonoidHom( \Rem_n(v_j) ).
	\end{align*}
	Thus,
	\[
		E' = (E\setminus\{(x+1,\leftSide),(x+1,\rightSide)\})\cup \{(x,\leftSide), (x,\rightSide)\}
	\]
	is an $n$-decomposition for $\vec v$.

	\medskip

	\noindent
	\underline{Case 5.2}: $w_{x+1} = \Affix(w_{x+1},\leftSide)\, u\, \Affix(w_{x+1},\rightSide)$ where $u\in \Sigma^*$ with $|u|\geq 1$.
	\nopagebreak

	\smallskip
	\noindent
	We separate this case into five subcases based on the length of $v_x$ as follows.

	\medskip

	\noindent
	\underline{Case 5.2.1}: $|v_x|< 2\Lambda_n$.
	\nopagebreak

	\smallskip
	\noindent
	We then see that $v_x$ is $n$-light, $v_{x+1}$ is $n$-heavy and thus $L_n(\vec v) = L_n(\vec w)$.

	From \cref{lem:affix-bound-from-decomposition}, we see that $v_x$ and $v_{x+1}$ are of the form
	\[
		v_x = a^p
		\qquad
		\text{and}
		\qquad
		v_{x+1} = a^q\, u\, \Affix(w_{x+1},\rightSide)
	\]
	where $p+q = |\Affix(w_{x+1},\leftSide)|\geq 5\Lambda_n$, $p<2\Lambda_n$ and $q \geq 3\Lambda_n$.
	We then see that
	\[
		\MonoidHom(\Rem_n(v_x))) = p
		\qquad
		\text{and}
		\qquad
		\MonoidHom(\Rem_n(v_{x+1})) = \MonoidHom(\Rem_n(w_{x+1})) = 0.
	\]
	For each $(i,s)\in E$, we then have
	\begin{align*}
		|\Affix(v_i,s)| & \geq |\Affix(w_i,s)| - p
		\\
		                & \geq
		\Lambda_{n-1}
		-
		(C - \MonoidHom(\vec w))
		-
		(2k - n - |L_n(\vec w)|)\sigma_n
		-p-
		\sum_{j=1}^{k}\MonoidHom( \Rem_n(w_j) )
		\\&=
		\Lambda_{n-1}
		-
		(C - \MonoidHom(\vec v))
		-
		(2k - n - |L_n(\vec v)|)\sigma_n
		-
		\sum_{j=1}^{k}\MonoidHom( \Rem_n(v_j) ).
	\end{align*}
	Notice that this follows since $\MonoidHom(\Rem_n(w_x)))=\MonoidHom(\Rem_n(w_{x+1}))=0$.

	Thus, we see that $E$ is an $n$-decomposition for $\vec v$.

	\medskip

	\noindent
	\underline{Case 5.2.2}: $2\Lambda_n\leq |v_x| \leq |\Affix(w_{x+1},\leftSide)|$.
	\nopagebreak

	\smallskip
	\noindent
	In this case we see that
	\[
		v_x = a^p
		\qquad
		\text{and}
		\qquad
		v_{x+1} = a^q\,u\, \Affix(w_{x+1},\rightSide)
	\]
	for some $u\in \Sigma^*$ where $p+q = |\Affix(w_{x+1},\leftSide)|$ and $p\geq 2\Lambda_n$.

	From \cref{lem:affix-bound-from-decomposition,lem:affix-bound-to-decomposition}, we see that
	\[
		E' = (E\setminus\{(x+1,\leftSide)\})\cup\{(x,\leftSide),(x,\rightSide)\}
	\]
	is an $(n+1)$-decomposition of $\vec v$.

	\medskip

	\noindent
	\underline{Case 5.2.3}: $|\Affix(w_{x+1},\leftSide)|\leq |v_x| \leq |w_{x+1}| - |\Affix(w_{x+1},\rightSide)|$.
	\nopagebreak

	\smallskip
	\noindent
	We notice then that
	\[
		v_x = \Affix(w_{x+1},\leftSide)\,u
		\qquad
		\text{and}
		\qquad
		v_{x+1}=u'\,\Affix(w_{x+1},\rightSide)
	\]
	for some $u,u'\in \Sigma^*$. Thus, both $v_x$ and $v_{x+1}$ are $n$-heavy.

	From \cref{lem:technical}, we are in one of the following five cases.
	\begin{enumerate}[leftmargin=*]
		\item We have $|\Affix(v_x,\rightSide)| \geq \Lambda_n$.

		      Then, from \cref{lem:affix-bound-from-decomposition,lem:affix-bound-to-decomposition}, we see that
		      \[
			      E' = (E\setminus\{(x+1,\leftSide)\})\cup\{ (x,\leftSide), (x,\rightSide) \}
		      \]
		      is an $(n+1)$-decomposition for $\vec v$.

		\item We have $|\Affix(v_{x+1},\leftSide)| \geq \Lambda_n$.

		      Then, from \cref{lem:affix-bound-from-decomposition,lem:affix-bound-to-decomposition}, we see that
		      \[
			      E' = E\cup\{ (x,\leftSide) \}
		      \]
		      is an $(n+1)$-decomposition for $\vec v$.

		\item We have that $v_x$ and $v_{x+1}$ are both $n$-heavy, and \[\MonoidHom(\Rem_n(v_x)) + \MonoidHom(\Rem_n(v_{x+1})) \geq \MonoidHom(\Rem_n(w_{x+1}))-\sigma_n.\]

		      We then see that $L_n(\vec v) = L_n(\vec w)\setminus \{x\}$, in particular, $|L_n(\vec v)| = |L_n(\vec w)|-1$.
		      Let
		      \[
			      E' = (E\setminus\{(x+1,\leftSide)\})\cup\{(x,\leftSide)\}
		      \]
		      and notice that
		      \begin{align*}
			      |\Affix(v_i,s)|
			       & \geq
			      \Lambda_{n-1}
			      -
			      (C - \MonoidHom(\vec w))
			      -
			      (2k - n - (|L_n(\vec w)|-1))\sigma_n
			      +\sigma_n
			      -
			      \sum_{j=1}^{k}\MonoidHom( \Rem_n(w_j) )
			      \\&\geq
			      \Lambda_{n-1}
			      -
			      (C - \MonoidHom(\vec v))
			      -
			      (2k - n - |L_n(\vec v)|)\sigma_n
			      -
			      \sum_{j=1}^{k}\MonoidHom( \Rem_n(v_j) )
		      \end{align*}
		      for each $(i,s)\in E'$.
		      Thus, we see that $E'$ is an $n$-decomposition for $\vec v$.

		\item We have that $v_x$ is $n$-light, $v_{x+1}$ is $n$-heavy, and \[\MonoidHom(\Rem_n(v_x)) + \MonoidHom(\Rem_n(v_{x+1})) = \MonoidHom(\Rem_n(w_{x+1})).\]
		      This case does not need to be considered as it contradicts our case assumption $v_x$ is $n$-heavy.

		\item We have that $v_x$ is $n$-heavy, $v_{x+1}$ is $n$-light, and \[\MonoidHom(\Rem_n(v_x)) + \MonoidHom(\Rem_n(v_{x+1})) = \MonoidHom(\Rem_n(w_{x+1})).\]
		      This case does not need to be considered as it contradicts our case assumption $v_{x+1}$ is $n$-heavy.
	\end{enumerate}

	\medskip

	\noindent
	\underline{Case 5.2.4}: $|w_{x+1}| - |\Affix(w_{x+1},\rightSide)|\leq |v_x| \leq |w_{x+1}| - 2\Lambda_n$.
	\nopagebreak

	\smallskip
	\noindent
	In this case we see that
	\[
		v_x = \Affix(w_{x+1},\leftSide)\,u\,a^p
		\qquad
		\text{and}
		\qquad
		v_{x+1} = a^q
	\]
	for some $u\in \Sigma^*$ where $p+q = |\Affix(w_{x+1},\leftSide)|$ and $q\geq 2\Lambda_n$.

	From \cref{lem:affix-bound-from-decomposition,lem:affix-bound-to-decomposition}, we see that
	\[
		E' = E\cup\{(x,\leftSide)\}
	\]
	is an $(n+1)$-decomposition of $\vec v$.

	\medskip

	\noindent
	\underline{Case 5.2.5}: $|v_x| > |w_{x+1}| - 2\Lambda_n$.
	\nopagebreak

	\smallskip
	\noindent
	We then see that $v_x$ is $n$-heavy, $v_{x+1}$ is $n$-light and thus \[L_n(\vec v) = (L_n(\vec w)\setminus\{x\})\cup \{x+1\},\]
	in particular, this means that $|L_n(\vec v)| = |L_n(\vec w)|$.

	From \cref{lem:affix-bound-from-decomposition}, we see that $v_x$ and $v_{x+1}$ are of the form
	\[
		v_x = \Affix(w_{x+1},\leftSide) \, u\, a^p
		\qquad
		\text{and}
		\qquad
		v_{x+1} = a^q
	\]
	where $p+q = |\Affix(w_{x+1},\leftSide)|\geq 5\Lambda_n$, $p> 3\Lambda_n$ and $q < 2\Lambda_n$.
	We then see that
	\[
		\MonoidHom(\Rem_n(v_x))) = \MonoidHom(\Rem_n(w_{x+1})) = 0
		\qquad
		\text{and}
		\qquad
		\MonoidHom(\Rem_n(v_{x+1})) = q
	\]
	For each $(i,s)\in E\setminus\{(x+1,\leftSide),(x+1,\rightSide)\}$, we then have
	\begin{align*}
		|\Affix(v_i,s)| & = |\Affix(w_i,s)| \\ &\geq |\Affix(w_i,s)| - q
		\\
		                & \geq
		\Lambda_{n-1}
		-
		(C - \MonoidHom(\vec w))
		-
		(2k - n - |L_n(\vec w)|)\sigma_n
		-q-
		\sum_{j=1}^{k}\MonoidHom( \Rem_n(w_j) )
		\\&=
		\Lambda_{n-1}
		-
		(C - \MonoidHom(\vec v))
		-
		(2k - n - |L_n(\vec v)|)\sigma_n
		-
		\sum_{j=1}^{k}\MonoidHom( \Rem_n(v_j) ).
	\end{align*}
	Moreover, for each $s\in \{\leftSide,\rightSide\}$, we see that
	\begin{align*}
		|\Affix(v_{x},s)| & \geq |\Affix(w_{x+1},s)| - q
		\\
		                  & \geq
		\Lambda_{n-1}
		-
		(C - \MonoidHom(\vec w))
		-
		(2k - n - |L_n(\vec w)|)\sigma_n
		-q-
		\sum_{j=1}^{k}\MonoidHom( \Rem_n(w_j) )
		\\&=
		\Lambda_{n-1}
		-
		(C - \MonoidHom(\vec v))
		-
		(2k - n - |L_n(\vec v)|)\sigma_n
		-
		\sum_{j=1}^{k}\MonoidHom( \Rem_n(v_j) ).
	\end{align*}
	Thus,
	\[
		E' = (E\setminus\{(x+1,\leftSide),(x+1,\rightSide)\})\cup\{(x,\leftSide),(x,\rightSide)\}
	\]
	is an $n$-decomposition for $\vec v$.

	\medskip

	\noindent
	\underline{Conclusion}:
	\nopagebreak

	\smallskip
	\noindent
	In all cases, if $\vec w$ has a maximal $n$-decomposition, then we can either construct an $n$-decomposition or construct an $(n+1)$-decomposition for $\vec v$.
\end{proof}

\subsection{Mirroring decompositions}\label{sec:main-proof/mirror}
We introduce $\Reverse\colon \Sigma^*\to\Sigma^*$ such that
$
	\Reverse(w) = w_n \cdots w_2 w_1
$
for each $w = w_1 w_2\cdots w_n \in \Sigma^*$.
We then extend this function to operate on tuples of words as follows.
We define $\Reverse\colon (\Sigma^*)^k\to (\Sigma^*)^k$ such that for each $\vec w = (w_1,w_2,\ldots,w_k)\in (\Sigma^*)^k$,
\[
	\Reverse(\vec w)
	=
	(\Reverse(w_k),\ldots, \Reverse(w_2),\Reverse(w_1))
\]
We define $\Reverse \colon \mathcal P (\{1,2,\ldots,k\}\times \{\leftSide,\rightSide\})\to \mathcal P (\{1,2,\ldots,k\}\times \{\leftSide,\rightSide\})$ on decompositions as follows.
Suppose that $E\subseteq \{1,2,\ldots,k\}\times\{\leftSide,\rightSide\}$ is an $n$-decomposition of $\vec w$.
We then define $\Reverse(E)$ such that
\begin{itemize}
	\item $(i,\leftSide) \in \Reverse(E)$ if and only if $(k+1-i,\rightSide)\in E$; and
	\item $(i,\rightSide) \in \Reverse(E)$ if and only if $(k+1-i,\leftSide)\in E$.
\end{itemize}
We then see that $\Reverse(E)$ is an $n$-decomposition for the word vector $\Reverse(\vec w)$.

Notice from the definition of the word $\mathcal W$ in \cref{def:counterexample} that $\Reverse(\mathcal W)=\mathcal W$.
Thus, from the definition of $\mathcal W$-sentential tuple in \cref{def:sentential-tuple}, we see that if $\vec w$ is a $\mathcal W$-sentential tuple, then $\Reverse(\vec w)$ is also a $\mathcal W$-sentential tuple.
Finally, notice that in each of the above usages of $\Reverse$, the function $\Reverse$ is an involution.
Using these maps, we may now derive the following proposition which is a generalisation of \cref{prop:decomp1,prop:decomp2,prop:decomp3}.

\begin{proposition}\label{prop:decomp-full}
	Let $\vec w = (w_1,w_2,\ldots,w_k)$ and $\vec v = (v_1,v_2,\ldots,v_k)$ be $\mathcal W$-sentential tuples as in \cref{def:sentential-tuple}, and suppose that $\vec w$ has one of the four forms
	\begin{align}
		\vec w=(w_1,w_2,\ldots,w_k) & = (v_1,v_2,\ldots,v_{x-1},a v_{x}, v_{x+1},\ldots,v_k),\label{eq:decomp-full/1}
		\\\vec w=(w_1,w_2,\ldots,w_k) &= (v_1,v_2,\ldots,v_{x-1},v_{x}a, v_{x+1},\ldots,v_k),\label{eq:decomp-full/2}
		\\\vec w=(w_1,w_2,\ldots,w_k) &= (v_1,v_2,\ldots,v_{x-1},b v_{x}, v_{x+1},\ldots,v_k),\label{eq:decomp-full/3}
		\\\vec w=(w_1,w_2,\ldots,w_k) &= (v_1,v_2,\ldots,v_{x-1},v_{x}b, v_{x+1},\ldots,v_k)\label{eq:decomp-full/4}
	\end{align}
	for some $x\in \{1,2,\ldots, k\}$, that is, $\vec v$ is obtained from $\vec w$ by deleting an $a$ or $b$ from the left or right-hand side of some component of $\vec w$; or $\vec w$ has one of the following two forms
	\begin{align}
		\vec w = (w_1,w_2, \ldots,w_k) & = (v_1,\ldots,v_{x-1}, \varepsilon, v_x v_{x+1}, v_{x+2},\ldots,v_k),\label{eq:decomp-full/5}
		\\\vec w = (w_1,w_2,\ldots,w_k) &= (v_1,\ldots, v_{x-1},v_x v_{x+1},\varepsilon,  v_{x+2},\ldots,v_k)\label{eq:decomp-full/6}
	\end{align}
	for some $x\in \{1,2,\ldots,k-1\}$, that is, $\vec v$ is a vector for which $(w_x,w_{x+1})\in \{ (v_xv_{x+1},\varepsilon), (\varepsilon, v_x v_{x+1}) \}$ and $w_{i} = v_i$ for each $i\notin \{x,x+1\}$.
	If the tuple $\vec w$ has an $n$-decomposition, then $\vec v$ has an $n'$-decomposition for some $n'\geq n$.
	(Compare the above cases with the grammar rules in \cref{def:rule-accept,def:rule-insertion,def:rule-merge}.)
\end{proposition}
\begin{proof}
	We see that cases (\ref{eq:decomp-full/1}), (\ref{eq:decomp-full/3}) and (\ref{eq:decomp-full/5}) follow from \cref{prop:decomp1,prop:decomp2,prop:decomp3}, respectively.
	We consider the remaining cases as follows.

	\begin{itemize}
		\item
		      Suppose our vectors $\vec w$ and $\vec v$ are in (\ref{eq:decomp-full/2}).
		      Thus, the pair of vectors $\Reverse(\vec w)$ and $\Reverse(\vec v)$ are related as in (\ref{eq:decomp-full/1}), and $\Reverse(\vec w)$ has an $n$-decomposition.
		      From \cref{prop:decomp1}, we obtain an $n'$-decomposition $E$, with $n'\geq n$, for $\Reverse(\vec v)$.
		      Thus, $\Reverse(E)$ is an $n'$-decomposition for $\vec v$.

		\item
		      Suppose our vectors $\vec w$ and $\vec v$ are in (\ref{eq:decomp-full/4}).
		      Thus, the pair of vectors $\Reverse(\vec w)$ and $\Reverse(\vec v)$ are related as in (\ref{eq:decomp-full/3}), and $\Reverse(\vec w)$ has an $n$-decomposition.
		      From \cref{prop:decomp2}, we obtain an $n'$-decomposition $E$, with $n'\geq n$, for $\Reverse(\vec v)$.
		      Thus, $\Reverse(E)$ is an $n'$-decomposition for $\vec v$.

		\item
		      Suppose our vectors $\vec w$ and $\vec v$ are in (\ref{eq:decomp-full/6}).
		      Thus, the pair of vectors $\Reverse(\vec w)$ and $\Reverse(\vec v)$ are related as in (\ref{eq:decomp-full/5}), and $\Reverse(\vec w)$ has an $n$-decomposition.
		      From \cref{prop:decomp3}, we obtain an $n'$-decomposition $E$, with $n'\geq n$, for $\Reverse(\vec v)$.
		      Thus, $\Reverse(E)$ is an $n'$-decomposition for $\vec v$.
	\end{itemize}
	Hence, we have proven each of the cases of the proposition.
\end{proof}

\subsection{Main result}\label{sec:main-proof/result}
We are now ready to prove our main result as follows.

\MainTheorem*

\begin{proof}
	Earlier in this section, we assumed for contradiction that there exists some \RMCFG\ $M = (\Sigma,Q,S,P)$ which is in normal form (as in \cref{lem:normal-form}) that generates the language $L$ as introduced in the beginning of \cref{sec:main-theorem}.
	Moreover, we let $C$ be the constant for the grammar $M$ as in \cref{lem:derivation_bound}.

	From \cref{lem:m-hom-recursive-definition}, we see that $\mathcal W\in L$ and so there must be a derivation for $\mathcal W$ in the grammar $M$ of the form
	\begin{multline*}
		S(\mathcal{W})\leftarrow
		H_1(\mathcal{W},\varepsilon, \varepsilon,\ldots,\varepsilon)
		\leftarrow
		H_2(w_{1,1}, w_{1,2},\ldots,w_{1,k})
		\\\leftarrow
		H_3(w_{2,1}, w_{2,2},\ldots,w_{2,k})
		\leftarrow
		\cdots
		\leftarrow
		H_{t+1}(w_{t,1}, w_{t,2},\ldots,w_{t,k})
		\\\leftarrow
		H_{t+2}(\varepsilon,\varepsilon,\ldots,\varepsilon) \leftarrow
	\end{multline*}
	for some $t\in \mathbb N$ where each $w_{i,j}\in \Sigma^*$, and each $H_i \in Q$.

	We may then define a sequence of tuples $\vec w_1, \vec w_2,\ldots,\vec w_{t+2}$ such that
	\[
		\vec w_1 = (\mathcal{W},\varepsilon,\varepsilon,\ldots,\varepsilon),\quad
		\vec w_{t+2} = (\varepsilon,\varepsilon,\ldots,\varepsilon)
		\quad\text{and}\quad
		\vec w_{i} = (w_{i-1,1}, w_{i-1,2},\ldots,w_{i-1,k})\\
	\]
	for each $i\in \{2,3,\ldots,t+1\}$.
	From the definition of the constant $C$ in \cref{lem:derivation_bound}, and the definition of an \RMCFG\ multiple context-free grammar, we see that each vector $\vec w_j$, as above, is a $\mathcal W$-sentential tuple as defined in \cref{def:sentential-tuple}.

	Notice here that the six cases of \cref{prop:decomp-full} correspond to the reverse of the insertion and merge replacement rules as in \cref{def:rule-insertion,def:rule-merge}.
	In particular, given an application of such a rule as
	\[
		H(w_1, w_2, \ldots,w_k) \leftarrow K(v_1,v_2,\ldots,v_k),
	\]
	the vectors $\vec w = (w_1,w_2,\ldots,w_k)$ and $\vec v = (v_1,v_2,\ldots,v_k)$ satisfy one of the relations in \cref{prop:decomp-full}.
	In particular, (\ref{eq:decomp-full/1}--\ref{eq:decomp-full/4}) correspond to left and right insertion rules (as in \cref{def:rule-insertion}); and (\ref{eq:decomp-full/5}) and (\ref{eq:decomp-full/6}) corresponds to left and right merge rules (as in \cref{def:rule-merge}).

	Hence, from \cref{prop:decomp-full}, for each $i\in \{1,2,\ldots,t+1\}$, we see that if $\vec w_i$ has an $n$-decomposition, then $\vec w_{i+1}$ has an $n'$-decomposition for some $n'\geq n$.
	Thus, by an induction on $i$, we find that if $\vec w_1$ has a decomposition, then so must $\vec w_{t+2}$.

	From \cref{cor:affix-of-w,lem:bound1,lem:affix-bound-to-decomposition}, we see that
	$
		E = \{(1,\leftSide),(1,\rightSide)\}
	$
	is a $2$-decomposition of the $\mathcal W$-sentential tuple
	$
		\vec w_1 = (\mathcal{W},\varepsilon,\varepsilon,\ldots,\varepsilon)
        $
	(as demonstrated in  \cref{remark:Compute2decompExplicitly}).
	Thus, the sentential tuple
$
		\vec w_{t+2} = (\varepsilon,\varepsilon,\ldots,\varepsilon)
	$
	must have a decomposition.
	However, it is not possible for $\vec w_{t+2}$ to have a decomposition since none of its components are $n$-heavy for any $n \in \{2,3,\ldots,2k\}$.
	In particular, for a component to be $n$-heavy, it must contain a factor $a^{2\Lambda_n}$ where each $\Lambda_n>0$ from \cref{lem:bound1}.
	We thus conclude that no such grammar $M$ for the language $L$ can exist, and thus $L$ cannot be an EDT0L language.
\end{proof}

\subsection{Corollary}
We can immediately generalise \cref{thm:main} to obtain \cref{thm:cormain}.

\MainCorollary*

\begin{proof}
	Suppose for contradiction that the group $G$ has an EDT0L word problem and an element $g\in G$ of infinite order, so $\left\langle g\right\rangle \cong \mathbb{Z}$.
  From \cref{lem:taking-a-submonoid-of-EDT0L}, it follows that the subgroup $\left\langle g\right\rangle\cong \mathbb Z$ has an EDT0L word problem with respect to the generating set $\{g,g^{-1}\}$.
  Thus, $\mathbb Z$ has an EDT0L word problem which contradicts \cref{thm:main}.
\end{proof}

\section*{Acknowledgements}
The first author acknowledges support from Swiss NSF grant 200020-200400.
The second author was supported by Australian Research Council grant DP210100271, and the London Mathematical Society 
Visiting Speakers to the UK--Scheme 2, 2024.
The third author was supported by the Heilbronn Institute for Mathematical Research.
The fifth author was supported by the Heilbronn Institute for Mathematical Research and the EPSRC Fellowship grant EP/V032003/1 `Algorithmic, topological and geometric aspects of infinite groups, monoids and inverse semigroups'. 

\bibliographystyle{plain}
\bibliography{main}

\makeatletter
\enddoc@text
\let\enddoc@text\relax
\makeatother

\clearpage
\appendix
\section{An alternative proof of \texorpdfstring{\cref{lem:taking-a-submonoid-of-EDT0L}}{Proposition~\protect\ref*{lem:taking-a-submonoid-of-EDT0L}}}

We note here that this paper is almost completely self-contained.
In particular, the only result which we do not prove is in \cref{sec:edt0l-wp} where we cite that the family of finite-index EDT0L languages is closed under inverse monoid homomorphism.
This result is  used in \cref{lem:taking-a-submonoid-of-EDT0L} to show that having an EDT0L word or co-word problem is invariant under change of generating set.
Our reason for structuring the paper in this way is to further motivate the study of finite-index EDT0L languages.

Here we give a short alternative proof of \cref{lem:taking-a-submonoid-of-EDT0L} which only uses results proven in this paper.
This alternative proof is a corollary to the following lemma which is itself a specialisation of a result of Gilman (see Theorem~4.26 and~4.29 in \cite{GilmanBook}).

\begin{lemma}\label{lem:taking-a-submonoid}
	Let $\mathcal F$ be a family of formal languages which contains the regular languages and is closed under mapping by string transducer;
  let $G$ be a group with finite monoid generating set $X$, and let $Y$ be a finite set that generates a submonoid of $G$.
	If $\mathrm{WP}(G,X)$ (resp.\@ $\mathrm{coWP}(G,X)$) lies in the family of languages $\mathcal F$, then $\mathrm{WP}(G,Y)$ (resp.\@ $\mathrm{coWP}(G,Y)$) also lies in $\mathcal F$.
\end{lemma}

\begin{proof}
	Notice that if $G$ is finite, then this result follows from the fact that the word problem of a finite group, and any submonoid thereof, is a regular language (see~\cite[Theorem 1]{Anisimov1971}).
	Thus, in the remainder of this proof, we may assume without loss of generality that $G$ is an infinite group.

	In the remainder of this proof, suppose that $Y = \{y_1,y_2,\ldots,y_n\}$.
	Notice that, for each $y_i\in Y$, we can choose a word $u_i\in X^*$ for which $\overline{y_i} =_G \overline{u_i}$.
	From \cref{lem:example-antichain}, there exists a choice of words $w_1,w_2,\ldots,w_n,w_{n+1}\in X^*$ such that
	$
		W
		=
		\{
		w_1 u_1, w_2 u_2,\ldots,w_n u_n
		\}
	$
	is an antichain with respect to the prefix order where each $\overline{w_i}=1$.

	We define a map $f\colon \mathcal{P}(X^*)\to\mathcal{P}(Y^*)$ as
	\[
		f(L)
		=
		\left\{
		y_{i_1} y_{i_2} \cdots y_{i_k} \in Y^*
		\ \middle|\
		\begin{aligned}
			(w_{i_1} u_{i_1})(w_{i_2} u_{i_2})\cdots(w_{i_k} u_{i_k}) \in L
			\\\text{ where each }
			i_j \in \{1,2,\ldots,n\}
		\end{aligned}
		\right\}.
	\]
	From \cref{lem:f_is_detfsm}, we see that $f$ is a string transducer.

	We that $f(\mathrm{WP}(G,X)) = \mathrm{WP}(G,Y)$ and $f(\mathrm{coWP}(G,X)) = \mathrm{coWP}(G,Y)$.
	Since the family $\mathcal F$ is closed under string transduction, we conclude that $\mathrm{WP}(G,Y)$ and $\mathrm{coWP}(G,Y)$ both belong to $\mathcal F$.
\end{proof}

We then obtain an alternative proof of \cref{lem:taking-a-submonoid-of-EDT0L} as follows.

\CorInvariantsDFSM*

\begin{proof}
This follows by combining \cref{lem:taking-a-submonoid} with \cref{lem:reg_is_edt0l,lem:edt0l-closed-under-fs-transduction}, and \cref{prop:wp-edt0l-fi}.
\end{proof}

\end{document}